\documentclass[a4paper,11pt]{amsart}
\usepackage{amsfonts}
 \usepackage{enumerate}
\usepackage{amsmath, amsthm}
\usepackage{amscd}
\usepackage{amssymb}
\input xy
\xyoption{all}
\usepackage{latexsym,graphicx}
\usepackage{varioref}

\newcommand{\sspace}{\vspace{0.25cm}}
\newcommand{\noi}{\noindent}

\theoremstyle{plain}
\newtheorem{theor}{Theorem}[subsection]
\newtheorem{conj}[theor]{Conjecture}
\newtheorem{prop}[theor]{Proposition}
\newtheorem{lem}[theor]{Lemma}
\newtheorem{sublem}[theor]{Sublemma}
\newtheorem{cor}[theor]{Corollary}

\newtheorem{theorA}{Theorem}[section]

\newtheorem{propA}[theorA]{Proposition}
\newtheorem{lemA}[theorA]{Lemma}

\theoremstyle{remark}
\newtheorem{rem}[theor]{Remark}
\newtheorem{rems}[theor]{Remarks}

\theoremstyle{plain}
\newtheorem{defi}[theor]{Definition}

\numberwithin{equation}{section}

\newcommand{\CC}{{\mathbb C}}
\newcommand{\RR}{{\mathbb R}}
\renewcommand{\SS}{{\mathbf S}}
\newcommand{\QQ}{{\mathbb Q}}
\newcommand{\FF}{{\mathbb F}}
\newcommand{\ZZ}{{\mathbb Z}}
\renewcommand{\AA}{{\mathbf A}}
\newcommand{\G}{{\mathbf G}}
\newcommand{\HH}{{\mathbf H}}

\newcommand{\TT}{{\mathbf T}}
\newcommand{\PP}{{\mathbf P}}
\newcommand{\MM}{{\mathbf M}}
\newcommand{\NNN}{{\mathbf N}}
\newcommand{\Z}{{\mathbf Z}}
\newcommand{\C}{{\mathbf C}}
\newcommand{\NN}{{\mathbb N}}
\newcommand{\Zhat}{\hat{\ZZ}}
\newcommand{\AAf}{\AA_{\rm f}}
\newcommand{\Qbar}{{\overline{\QQ}}}

\newcommand{\Spec}{{\rm Spec}}

\newcommand{\Ga}{\Gamma}

\newcommand{\into}{\hookrightarrow}

\newcommand{\ol}{\overline}

\newcommand{\wt}{\widetilde}
\newcommand{\lto}{\longrightarrow}

\newcommand{\sm}{{\rm sm}}

\newcommand{\Hom}{{\rm Hom}}
\newcommand{\Res}{{\rm Res}}
\newcommand{\Sh}{{\rm Sh}}
\newcommand{\Gal}{{\rm Gal}}
\newcommand{\ad}{{\rm ad}}

\newcommand{\X}{{\mathcal X}}
\newcommand{\A}{{\mathcal A}}
\newcommand{\He}{{\mathcal H}}

\newcommand{\Cone}{{\mathcal C}}
\newcommand{\der}{{\rm der}}

\newcommand{\GL}{{\rm \bf GL}}

\newcommand{\MT}{{\rm \bf MT}}

\newcommand{\ordre}{{\textnormal{ord}}}

\newcommand{\alg}{\textnormal{alg}}

\newcommand{\Iw}{\mathcal{I}}
\newcommand{\OO}{\mathcal{O}}

\newcommand{\K}{\mathcal{K}}
\newcommand{\V}{\mathcal{V}}

\begin{document}

\title{The Andr{\'e}-Oort conjecture.
\footnote{Submitted to Annals of mathematics. Version of September 2013}}
\author{B. Klingler, A. Yafaev}
\thanks{B.K. was supported by NSF grant DMS 0350730}
\date{}


\begin{abstract}
In this paper we prove, assuming the Generalized Riemann Hypothesis, the Andr{\'e}-Oort conjecture
on the Zariski closure of sets of special points in a Shimura variety.
In the case of sets of special points satisfying an additional assumption, we prove the conjecture
without assuming the GRH.
\end{abstract}

\maketitle

\tableofcontents

\section{Introduction.}\label{section1}

\subsection{The Andr{\'e}-Oort conjecture.}
The purpose of  this paper is to prove, under certain assumptions, the
Andr{\'e}-Oort conjecture on special subvarieties of Shimura varieties.

Before stating the Andr{\'e}-Oort conjecture we provide some motivation  
from algebraic geometry. 
Let $Z$ be a smooth complex algebraic variety and let $\mathcal{F} \lto Z$ be a
variation of polarizable $\QQ$-Hodge structures on $Z$
(for example $\mathcal{F} = R^if_* \QQ$ for a smooth proper
morphism $f:Y \lto Z$). To every $z\in Z$ one associates a reductive algebraic
$\QQ$-group
$\MT(z)$, called the Mumford-Tate group of the Hodge structure
$\mathcal{F}_z$. This group is the stabiliser of the Hodge classes in  
the rational Hodge
structures tensorially generated by $\mathcal{F}_z$ and its
dual. A point $z \in Z$ is said to be Hodge generic if $\MT(z)$ is
maximal. If $Z$ is irreducible, two Hodge generic points of $Z$ have
the same Mumford-Tate group, called the generic Mumford-Tate group
$\MT_Z$.
The complement of the Hodge generic locus is a countable union of closed irreducible algebraic
subvarieties of $Z$, each not contained in the union of the
others. This is proved in \cite{CDK}.
Furthermore, it is shown in \cite{Vo} that when $Z$ is defined over $\ol\QQ$
(and under certain simple assumptions) these components are also defined
over $\ol\QQ$.
The irreducible components of the intersections  
 of these subvarieties are called
{\em special subvarieties} (or subvarieties of Hodge type) of $Z$ relative to
$\mathcal{F}$. Special subvarieties of dimension zero are called {\em special
points}. 

\sspace
\noi
{\bf Example:}
Let $Z$ be the  modular curve $Y(N)$ (with $N\geq 4$) and  
let $\mathcal{F}$ be the variation of
polarizable $\QQ$-Hodge structures $R^1f_*\QQ$ of weight one on $Z$  
associated to the
universal elliptic curve $f : E \lto Z$. Special points on $Z$
parametrize elliptic curves with complex multiplication. The
generic Mumford-Tate group on $Z$ is $\GL_{2,\QQ}$. The Mumford-Tate  
group of a
special point corresponding to an elliptic curve with complex
multiplication by a quadratic imaginary field $K$ is the torus
${\textnormal{Res}}_{K /\QQ} \mathbf{G}_{\mathbf{m},K}$ obtained by restriction of scalars from
$K$ to $\QQ$ of the multiplicative group $\mathbf{G}_{\mathbf{m},K}$
over $K$.

\sspace
The general Noether-Lefschetz problem consists in describing  
the geometry of these
special subvarieties, in particular the distribution of special
points. 
Griffiths transversality condition prevents, in general, the existence
of moduli spaces for variations of polarizable $\QQ$-Hodge structures.
Shimura varieties naturally appear as solutions to such moduli
problems with additional data (c.f. \cite{De1}, \cite{De2},
\cite{Milne}). Recall that a $\QQ$-Hodge structure on a $\QQ$-vector space
$V$ is a structure of
$\SS$-module on $V_{\RR}:= V \otimes_{\QQ} \RR$, where
$\SS= {\textnormal{Res}}_{\CC/\RR} \mathbf{G}_{\mathbf{m},\CC}$. In other words it is a  
morphism of real algebraic groups
$
h \colon \SS \lto \GL(V_{\RR})
$.

The Mumford-Tate group $\MT(h)$ is the
smallest algebraic $\QQ$-subgroup $\HH$ of $\GL(V)$ such that
$h$ factors through $\HH_{\RR}$. A {\em Shimura datum} is a
pair $(\G, X)$, with $\G$ a linear connected reductive group over $\QQ$
and $X$ a $\G(\RR)$-conjugacy class in the set of morphisms of
real algebraic groups $\Hom(\SS, \G_{\RR})$, satisfying the
  ``Deligne's conditions'' \cite[1.1.13]{De2}. These conditions imply,  
in particular, that
the connected components of $X$ are Hermitian symmetric domains and that
  $\QQ$-representations of $\G$ induce polarizable variations of
$\QQ$-Hodge structures on $X$. A morphism of Shimura data from $(\G_1, X_1)$
to $(\G_2, X_2)$ is a $\QQ$-morphism $f: \G_1 \lto \G_2$ that maps
$X_1$ to $X_2$.

Given  a compact open subgroup $K$ of $\G(\AAf)$ (where $\AAf$ denotes
the ring of finite ad{\`e}les of $\QQ$) the set $\G(\QQ)\backslash (X
\times \G(\AAf)/K)$ is naturally the set of $\CC$-points 
of a quasi-projective variety (a {\em Shimura variety}) over $\CC$, denoted
$\Sh_{K}(\G,X)_{\CC}$. The projective limit 
$\Sh(\G, X)_{\CC}=\varprojlim_{\substack{K}} \Sh_{K}(\G,X)_{\CC}$ is a
$\CC$-scheme on which $\G(\AAf)$ acts continuously by multiplication on
the right (c.f. section~\ref{neutral}).  The multiplication
by $g \in \G(\AAf)$ on $\Sh(\G, X)_{\CC}$ induces an algebraic
correspondence $T_g$ on $\Sh_{K}(\G,X)_{\CC}$, called a {\em Hecke
correspondence}. One shows that a subvariety $V \subset \Sh_{K}(\G,
X)_{\CC}$ is {\em special} (with respect to some variation of Hodge structure
associated to a faithful $\QQ$-representation of $\G$) if and only if there is  
a Shimura datum
$(\HH, X_{\HH})$, a morphism of Shimura data $f: (\HH, X_{\HH}) \lto (\G, X)$
and an element $g \in \G(\AAf)$ such that $V$ is an irreducible
component of the image of the morphism:

$$ \Sh(\HH, X_{\HH})_{\CC} \stackrel{\Sh(f)}{\lto} \Sh(\G, X)_{\CC}
\stackrel{.g}{\lto} \Sh(\G, X)_{\CC} \lto \Sh_{K}(\G, X)_{\CC}\;\;. $$

It can also be shown that the Shimura datum $(\HH,X_{\HH})$
can be chosen in such a way that $\HH \subset \G$ is the generic Mumford-Tate 
group on $X_{\HH}$ (see Lemma 2.1 of \cite{UllmoYafaev}).
A {\em special} point is a special subvariety of dimension zero.
One sees that a point $\ol{(x, g)} \in \Sh_{K}(\G,X)_\CC(\CC)$ (where
$x \in X$ and $g \in \G(\AAf)$) is {\em  
special} if and only if
the group $\MT(x)$ is commutative (in which case $\MT(x)$ is a torus).

\sspace
Given a special subvariety $V$ of $\Sh_K(\G,X)_{\CC}$, the set of  
special points
of $\Sh_K(\G,X)_\CC(\CC)$ contained in $V$ is dense in $V$ for the strong  
(and in particular for the Zariski) topology.
Indeed, one shows that $V$ contains a special
point, say $s$. Let $\HH$ be a reductive group defining $V$ and let
$\HH(\RR)^+$ denote the connected component of the identity in the
real Lie group $\HH(\RR)$.
The fact that $\HH(\QQ)\cap \HH(\RR)^+$ is dense in $\HH(\RR)^+$  
implies that the
``$\HH(\QQ)\cap \HH(\RR)^+$-orbit'' of $s$, which is contained in $V$,  
is dense
in $V$. This ``orbit'' (sometimes referred to as the Hecke orbit of  
$s$) consists of special points.
The Andr{\'e}-Oort conjecture is the converse statement.

\begin{defi} \label{defi111}
Given a set $\Sigma$ of subvarieties of $\Sh_{K}(\G,X)_{\CC}$ we
denote by $\mathbf{\Sigma}$ the subset $\cup_{V \in \Sigma} V$ of
$\Sh_{K}(\G,X)_{\CC}$.
\end{defi}

\begin{conj}[Andr{\'e}-Oort] \label{conjecture}
Let $(\G,X)$ be a Shimura datum, $K$ a compact open subgroup of
$\G(\AAf)$ and let $\Sigma$ a set of special points in $\Sh_{K}(\G,
X)_\CC(\CC)$. Then every irreducible component of the Zariski closure of
$\mathbf{\Sigma}$ in $\Sh_{K}(\G, X)_{\CC}$ is a special subvariety.
\end{conj}

One may notice an analogy between this conjecture and the so-called  
Manin-Mumford conjecture (first proved by Raynaud)
which asserts that irreducible components of the Zariski closure of a
set of \emph{torsion} points in an Abelian variety
are translates of Abelian subvarieties by torsion points.
There is a large (and constantly growing) number of proofs of the  
Manin-Mumford conjecture. A proof of the Manin-Mumford conjecture
using a strategy similar to the one used in this paper was recently given by Ullmo and Ratazzi
(see \cite{UllmoRata}).

\subsection{The results.}

Our main result is the following:

\begin{theor} \label{main-thma}
Let $(\G,X)$ be a Shimura datum, $K$ a compact open subgroup of
$\G(\AAf)$ and let $\Sigma$ be a set of
special points in $\Sh_K(\G,X)_{\CC}(\CC)$.
We make {\bf one} of the following assumptions:

\begin{enumerate}
\item[$(1)$] Assume the Generalized Riemann Hypothesis (GRH) for CM fields.
\item[$(2)$] Assume that there exists a faithful representation $\G  
\hookrightarrow \GL_n$ such that
with respect to this representation, the Mumford-Tate groups
$\MT_s$ lie in one $\GL_n(\QQ)$-conjugacy class as $s$ ranges through  
$\Sigma$.
\end{enumerate}

Then every irreducible component of the Zariski closure of
$\mathbf{\Sigma}$ in $\Sh_{K}(\G, X)_{\CC}$ is a special subvariety.
\end{theor}

In fact we prove the following 
\begin{theor} \label{main-thm}
Let $(\G,X)$ be a Shimura datum, $K$ a compact open subgroup of
$\G(\AAf)$ and let $\Sigma$ be a set of
special subvarieties in $\Sh_K(\G,X)_{\CC}$. 
We make {\bf one} of the following assumptions:
\begin{enumerate}
\item[$(1)$] Assume the Generalized Riemann Hypothesis (GRH) for CM fields.
\item[$(2)$] Assume that there exists a faithful representation $\G  
\hookrightarrow \GL_n$ such that
with respect to this representation, the generic Mumford-Tate groups
$\MT_V$ of $V$ lie in one $\GL_n(\QQ)$-conjugacy class as $V$ ranges through  
$\Sigma$. 
\end{enumerate}

Then every irreducible component of the Zariski closure of
$\mathbf{\Sigma}$ in $\Sh_{K}(\G, X)_{\CC}$ is a special subvariety.
\end{theor}

The case of theorem~\ref{main-thm} where $\Sigma$ is a set of special points is theorem~\ref{main-thma}.

\subsection{Some remarks on the history of the Andr{\'e}-Oort conjecture.}

For history and results obtained before 2002, we refer to the introduction of
\cite{EdHilbert}. We just mention that conjecture~\ref{conjecture} was stated by  
Andr{\'e} in 1989 in the case of an irreducible curve in
$\Sh_K(\G,X)_{\CC}$ containing a Zariski dense set of special
points, and in 1995 by  
Oort for irreducible subvarieties of moduli spaces of polarised
Abelian varieties containing a Zariski-dense set of special points.

Let us mention some results we will use in the course of our proof.

In \cite{CU1} (further generalized in \cite{Ullmo} and  
\cite{UllmoYafaev}), the conclusion of the
theorem~\ref{main-thm} is proved for sets $\Sigma$ of {\em strongly  
special} subvarieties in
$\Sh_K(\G,X)_{\CC}$ without assuming $(1)$ or $(2)$ (cf.
section~\ref{section2}).
The statement is proved using ergodic theoretic techniques.

Using Galois-theoretic techniques and geometric properties of Hecke
correspondences, Edixhoven and the second author (see \cite{EdYa})  
proved the conjecture for
curves in Shimura varieties containing infinite sets of special points  
satisfying our
assumption~$(2)$. Subsequently, the second author (in \cite{Ya})  
proved the Andr{\'e}-Oort conjecture for
curves in Shimura varieties assuming the GRH. The main new ingredient
in \cite{Ya} is a theorem on lower bounds for Galois orbits of special points.
In the work \cite{Ed}, Edixhoven proves, assuming the GRH, the  
Andr{\'e}-Oort conjecture for
products of modular curves. In \cite{YaMo}, the second author proves  
the Andr{\'e}-Oort conjecture
for sets of special points satisfying an additional condition.

The authors started working together on this conjecture in
2003 trying to generalize the Edixhoven-Yafaev strategy
to the general case of the Andr{\'e}-Oort conjecture. In the process
two main difficulties occur. One is the question of irreducibility of
transforms of subvarieties under Hecke correspondences.
This problem is dealt with in sections~\ref{critere} and
\ref{hecke}. The other difficulty consists in dealing with higher
dimensional special 
subvarieties. Our strategy is to proceed by induction on the generic
dimension of elements of $\Sigma$. The main ingredient for controlling
the induction was the discovery by Ullmo and the second author in
\cite{UllmoYafaev} of a possible combination of 
Galois theoretic and ergodic techniques. It took form
while the second author was visiting
the University of Paris-Sud in January-February 2005.

\subsection{Acknowledgements.}

The second author would like to express his gratitude to Emmanuel Ullmo for many
conversations he had with him on the topic of the
Andr{\'e}-Oort conjecture. We thank him for his careful reading of the previous versions of
the manuscript and for pointing out some inaccuracies.
We would like to extend our thanks to Richard Pink for going through the
details of the entire proof of the conjecture and contributing valuable comments which significantly  
improved the paper. The second author is grateful to Richard Pink for inviting him to ETH  
Zurich in April 2006. Laurent Clozel read one of the previous versions of the  
manuscript and pointed out
a flaw in the exposition.
We extend our thanks to Bas Edixhoven and Richard Hill
for many discussions on the topic of the Andr{\'e}-Oort conjecture.
This work was initiated during a `research in pairs' stay at
Oberwolfach and continued in many institutions, including the
University of Chicago, University College London, University of
Leiden, AIM at Palo Alto and University of Montreal.
We thank these institutions for their hospitality and sometimes  
financial support. 
The first author is grateful to the NSF for financial support and the
University of Chicago for excellent working conditions, 
the second author to the Leverhulme Trust.

Finally we thank the referee for his unfailing criticism and devoted
work which improved the paper greatly.

\subsection{Conventions.} \label{Firreducible}
Let $F$ be a field. An $F$-algebraic variety is a reduced separated
scheme over $F$, not necessarily irreducible. It is of finite type
over $F$ unless mentioned. A subvariety is always assumed to be a closed subvariety. 

Let $F \subset \CC$ be a number field, $Y_F$ an $F$-algebraic variety
and $Z \subset Y:=Y_F \times_{\Spec \; F} \Spec \, \CC$ a
$\CC$-subvariety. We will use the following common abuse of notation:
$Z$ is said to be $F$-irreducible if $Z = Z_F \times_{\Spec \, F} \Spec \;
\CC$, where $Z_F \subset Y_F$ is an irreducible closed
subvariety.

\subsection{Organization of the paper.}
Sections~\ref{section2} and~\ref{strategy} of the paper explain how to reduce the
theorem~\ref{main-thm} to the more geometric theorem~\ref{main-thm1}
using the Galois/ergodic alternative proven in
\cite{UllmoYafaev}. In these sections we freely use notations
recalled in section~\ref{prelim} and \ref{degree}, which consist in preliminaries. In
addition to fixing notations we prove there the crucial corollary~\ref{compare-degrees}
comparing the degrees of subvarieties under morphisms of Shimura
varieties. The sections~\ref{partb}, \ref{critere}, \ref{hecke}, \ref{le_premier},
\ref{choix} contain the proof of theorem~\ref{main-thm1}. Their role
and their organisation is described in details in section~\ref{strat}.


\section{Equidistribution and Galois orbits.} \label{section2}

In this section we recall a crucial ingredient in the proof of the
theorem~\ref{main-thm}: the Galois/ ergodic alternative from
\cite{UllmoYafaev}.

\subsection{Some definitions.}

\subsubsection{Shimura subdata defining special subvarieties.}
\begin{defi}
Let $(\G,X)$ be a Shimura datum where $\G$ is the generic Mumford-Tate
group on $X$. Let $X^+$ be a connected component of $X$ and
let $K$ be a neat compact open subgroup of
$\G(\AAf)$. We denote by
$S_K(\G, X)_\CC$ the connected component of $\Sh_K(\G,X)_\CC$ image of $X^+ \times \{1\}$ in
$\Sh_{K}(\G,X)_\CC$. Thus $S_K(\G,X)_\CC= \Gamma_K \backslash X^+$, where $\Gamma_K = \G(\QQ)_+
\cap K$ is a neat arithmetic subgroup of the stabiliser $\G(\QQ)_+$ of
$X^+$ in $\G(\QQ)$.
\end{defi}

\begin{defi}
Let $V$ be a special subvariety of $S_K(\G,X)_{\CC}$. We say that a Shimura subdatum $(\HH_V,X_V)$ of $(\G,X)$
defines $V$ if $\HH_V$ is the generic Mumford-Tate group on $X_V$ and there exists a connected component
$X_{V}^+$ of $X_{V}$ contained in $X^+$ such that $V$ is the image of
$X_{V}^+\times \{ 1\}$ in $S_K(\G,X)_\CC$. 

From now on, when we say that a Shimura subdatum $(\HH_V, X_V)$ defines
$V$, the choice of the component $X_V^+ \subset X^+$ will always be tacitly assumed.
\end{defi}

Given a special subvariety $V$ of $S_K(\G,X)_\CC$ there exists a Shimura subdatum $(\HH_V,X_V)$ defining 
$V$ by \cite[lemma 2.1]{UllmoYafaev}. Notice that as an abstract
$\QQ$-algebraic group $\HH_V$ is uniquely defined by $V$ whereas the
embedding $\HH_V\hookrightarrow \G$ is uniquely defined by $V$  up to conjugation by $\Gamma_K$.

\subsubsection{The measure $\mu_V$}
Let $(\G, X)$ be a Shimura datum, $K$ a neat compact open subgroup of
$\G(\AAf)$ and $X^+$ a connected component of $X$. 
Let $(\HH_V, X_V)$ be a Shimura subdatum of $(\G,X)$ defining a special
subvariety $V$ of $S_K(\G, X)_\CC$.
Thus there exists a neat arithmetic group
$\Gamma_V$ of the stabiliser $\HH_V(\QQ)_+$ of $X_V^+$ in
$\HH_V(\QQ)$ and a (finite) morphism 
$$f:
\Gamma_V \backslash X_V^+ \lto S_K(\G, X)_\CC 
$$ 
whose image is $V$.
%

\begin{defi} \label{definition_mesure} 
We define $\mu_V$ to be the probability measure on
$\Sh_K(\G,X)_{\CC}(\CC)$ supported on $V$, push-forward by $f$ of the
standard probability measure on the Hermitian 
locally symmetric space $\Gamma_V\backslash X^{+}_{V}$ induced
by the Haar measure on $\HH_{V}(\RR)_{+}$. 
\end{defi}

\begin{rem}
Notice that the measure $\mu_V$ depends
only on $V$, not on the choice of the embedding $\HH_V
\hookrightarrow \G$. 
\end{rem}

\subsubsection{$\TT$-special subvarieties.}
\begin{defi} \label{strongly_special}
Let $(\G, X)$ be a Shimura datum and let $\lambda: \G \lto \G^\ad$ be
the canonical morphism.
Fix a (possibly trivial) $\RR$-anisotropic $\QQ$-subtorus $\TT$ of $\G^\ad$.
A $\TT$-special subdatum $(\HH,X_{\HH})$ of $(\G,X)$ is a 
Shimura subdatum such that $\HH$ is the generic Mumford-Tate group of
$X_{\HH}$ and $\TT$ is the connected centre of $\lambda(\HH)$.

Let $X^+$ be a connected component of $X$ and let $K$ be a neat compact open subgroup of
$\G(\AAf)$. A special subvariety $V$ of $S_K(\G,X)_\CC$ is $\TT$-special if there
exists a $\TT$-special subdatum $(\HH, X_\HH)$ of $(\G,X)$ such that
$V$ is an irreducible component of the image of $\Sh_{K \cap
  \HH(\AAf)}(\HH, X_\HH)_\CC$ in $\Sh_K(\G, X)_\CC$. 

In the case where $\TT$ is trivial, we call $V$ \emph{strongly} special.
\end{defi}

\begin{rems} 
\begin{itemize}
\item[(a)] If moreover $(\HH, X_\HH)$ defines $V$ then $V$ is said to be
  $\TT$-special {\em standard} in \cite{UllmoYafaev}.
\item[(b)] The definition of {\em strongly special} given
in \cite{CU1} requires that $\HH_V$ is not contained in a proper
parabolic subgroup of $\G$ but as explained in
\cite[rem. 3.9]{Ullmo} this last condition is automatically satisfied.
\end{itemize}
\end{rems}
 
\subsection{The rough alternative.}
With these definitions, the alternative from \cite{UllmoYafaev} can roughly be stated as
follows. 

Let $(\G,X)$ be a Shimura datum with $\G$ semisimple of adjoint type,
$X^+$ a connected component of $X$, let $K$ be a neat compact open
subgroup of $\G(\AAf)$ and $E$ a
number field over which $\Sh_K(\G,X)_{\CC}$ admits a 
canonical model (cf. section~\ref{canonic}). Let $Z \subset
S_K(\G,X)_{\CC}$ be an irreducible subvariety containing a
Zariski-dense union $\cup_{n \in \NN}V_n$ of special subvarieties
$V_n$ of $S_K(\G,X)_{\CC}$. 
\begin{itemize}
\item either there exists an $\RR$-anisotropic $\QQ$-subtorus $\TT$ of
  $\G$ and a subset $\Sigma \subset \NN$ such that each $V_n$, $n \in
  \Sigma$, is $\TT$-special and $\mathbf{\Sigma}=\cup_{n \in \Sigma} V_n$ is
    Zariski-dense in $Z$. Then one can choose $\Sigma$ so that the sequence (after possibly replacing by a subsequence) of  
probability measures $(\mu_{V_{n}})_{n \in \Sigma}$ weakly converges to the probability measure $\mu_{V}$ of
some special subvariety $V$ and for $n$ large, $V_n$ is contained in
$V$. This implies that $Z=V$ is special (cf. theorem~\ref{ul-ya}).
\item otherwise the function $\deg_{L_{K}}(\Gal(\ol\QQ/E) \cdot   
V_n)$ is an
unbounded function of $n$ as $n$ ranges through $\Sigma$ and we can use Galois-theoretic 
methods to study $Z$
(cf. definition~\ref{defdegree} for the definition of the degree $\deg_{L_{K}}$) .
\end{itemize}

\sspace
We now explain this alternative in more details.

\subsection{Equidistribution results.} \label{equidibresult}

Ratner's classification of probability measures on homogeneous spaces of the form $\Gamma \backslash
\G(\RR)^{+}$ (where $\Gamma$ denotes a lattice in $\G(\RR)^+$),  
ergodic under some unipotent flows \cite{Ratner}, and Dani-Margulis
recurrence lemma \cite{Dani-Margulis} enable Clozel and Ullmo
\cite{CU1} to prove the following equidistribution result in the
strongly special case, generalized by Ullmo and Yafaev \cite[theorem
3.8 and corollary 3.9]{UllmoYafaev} to the $\TT$-special case:

\begin{theor} [Clozel-Ullmo, Ullmo-Yafaev]  \label{ul-ya}
Let $(\G,X)$ be a Shimura datum with $\G$ semisimple of adjoint type,
$X^+$ a connected component of $X$ and $K$ a neat compact open subgroup of
$\G(\AAf)$. Let $\TT$ be an $\RR$-anisotropic $\QQ$-subtorus of $\G$.
Let $( V_n )_{n \in \NN}$ be a sequence of $\TT$-special subvarieties
of $S_K(\G,X)_{\CC}$. Let $\mu_{V_{n}}$ be the canonical probability
measure on $\Sh_K(\G,X)_{\CC}$ supported on $V_n$. There exists a
$\TT$-special subvariety $V$ of $S_K(\G,X)_{\CC}$ and a subsequence
$(\mu_{n_{k}})_{k \in \NN}$ weakly converging to   
$\mu_V$. Furthermore
$V$ contains $V_{n_{k}}$ for all $k$ sufficiently large.
In particular, the irreducible components of the Zariski closure of a set of
$\TT$-special subvarieties of $S_K(\G,X)_{\CC}$ are special.
\end{theor}

\begin{rems} \label{rem0} 

\begin{itemize}
\item[(1)]
Note that a special point of $S_K(\G,X)_\CC$, whose Mumford-Tate group is
a non-central torus, is not strongly special. Moreover, given an $\RR$-anisotropic $\QQ$-subtorus $\TT$
of $\G$, the connected Shimura variety $S_{K}(\G, X)_{\CC}$ contains only a  
finite number of $\TT$-special points (cf. \cite[lemma
3.7]{UllmoYafaev}). Thus theorem~\ref{ul-ya} says nothing {\em
  directly} on the Andr{\'e}-Oort conjecture.

\item[(2)] In fact the conclusion of the
theorem~\ref{ul-ya} is simply not true for special points: they are  
dense for the Archimedian topology in $S_K(\G,X)_\CC({\CC})$, so just consider a sequence of
special points converging to a non-special point in
$S_K(\G,X)_\CC({\CC})$ (or diverging to a cusp if
$S_K(\G,X)_\CC({\CC})$ is non-compact). In this case the
corresponding sequence of Dirac delta measures will converge to the Dirac delta  
measure of the non-special point (respectively escape to infinity).

\item[(3)] 
There is a so-called equidistribution conjecture which implies the
Andr{\'e}-Oort conjecture and
much more. A sequence $(x_n)$ of points of 
$S_{K}(\G, X)_\CC(\CC)$ is called \emph{strict}
if for any proper special subvariety $V$ of $\Sh_{K}(\G, X)_\CC(\CC)$, the set
$$
\{ n : x_n \in V\}
$$
is finite.
Let $E$ be a field of definition of a canonical model
of $\Sh_{K}(\G, X)_\CC(\CC)$. To any special point $x$, one associates a probability
measure $\Delta_x$ on $\Sh_{K}(\G, X)_\CC(\CC)$ as follows :
$$
\Delta_x = \frac{1}{| \Gal(\ol E/E)\cdot x|} \sum_{y\in \Gal(\ol E/E)\cdot x} \delta_{y}
$$
where $\delta_y$ is the Dirac measure at the point
$y$ and $| \Gal(\ol E/E)\cdot x|$ denotes the cardinality of the
Galois orbit $\Gal(\ol E/E)\cdot x$.
The equidistribution conjecture predicts that if $(x_n)$ is a strict sequence
of special points, then the sequence of measures $\Delta_{x_n}$ 
weakly converges to the canonical probability measure attached to $\Sh_{K}(\G, X)_\CC(\CC)$. 
This statement implies the Andr{\'e}-Oort conjecture.
The equidistribution conjecture is known for modular curves
and is open in general. There are some recent conditional results for Hilbert modular varieties
due to Zhang (see \cite{Zhang}).
For more on this, we refer to the survey \cite{UllmoMontreal}.

\end{itemize}
\end{rems}

\subsection{Lower bounds for Galois orbits.} \label{lowerGalois}
In this section, we recall the lower bound obtained in
\cite{UllmoYafaev} for the degree of the Galois orbit of a special subvariety which is not strongly special.

\subsubsection{Data associated to a special subvariety}
\begin{defi} \label{notation}
Let $(\G, X)$ be a Shimura datum and $X^+$ a connected component of $X$.
Let $K = \prod_{p \;\textnormal{prime}} K_p$ be a neat compact open
subgroup of $\G(\AAf)$. Let $(\HH_V, X_V)$ be a Shimura subdatum of
$(\G,X)$ defining a special subvariety $V$ of $S_K(\G, X)_\CC$.

We denote by:
\begin{itemize}
\item $E_{V}$ the reflex field of $(\HH_{V}, X_{\HH_{V}})$.
\item $\TT_{V}$ the connected centre of $\HH_{V}$. It is a (possibly trivial) torus. 
\item $K^{\mathrm{m}}_{\TT_{V}}= \prod_{p\;\textnormal{prime}}  K^{\mathrm{m}}_{\TT_{V},p}$ the maximal compact open subgroup of
  $\TT_{V}(\AAf)$, where $K^{\mathrm{m}}_{\TT_{V},p}$ denotes the
  maximal compact open subgroup of $\TT_V(\QQ_p)$.
\item $K_{\TT_{V}}$ the compact open subgroup $\TT_{V}(\AAf)\cap K
  \subset K^{\mathrm{m}}_{\TT_{V}}$. Thus $K_{\TT_{V}} = \prod_{p
    \;\textnormal{prime}}  K_{{\TT_V},p}$, where $ K_{{\TT_V},p}:=
  \TT_V(\QQ_p) \cap K_p$.
\item $\C_{V}$ the torus $\HH_{V} /\HH_{V}^{\textnormal{der}}$  
isogenous to $\TT_{V}$.
\item $d_{\TT_{V}}$ the absolute value of the discriminant of the
  splitting field $L_V$ of $\C_V$, and $n_V$ the degree of $L_V$ over $\QQ$.
\item $\beta_V := \log (d_{\TT_{V}})$. 

\end{itemize}
\end{defi}

\begin{rem} \label{independance}
Notice that the group $K_{\TT_{V}}$ depends on the particular embedding $\HH_V
\hookrightarrow \G$ (which is determined by $V$ up to conjugation by
$\Gamma= \G(\QQ)_+ \cap K$). On the other hand the other quantities
defined above, and also the
indices $| K^{\mathrm{m}}_{\TT_{V},p} /K_{{\TT_V},p}|$, $p$ prime,
depend on $V$ but not on  the particular embedding $\HH_V
\hookrightarrow \G$.
\end{rem}

We will frequently make use of the following lemma:
\begin{lem} \label{compactness}
With the above notations assume moreover that the group
$\G$ is semisimple of adjoint type.
Then the $\QQ$-torus $\TT_V$ is $\RR$-anisotropic.
\end{lem}
\begin{proof}
As $\TT_V$ is the connected centre of the generic Mumford-Tate group
$\HH_V \subset \G$ of $V$ the group $\TT_V(\RR)$ fixes some point $x$ of $X$. As
$\G$ is semisimple of adjoint type the stabiliser of $x$ in $\G(\RR)$ is
compact.
\end{proof}

\subsubsection{The lower bound}
One of the main ingredients of our proof of theorem~\ref{main-thm} is the following lower bound for the degree of Galois orbits obtained in 
\cite[theorem 2.19]{UllmoYafaev} (we refer to the section~\ref{degree} for
the definition of the degree function $\deg_{L_{K}}$):

\begin{theor}[Ullmo-Yafaev] \label{GaloisOrbits} 
Assume the GRH for CM fields. 
Let $(\G, X)$ be a Shimura datum with $\G$ semisimple of adjoint type
and let $X^+$ be a fixed connected component of $X$.

Fix positive integers $R$ and $N$. There exist a positive real number
$B$ depending only on $\G$, $X$ and $R$ and a positive constant $C(N)$
depending on $\G$, $X$, $R$ and $N$ such that the following holds.

Let $K =\prod_{p \; \textnormal{prime}} K_p$ be a neat compact open
subgroup of $\G(\AAf)$. Let $V$ be a special subvariety of $S_{K}(\G, X)_\CC$ and
$(\HH_V, X_V)$ a Shimura subdatum of $(\G, X)$ defining $V$. Let $F$
be an extension of $\QQ$ of degree at most $R$ containing the reflex
field $E_V$ of $(\HH_V, X_V)$. Let $K_{\HH_{V}}
:= K \cap \HH_V(\AA_f)$.
Then:
\begin{equation} \label{lower}
\deg_{L_{K_{{\HH}_{V}}}}(\Gal(\ol\QQ/F) \cdot V) >
 C(N) \cdot  \Big(\prod_{\substack{ p \, \textnormal{prime} \\
       K^\mathrm{m}_{{\TT_V},p}\not= K_{{\TT_V},p}}}
 \max(1, B\cdot 
   |K^\mathrm{m}_{\TT_V,p}/K_{\TT_V,p}|) \Big ) \cdot  
\beta_V^N \;\;.
\end{equation}

Furthermore, if one fixes a faithful representation $\G  
\hookrightarrow \GL_n$ and one considers only the subvarieties $V$ such that the  
associated
tori $\TT_{V}$ lie in one $\GL_n(\QQ)$-conjugacy class, then the  
assumption of the GRH can be dropped.

\end{theor}

\begin{rem}
The lower bound~(\ref{lower}) still holds if we replace $V$ by $Y$
an irreducible subvariety of $V$ defined over $\ol \QQ$
whose Galois orbits are ``sufficiently similar'' to those of $V$. For
simplicity we refer to \cite[theor.2.19]{UllmoYafaev} for this refined
statement, which we will use in the proof of the lemma~\ref{notspecial}.
\end{rem}

\subsection{The precise alternative.}

Throughout the  paper we will be using the following notations.

\begin{defi} \label{alpha}
 Let $(\G, X)$ be a Shimura datum with $\G$ semisimple of adjoint
 type. Let $X^+$ be a fixed connected
 component of $X$. 

We fix $R$ a positive integer such that for any Shimura subdatum $(\HH, X_\HH)$ of $(\G,
 X)$ there exists an extension $F$ of $\QQ$ of degree at most $R$ containing the Galois
closure of the reflex field $E_\HH$ of $(\HH, X_\HH)$. Such an $R$
exists by \cite[lemma 2.5]{UllmoYafaev}.

Let $K = \prod_{p \; \textnormal{prime}} K_p$ be a neat 
 compact open subgroup of $\G(\AAf)$. Let $V$ be a special subvariety of $S_{K}(\G, X)_\CC$ and
$(\HH_V, X_V)$ a Shimura subdatum of $(\G, X)$ defining $V$.

With the notations of definition~\ref{notation} and with $B$ as in 
theorem~\ref{GaloisOrbits}  we define:
$$\alpha_V:=\prod_{\substack{p \, \textnormal{prime} \\
    K^m_{{\TT_V},p}\not= K_{{\TT_V},p}}} \max(1, B \cdot |K^m_{\TT_V,p}/K_{\TT_V,p}|) \;\;.$$  
\end{defi}

\begin{rem}
By remark~\ref{independance} the quantity $\alpha_V$ depends only on $V$
and not on the particular embedding $\HH_V
\hookrightarrow \G$.
\end{rem}

The alternative roughly explained in the introduction to section~\ref{section2} can
now be formulated in the following theorem (\cite[theorem 3.10]{UllmoYafaev}).

\begin{theor}[Ullmo-Yafaev] \label{alter}
 Let $(\G, X)$ be a Shimura datum with $\G$ semisimple of adjoint
 type. Let $X^+$ be a fixed connected
  component of $X$. Fix $R$ a positive integer as in
  definition~\ref{alpha}.

Let $K =\prod_{p \; \textnormal{prime}} K_p$ be a neat 
  compact open subgroup of $\G(\AAf)$ and let $\Sigma$ be a set of special subvarieties $V$ of $S_{K}(\G,X)_{\CC}$
such that $\alpha_V\beta_V$ is bounded as $V$ ranges through $\Sigma$.

There exists a finite set $\{\TT_1, \cdots, \TT_r\}$ of $\RR$-anisotropic $\QQ$-subtori of
$\G$ such that any $V$ in $\Sigma$ is $\TT_i$-special for some $i \in
\{1, \cdots, r\}$.
\end{theor}

\section{Reduction and strategy.} \label{strategy}

From now on we will use the following convenient terminology:

\begin{defi}
Let $(\G,X)$ be a Shimura datum and $K$ a compact open subgroup of $\G(\AAf)$.
Let $\Sigma$ be a set of special subvarieties of $\Sh_K(\G,X)_{\CC}$.
A subset $\Lambda$ of $\Sigma$ is called a modification of $\Sigma$ if
$\mathbf{\Lambda}$ and $\mathbf{\Sigma}$ have the same Zariski closure
in $\Sh_K(\G,X)_{\CC}$ (recall, cf. definition~\ref{defi111}, that $\mathbf{\Lambda}$ and $\mathbf{\Sigma}$
denote the unions of subvarieties in $\Lambda$ and $\Sigma$ respectively).
\end{defi}

\subsection{First reduction.}
We first have the following reduction of the proof of theorem~\ref{main-thm}:

\begin{theor} \label{main-thm1}
Let $(\G,X)$ be a Shimura datum and $K$ a compact open subgroup of $\G(\AAf)$.
Let $Z$ be an irreducible subvariety of $\Sh_K(\G,X)_{\CC}$.
Suppose that $Z$ contains a Zariski dense set $\mathbf{\Sigma}$, which is a union of
special subvarieties $V$, $V \in \Sigma$, all of the same dimension
$n(\Sigma)< \dim Z$.

We make {\bf one} of the following assumptions:
\begin{enumerate}
\item[$(1)$] Assume the Generalized Riemann Hypothesis (GRH) for CM fields.
\item[$(2)$] Assume that there is a faithful representation $\G  
\hookrightarrow \GL_n$ such that
with respect to this representation, the connected centres
$\TT_V$ of the
generic Mumford-Tate groups $\HH_V$ of $V$ lie in one  
$\GL_n(\QQ)$-conjugacy class as $V$ ranges through $\Sigma$.
\end{enumerate}

Then
\begin{itemize}
\item[(a)] The variety $Z$ contains a Zariski dense set $\Sigma'$ of  
special subvarieties
of constant dimension
$
n(\Sigma') > n(\Sigma)
$.
\item[(b)]
Furthermore, if $\Sigma$ satisfies the condition~$(2)$, one can choose
$\Sigma'$ also satisfying~$(2)$.
\end{itemize}
\end{theor}
 
\begin{prop}
Theorem~\ref{main-thm1} implies the main theorem~\ref{main-thm}.
\end{prop}
\begin{proof}
Let $\G$, $X$, $K$ and $\Sigma$ as in the main theorem~\ref{main-thm}. Without loss of generality one
can assume that the Zariski closure $Z$ of $\mathbf{\Sigma}$ is
irreducible. Moreover by Noetherianity one can assume that all the $V
\in \Sigma$ have the same dimension $n(\Sigma)$.  

 Notice that the assumption~$(2)$ of the
theorem~\ref{main-thm}  implies the assumption~$(2)$ of the
theorem~\ref{main-thm1}. We then apply theorem~\ref{main-thm1},$(a)$ to
$\Sigma$: the subvariety $Z$ contains a Zariski-dense set
$\mathbf{\Sigma'}$ of special subvarieties $V'$, $V' \in \Sigma'$, of
constant dimension $n(\Sigma') > n(\Sigma)$.

By theorem~\ref{main-thm1},$(b)$ one can replace $\Sigma$ by $\Sigma'$.
Applying this process recursively and as $n(\Sigma') \leq \dim(Z)$, we  
conclude that $Z$ is special.
\end{proof}

\subsection{Second reduction.}
Part~$(b)$ of theorem~\ref{main-thm1}  will be dealt with in  
section~\ref{partb}.
Part~$(a)$ of theorem~\ref{main-thm1} can itself be reduced to the
following theorem:

\begin{theor} \label{main-thm2}
Let $(\G, X)$ be a Shimura datum with $\G$ semisimple of adjoint type
and let $X^+$ be a connected component of $X$. Fix $R$ a positive
integer as in the definition~\ref{alpha}.

Let $K= \prod_{p \; \textnormal{prime}} K_p$ be a neat compact
open subgroup of $\G(\AAf)$. Let $Z$ be a Hodge generic geometrically irreducible subvariety of the
connected component $S_K(\G,X)_{\CC}$ of $\Sh_{K}(\G, X)_{\CC}$.
Suppose that $Z$ contains a Zariski dense set $\mathbf{\Sigma}$, which
is a union of special subvarieties $V$, $V \in \Sigma$, all of the same dimension
$n(\Sigma)$ and such that for any
modification $\Sigma'$ of $\Sigma$ the set $\{ \alpha_V\beta_V, \; V \in \Sigma'\}$ is
unbounded (with the notations of definitions~\ref{notation} and \ref{alpha}).

We make {\bf one} of the following assumptions:
\begin{enumerate}
\item[$(1)$] Assume the Generalized Riemann Hypothesis (GRH) for CM fields.
\item[$(2)$] Assume that there is a faithful representation $\G  
\hookrightarrow \GL_n$ such that
with respect to this representation, the connected centres
$\TT_V$ of the
generic Mumford-Tate groups $\HH_V$ of $V$ lie in one  
$\GL_n(\QQ)$-conjugacy class as $V$ ranges through $\Sigma$.
\end{enumerate}

After possibly replacing $\Sigma$ by a modification, for every $V$ in
$\Sigma$ there exists a special subvariety $V'$ such that $V
\subsetneq  V' \subset Z$. 
\end{theor}

\begin{prop}  \label{imply}
Theorem~\ref{main-thm2} implies theorem~\ref{main-thm1} $(a)$.
\end{prop}

\begin{proof}
Let $(\G, X)$, $K$, $Z$ and $\Sigma$ be as in theorem~\ref{main-thm1}.

First let us reduce the proof of
theorem~\ref{main-thm1} to the case where in addition $Z$ satisfies the
assumptions of theorem~\ref{main-thm2}.

Notice that the image of a special subvariety by a morphism of
Shimura varieties deduced from a 
morphism of Shimura data is a special subvariety. Conversely any
irreducible component of the preimage of a special subvariety by such
a morphism is special. This implies that if $K\subset \G(\AAf)$ is a compact open
subgroup and if $K' \subset K$ is a finite index subgroup then
theorem~\ref{main-thm1}$(a)$ is true at level $K$ if and 
only if it is true at level $K'$. In particular we can assume without
loss of generality that $K$ is a product $\prod_{p \;
  \textnormal{prime}} K_p$ and that $K$ is neat.

We can assume that the variety $Z$ in theorem~\ref{main-thm1} is Hodge generic.
To fulfill this condition, replace $\Sh_K(\G,X)_{\CC}$ by the smallest
special subvariety of $\Sh_K(\G,X)_{\CC}$ containing $Z$
(cf. \cite[prop.2.1]{EdYa}). This comes down to replacing $\G$ with
the generic Mumford-Tate group on $Z$.

Let $(\G^\ad, X^\ad)$ be the Shimura datum adjoint to $(\G, X)$ and
$\lambda : (\G, X) \lto (\G^\ad, X^\ad)$ the natural morphism of
Shimura data. For $K \subset \G(\AAf)$ sufficiently small let
$K^\ad$ be a neat compact open subgroup of $\G^\ad(\AAf)$ containing
$\lambda(K)$.
Consider the finite morphism of Shimura varieties
$f:\Sh_K(\G, X)_\CC \lto \Sh_{K^\ad}(\G^\ad, X^\ad)_\CC$. Let $\Sigma^\ad$ be the set of special
subvarieties $f(V)$ of $\Sh_{K^\ad}(\G^\ad, X^\ad)_\CC$, $V \in \Sigma$.
In order to be able to replace $\G$ by $\G^{\ad}$, we need to check
that if $\Sigma$ satisfies the assumption (2), then $\Sigma^\ad$ also satisfies
the assumption~$(2)$. 
For $V$ in $\Sigma$, let $(\HH_V,X_V)$ be the Shimura datum defining $V$ and $\TT_V$ be the connected centre of $\HH_V$.
Then the tori $\TT_V$ (and hence the tori $\lambda(\TT_V)$) are split by the same field.
Choose a faithful representation $\G^{\ad}\hookrightarrow \GL_m$. By \cite{UllmoYafaev}, lemma 3.13, part (i), the
tori $\lambda(\TT_V)$ lie in finitely many $\GL_m(\QQ)$-conjugacy
classes. It follows that, after replacing $\Sigma^\ad$ by a
modification, the assumption~$(2)$ for $f(Z)$ is satisfied.
 Applying our
first remark to the morphism $f$ we obtain that theorem~\ref{main-thm1} $(a)$ for
$(\G^\ad, X^\ad)$ implies theorem~\ref{main-thm1} $(a)$ for $(\G,
X)$. Thus we reduced the proof of theorem~\ref{main-thm1} $(a)$ to the
case where $\G$ is semisimple of adjoint type.

We can also assume that $Z$ is contained in $S_{K}(\G, X)_{\CC}$ as
proving theorem~\ref{main-thm1} 
for $Z$ is equivalent to proving theorem~\ref{main-thm1} for any
irreducible component of its image under some Hecke correspondence.
In particular the quantities $\alpha_V$ and $\beta_V$, $V \in \Sigma$,
are well-defined. 

If for some modification $\Sigma'$ of
$\Sigma$ the set $\{ \alpha_V\beta_V, \; V \in \Sigma'\}$ is
bounded, by theorem~\ref{alter} and by Noetherianity there exists an $\RR$-anisotropic
$\QQ$-subtorus $\TT$ of $\G$ and a modification of
$\Sigma$ such that any element of this modification is $\TT$-special. Applying theorem~\ref{ul-ya} one obtains that $Z$ is special.

Finally we can assume that $Z$ satisfies the
hypothesis of theorem~\ref{main-thm2}: we have reduced the proof of
theorem~\ref{main-thm1} to the case where in addition $Z$ satisfies the
assumptions of theorem~\ref{main-thm2}.

\sspace
Let $\Sigma'$ be the set of the special subvarieties $V'$ obtained
from theorem~\ref{main-thm2} applied to $Z$. Thus $Z$ contains the
Zariski-dense set $\mathbf{\Sigma'} = \cup_{V' \in \Sigma'} V'$.
After possibly replacing $\Sigma'$ by a modification, we can assume by
Noetherianity of $Z$ that the subvarieties in $\Sigma'$ have the
same dimension $n(\Sigma')> n(\Sigma)$. This
proves the theorem~\ref{main-thm1} $(a)$ assuming theorem~\ref{main-thm2}.
\end{proof}

\subsection{Sketch of the proof of the Andr\'e-Oort conjecture in the
  case where $Z$ is a curve.} 

The strategy for proving theorem~\ref{main-thm2} is fairly complicated. We first recall the
strategy developed in \cite{EdYa} in the case where $Z$ is a curve.
In the next section we explain why this strategy {\em cannot} be
directly generalized to  higher dimensional cases.

\sspace
As already noticed in the proof of proposition~\ref{imply} one can
assume without loss of generality that the group $\G$ is semisimple
of adjoint type,
$Z$ is Hodge generic (i.e. its generic Mumford-Tate group is equal to
$\G$),  and $Z$ is contained in the connected component $S_K(\G, X)_{\CC}$
of $\Sh_K(\G, X)_{\CC}$. The proof of the
theorem~\ref{main-thma} in the case where $Z$ is a curve then relies on three
ingredients.

\subsubsection{}
The first one is a geometric {\em criterion} for a Hodge generic  
subvariety $Z$ to be
special in terms of Hecke correspondences. Given a Hecke
correspondence $T_m$, $m \in \G(\AAf)$ (cf. section~\ref{neutral}) we
denote by $T_{m}^{0}$ the correspondence it induces on $S_K(\G,
X)_{\CC}$. Let $q_i$, $1 \leq i \leq n$, be elements of
$\G(\QQ)_{+} \cap KmK$ defined by the equality
$$
\G(\QQ)_{+} \cap KmK = \coprod_{1 \leq i \leq n } \Gamma_K q_{i}^{-1} \Gamma_K \;\;.
$$
Let $T_{q_i}$, $1 \leq i \leq n$, denote the correspondence on $S_K(\G,X)_\CC$ induced by the action of $q_i$
on $X^+$ (in general it does not coincide with the correspondence on
$S_K(\G,X)_\CC$  induced by the Hecke correspondence $T_{q_{i}}$
on $\Sh_K(\G, X)_\CC$). The correspondence $T_m^0$ decomposes as
$T_m^0 = \sum_{1 \leq i \leq n} T_{q_{i}}$.

\begin{theor} \cite[theorem 7.1]{EdYa} \label{crit1}
Let $\Sh_K(\G, X)_{\CC}$ be a Shimura variety, with $\G$ semisimple of
adjoint type. Let $Z \subset S_K(\G, X)_{\CC}$ be a Hodge generic
subvariety of the connected component $S_K(\G, X)_{\CC}$ 
of $\Sh_K(\G, X)_{\CC}$.
Suppose there exist a prime $l$ and an element $m \in \G(\QQ_l)$ such  
that the
neutral component $T_{m}^{0}= \sum_{i=1}^{n} T_{q_{i}}$ of the Hecke  
correspondence $T_m$
associated with $m$ has the following properties:
\begin{enumerate}
\item $Z \subset T_{m}^{0}Z$.
\item For any $i \in \{1, \cdots n \}$, the varieties  $T_{q_{i}}Z$ and $T_{q_i^{-1}}Z$ are irreducible. 
\item For any $i \in \{1, \cdots n \}$ the $T_{q_{i}} +
T_{q_{i}^{-1}}$-orbit is dense in $S_K(\G, X)$.
\end{enumerate}
Then $Z=S_K(\G, X)$, in particular $Z$ is special.
\end{theor}

\sspace
\noi
From $(1)$ and $(2)$ one deduces the
existence of one index $i$ such that $Z= T_{q_{i}}Z = T_{q_i^{-1}}Z$.
It follows that $Z$ contains a
$T_{q_{i}} + T_{q_{i}^{-1}}$-orbit.
The equality $Z=S_K$ follows from $(3)$.

\sspace
In the case where $Z$ is a curve one proves the existence of a prime  
$l$ and of an element
$m \in \G(\QQ_l)$ satisfying
these properties as follows. The property~$(3)$ is easy to obtain: it
is satisfied by any $m$ such that for each simple factor $\G_j$ of $\G$,
the projection of $m$ to $\G_j(\QQ_l)$ is not contained in a compact subgroup (see \cite{EdYa}, Theorem 6.1). 
The property~$(2)$, which is crucial for this strategy, is obtained by
showing that for any prime $l$ outside a
finite set of primes ${\mathcal P}_Z$ and any
$q \in \G(\QQ)^{+} \cap (\G(\QQ_l)\times \prod_{p \not =l}K_p)$, the variety $T_qZ$ is
irreducible. This is a corollary of a result 
due independently to
Weisfeiler and Nori (cf. theorem~\ref{Wei-Nori}) applied to the
Zariski closure of the image of the monodromy representation. This
result implies that for all 
$l$ except those in a finite set ${\mathcal P}_Z$, the closure in  
$\G(\QQ_l)$ of the image of the monodromy
representation for the $\ZZ$-variation of Hodge structure on the
smooth locus $Z^{{\textnormal{sm}}}$ of $Z$ coincides with the
closure of $K\cap \G(\QQ)^+$ in $\G(\QQ_l)$.
To prove the property~$(1)$ one uses Galois orbits of special points contained
in $Z$ and the fact that Hecke correspondences commute with the Galois  
action. First one notices that
$Z$ is defined over a number field $F$, finite extension of the reflex field
$E(\G, X)$ (cf. section~\ref{canonic}). If $s \in Z$ is a special
point, $r_s$ the associated reciprocity morphism and $m \in \G(\QQ_l)$  
belongs to $r_{s}((\QQ_l \otimes F)^*) \subset \MT(s)(\QQ_l)$ then the Galois orbit $\Gal(\Qbar/F)\cdot s$ is
contained in the intersection $Z \cap T_mZ$. If this intersection is
proper its cardinality $Z \cap T_mZ$ is bounded above by a uniform constant times the degree
$[K_l:K_l\cap mK_lm^{-1}]$  of the correspondence $T_m$. To find $l$  
and $m$ such
that $Z \subset T_mZ$ it is then enough to exhibit an
$m \in r_{s}((\QQ_l \otimes F)^*)$ such that the cardinality $|\Gal(\Qbar/F).s|$ is larger than $[K_l:K_l
\cap mK_lm^{-1}]$. This is dealt with by the next two ingredients.

\subsubsection{}
The second ingredient claims the existence of ``unbounded'' Hecke
correspondences of controlled degree
defined by elements in
$r_{s}((\QQ_l \otimes F)^*)$:

\begin{theor}\cite[corollary 7.4.4]{EdYa} \label{adeq}
There exists an integer $k$ such that for all $s \in \Sigma$ and  for
any prime $l$  splitting $MT(s)$ such that $\MT(s)_{\FF_l}$ is a
torus, there exists an
$m \in r_{s}((\QQ_l \otimes F)^*) \subset \MT(s)(\QQ_l)$ such that
\begin{enumerate}
\item for any simple factor $\G_i$ of $\G$ the image of $m$ in $\G_i(\QQ_l)$
   is not in a compact subgroup.
\item $[K_l:K_l \cap mK_lm^{-1}]\ll l^{k}$.
 \end{enumerate}
\end{theor}

\subsubsection{}
The third ingredient is a lower bound for
$|\Gal(\Qbar/F) \cdot s|$ due to Edixhoven, and improved in theorem~\ref{GaloisOrbits}.

\subsubsection{}
Finally using this lower bound for $|\Gal(\Qbar/F)\cdot s|$ and the  
effective Chebotarev theorem consequence of the GRH one
proves the existence for any special point $s \in \Sigma$ with a  
sufficiently big
Galois orbit of a prime $l$ outside ${\mathcal P}_Z$, splitting  
$\MT(s)$, such that
$\MT(s)_{\FF_{l}}$ is a torus and such that
$|\Gal(\Qbar/F).s|\gg l^{k}$. Effective Chebotarev is not needed under the
assumption that the $\MT(s)$, $s \in \Sigma$, are isomorphic. 
The reason being that in this case, the splitting field of the $\MT(s)$
is constant and the classical Chebotarev theorem provides us with a
suitable $l$.

We then
choose an $m$ satisfying the conditions of the theorem~\ref{adeq}.
As $|\Gal(\Qbar/F).s|\gg [K_l:K_l  \cap mK_lm^{-1}]$ one obtains $Z
\subset T_mZ$ and by the criterion~\ref{crit1} the subvariety $Z$ is special.

\subsection{Strategy for proving the theorem~\ref{main-thm2}: the
  general case.} \label{strat}

Let $\G$, $X$, $X^+$, $K$, $Z$ and $\Sigma$ be as in the statement of
the  theorem~\ref{main-thm2}.

Notice that the idea of the proof of \cite{EdYa} generalizes to the case where
$\dim Z = n(\Sigma)+1$ (cf. section~\ref{edix}).
In the general case, for a $V$ in $\Sigma$ with $\alpha_V\beta_V$ sufficiently large we
want to exhibit $V'$ special subvariety in $Z$ containing $V$
properly. 

\sspace
Our first step (section~\ref{critere}) is geometric: we give a criterion (theorem~\ref{theor2}) similar to
criterion~\ref{crit1} saying that an inclusion $Z\subset  T_mZ$, for a
prime $l$ and an element $m \in \HH_V(\QQ_l)$ satisfying certain
conditions, implies that $V$ is properly contained in a special
subvariety $V'$ of $Z$.  

The criterion we need has to be much more subtle than
the one in \cite{EdYa}. In the characterization of \cite{EdYa}, in
order to obtain the irreducibility of $T_mZ$ the prime $l$ must   
be outside some finite set $\mathcal{P}_Z$ of primes. It seems
impossible to make the set of bad primes ${\mathcal P}_Z$ explicit in
terms of numerical invariants of $Z$, except in a few cases where the
Chow ring of the Baily-Borel compactification of 
$\Sh_K(\G,X)_{\CC}$ is easy to describe (like the case considered by Edixhoven,
where $\Sh_{K}(\G, X)_{\CC}$ is a product $\prod_{i=1}^{n} X_{i}$ of  
modular curves, and where he shows that for a $k$-dimensional subvariety $Z$ dominant on all
factors $X_{i}$, $1 \leq i \leq n$, the bad primes
$p \in \mathcal{P}_Z$ are smaller than the supremum of the degree of the
projections of $Z$ on the $k$-factors $X_{i_{1}} \times \cdots \times  
X_{i_{k}}$ of $\Sh_{K}(\G, X)_{\CC}$). In particular that characterization is not suitable for our induction.

Our criterion~\ref{theor2} for an irreducible
subvariety $Z$ containing a special subvariety $V$ which is not strongly special and  
satisfying $Z \subset
T_{m}Z$ for some $m \in \TT_{V}(\QQ_{l})$ to contain a special  
subvariety $V'$ containing $V$ properly does no longer require the
irreducibility of $T_mZ$. In particular it is valid for \emph{any} prime $l$, outside $\mathcal{P}_Z$ or
not. Instead we notice that the inclusion $Z\subset T_mZ$ implies
that $Z$ contains the image $Z'$ in $\Sh_K(\G,X)_\CC$
of the $\langle K'_l, (k_1 m k_2)^n \rangle$-orbit of (one irreducible
component of) the preimage of $V$ in the
pro-$l$-covering of $\Sh_K(\G,X)_\CC$. Here $k_1$ and $k_2$ are some elements of
$K_l$, $n$ some positive integer and $K'_l$ the $l$-adic closure of  
the image of the
monodromy of $Z$. If the group $\langle K'_l, (k_1 m k_2)^n \rangle$  
is not compact,
then the irreducible component of $Z'$ containing $V$ contains a
special subvariety $V'$ of $Z$ containing $V$ properly.

The main problem with this criterion is that the group
$\langle K'_l, k_1 m k_2\rangle$ can be compact, containing $K'_l$  
with very small index.
This is the case in Edixhoven's counter-example  
\cite[Remark 7.2]{Ed2Curves}.
In this case $\G= {\rm PGL}_2\times {\rm PGL}_2$,  
$K'_l:=\Gamma_0(l)\times \Gamma_0(l)$ and $k_1 m k_2$ is
$w_l\times w_l$, the product of
two Atkin-Lehner involutions. The index
$[\langle K'_l, k_1 m k_2\rangle : K'_l ]$ is four.

\sspace
Our second step (section~\ref{hecke}) consists in getting rid of this
problem and is purely group-theoretic. We notice
that if $K_l$ is \emph{not a maximal} compact open   
subgroup but is contained in a well-chosen \emph{Iwahori} subgroup of $\G(\QQ_l)$, then for
``many'' $m$ in $\TT_{V}(\QQ_l)$ the element $k_1mk_2$ is not contained in
a compact subgroup for any $k_1$ and $k_2$ in $K_l$. This is our
theorem~\ref{good Hecke} about the existence of adequate Hecke
correspondences. 
The proof relies on simple properties of the Bruhat-Tits decomposition of
$\G(\QQ_{l})$.

\sspace
Our third step (section~\ref{le_premier}) is Galois-theoretic and geometric. We use
theorem~\ref{GaloisOrbits}, theorem~\ref{theor2}, theorem~\ref{good 
  Hecke} to show (under one of the assumptions of theorem~\ref{main-thm1}) that the
existence of a prime number $l$ satisfying certain conditions forces a
subvariety $Z$ of $\Sh_K(\G,X)_\CC$ containing a special but not strongly 
special subvariety $V$ to contain a special subvariety $V'$
containing $V$ properly. The proof is a nice geometric induction on $r = \dim Z -
\dim V$. 

\sspace
Our last step (section~\ref{choix}) is number-theoretic: we complete
the proof of the theorem~\ref{main-thm2} and hence of
theorem~\ref{main-thm} by exhibiting, using effective Chebotarev under the
GRH (or usual Chebotarev under the second assumption of
theorem~\ref{main-thm}), a prime $l$ satisfying our desiderata. For
this step it is crucial that both the index of 
an Iwahori subgroup in a maximal compact subgroup of $\G(\QQ_{l})$ and
the degree of the correspondence $T_m$ are bounded by a uniform power
of $l$.

\section{Preliminaries.} \label{prelim}

\subsection{Shimura varieties.} \label{notations}
In this section we define some notations and recall some standard facts about Shimura
varieties that we will use in this paper. We refer to \cite{De1},  
\cite{De2}, \cite{Milne} for details.

As far as groups are concerned, reductive algebraic groups are assumed
to be connected. The exponent $^{0}$ denotes the  
algebraic neutral component and
the exponent $^{+}$ the topological neutral component. 
Thus if $\G$ is a $\QQ$-algebraic group 
$\G(\RR)^{+}$ denotes the topological neutral component of the real Lie
group of $\RR$-points $\G(\RR)$. We also denote by
$\G(\QQ)^{+}$ the intersection $ \G(\RR)^{+} \cap \G(\QQ)$. 

When $\G$
is reductive we denote by $\G^{\textnormal{ad}}$ the adjoint group of
$\G$ (the quotient of $\G$ by its center) and by $\G(\RR)_{+}$ the  
preimage in $\G(\RR)$ of
$\G^{\textnormal{ad}}(\RR)^{+}$. The notation $\G(\QQ)_{+}$ denotes the
intersection $\G(\RR)_{+} \cap \G(\QQ)$. In particular when $\G$ is  
adjoint then $\G(\QQ)^+ = \G(\QQ)_+$.

For any topological space $Z$, we denote by $\pi_{0}(Z)$
 the set of connected components of $Z$.

\subsubsection{Definition}  \label{neutral}
Let $(\G,X)$ be a Shimura datum. We fix $X^{+}$ a connected component
of $X$. Given $K$ a
compact open subgroup of $\G(\AAf)$ one obtains the homeomorphic decomposition
\begin{equation} \label{e1}
\Sh_{K}(\G, X)_{\CC} = \G(\QQ) \backslash X \times \G(\AAf)/K \simeq
\coprod_{g \in \mathcal{C}} \Gamma_g \backslash X^{+} \;\;,
\end{equation}
where $\mathcal{C}$ denotes a set of representatives for the (finite)
double coset space $\G(\QQ)_{+}\backslash \G(\AAf)/K$, and $\Gamma_g$
denotes the arithmetic subgroup $gKg^{-1} \cap \G(\QQ)_{+}$ of
$\G(\QQ)_{+}$. We denote by $\Gamma_K$ the group $\Gamma_{e}$
corresponding to the identity element $e \in \mathcal{C}$
and by $S_K (\G, X)_{\CC}= \Gamma_K \backslash X^{+}$ the  
corresponding connected component
of $\Sh_{K}(\G, X)_{\CC}$.

The Shimura variety $\Sh(\G, X)_{\CC}$ is the $\CC$-scheme projective
limit of the $\Sh_K(\G, X)_{\CC}$ for $K$ ranging through compact open
subgroups of $\G(\AAf)$. The group $\G(\AAf)$ acts continuously on the
right on $\Sh(\G, X)_{\CC}$. The set of $\CC$-points of $\Sh(\G,X)_{\CC}$ is 
$$
\Sh(\G, X)_{\CC}(\CC)=
\frac{\G(\QQ)}{\Z(\QQ)} \backslash (X \times
\G(\AAf)/\overline{\Z(\QQ)})\;\;,
$$ 
 where $\Z$ denotes the centre of
$\G$ and $\overline{\Z(\QQ)}$ denotes the closure of $\Z(\QQ)$ in
$\G(\AAf)$ \cite[prop.2.1.10]{De2}. The action of $\G(\AAf)$ on the right
is given by: $\ol{(x, h)} \stackrel{.g}{\lto} \ol{(x, h\cdot g)}$. For $m
\in \G(\AAf)$, we denote by $T_m$ the Hecke
correspondence
$$ \Sh_{K}(\G, X)_{\CC} \longleftarrow \Sh(\G, X)_{\CC}
\stackrel{.m}{\lto} \Sh(\G, X)_{\CC} \lto \Sh_{K}(\G, X)_{\CC}\;\;. $$

\subsubsection{Reciprocity morphisms and canonical models.} \label{canonic}

Given $(\G,X)$ a Shimura datum, where $X$ is the $\G(\RR)$-conjugacy
class of some $h :\SS \lto \G_{\RR}$, we denote by $\mu_h : \G_{m,\CC} \lto
\G_{\CC}$ the $\CC$-morphism of $\QQ$-groups obtained by composing the
embedding of tori
$$
\begin{array}{ccc}
\G_{m,\CC} &\lto &\SS_{\CC} \\
z &\lto &(z,1)
\end{array}$$ 
with $h_\CC$.
Let $E(\G,X)$ be the field
of definition of the $\G(\CC)$-conjugacy class of $\mu_h$, it is called the reflex field of
$(\G,X)$. 
In the case where $\G$ is a torus $\TT$
and $X= \{ h\}$ we denote by
$$
r_{(\TT, \{h\})}: \Gal(\ol{\QQ}/E)^{{\rm ab}} \lto \TT(\AAf)/  
\overline{\TT(\QQ)}
$$
(where $\ol{\TT(\QQ)}$ is the closure of $\TT(\QQ)$ in $\TT(\AAf)$)
the reciprocity morphism defined
in \cite[2.2.3]{De2} for any field $E \subset \CC$ containing 
$E(\TT,\{ h\})$. Let $x=\ol{(h,g)}$ be a special point in $\Sh(\G, X)_{\CC}$ image
of the pair $(h : \SS \lto \TT \subset \G, g) \in X \times \G(\AAf)$. The
field $E(h)=E(\TT,\{h\} )$ depends only on $h$ and is an extension of  
$E(\G, X)$
\cite[2.2.1]{De2}. The Shimura variety $\Sh(\G, X)_{\CC}$ admits a unique
model $\Sh(\G, X)$ over $E(\G, X)$ such that the $\G(\AAf)$-action on  
the right is
defined over $E(\G, X)$, the special points are algebraic and if  
$x=\ol{(h,g)}$ is a
special point of $\Sh(\G, X)(\CC)$ then an element $\sigma \in
\Gal(\ol{\QQ}/E(h)) \subset \Gal(\ol{\QQ}/E(\G,X))$ acts on $x$ by
$\sigma(x)= \ol{(h, \tilde{r}(\sigma)g)}$, where $\tilde{r}(\sigma)  
\in \TT(\AAf)$
is any lift of $r_{(\TT, \{h\})}(x) \in \TT(\AAf)/
\overline{\TT(\QQ)}$, cf. \cite[2.2.5]{De2}. 
This is called the canonical model of $\Sh(\G, X)$.
For any compact open subgroup $K$ of $\G(\AAf)$, one obtains the 
canonical model for $\Sh_K(\G, X)$ over $E(\G,X)$.
For details on this definition, sketches of proofs of the existence and uniqueness and
all the relevant references we refer the reader to Chapters 12-14 of \cite{Milne}
as well as \cite{De2}.

For $m \in \G(\AAf)$ the Hecke
correspondence $T_m$ is defined over $E(\G,X)$.
We will denote by $\pi_{K}: \Sh(\G, X) \lto \Sh_{K}(\G, X)$ the natural
projection.

\subsubsection{The tower of Shimura varieties at a prime $l$.} \label{shiml}
Let $l$ be a prime. Suppose $K^l \subset \G(\AAf^{l})$ is a
compact open subgroup, where $\AAf^l$ denotes the ring of finite
ad{\`e}les outside $l$.

\begin{defi} \label{shim_l}
We denote by $\Sh_{K^{l}}(\G, X)$ the $E(\G,X)$-scheme $\varprojlim
\Sh_{K^l\cdot U_l}(\G, X)$ where $U_l$ runs over all compact open subgroups
of $\G(\QQ_l)$.
\end{defi}
The scheme $\Sh_{K^{l}}(\G, X)$ is the quotient $\Sh(\G,
X)/K^{l}$. It admits a continuous $\G(\QQ_l)$-action on the right.
Given a compact open subgroup $U_l \subset \G(\QQ_l)$ we denote by
$\pi_{U_{l}} : \Sh_{K^{l}}(\G, X) \lto \Sh_{K^{l} U_l}(\G, X)$ the
canonical projection.

\subsubsection{Neatness.} \label{neatness}
Let $\G\subset \GL_n$ be a linear algebraic group over $\QQ$. 
We recall the
definition of {\em neatness} for subgroups of $\G(\QQ)$ and its
generalization to subgroups of $\G(\AAf)$. We refer to \cite{bor} and  
\cite[0.6]{P}
for more details.

Given an element $g \in \G(\QQ)$  let ${\rm Eig}(g)$  be the subgroup  
of $\overline{\QQ}^{*}$
generated by the eigenvalues of $g$. We say that $g \in \G(\QQ)$ is  
{\em neat} if the
subgroup ${\rm Eig}(g)$ is torsion-free.  A subgroup $\Gamma \subset  
\G(\QQ)$  is neat if
any element of $\Gamma$ is neat.  In particular such a group is
torsion-free.

\begin{rem}
The notion of neatness is independent of the embedding $\G\subset \GL_n$.
\end{rem}

Given an element $g_p \in \G(\QQ_p)$ let ${\rm Eig}_p(g_p)$ be the
subgroup of $\overline{\QQ_p}^{*}$ generated by all eigenvalues of
$g_p$. Let $\overline{\QQ} \lto \overline{\QQ_p}$ be some embedding  
and consider
the torsion part $(\overline{\QQ}^{*} \cap {\rm Eig}_{p}(g_p))_{{\rm  
tors}}$. Since every subgroup of $\overline{\QQ}^{*}$
consisting of roots of unity is normalized by  
$\Gal(\overline{\QQ}/\QQ)$, this group does not depend on the choice of
the embedding $\overline{\QQ} \lto \overline{\QQ_p}^{*}$. We
say that $g_p$ is {\em neat} if $$(\overline{\QQ}^{*} \cap {\rm
Eig}_{p}(g_p))_{{\rm tors}} =\{ 1\}\;\;.$$
We say that $g=(g_{p})_{p} \in \G(\AAf)$ is neat if
$$\bigcap_p (\overline{\QQ}^{*} \cap {\rm Eig}_{p}(g_{p}))_{{\rm
tors}} =\{ 1\}\;\;.$$ A subgroup $K \subset \G(\AAf)$ is neat
if any element of $K$ is neat. Of course if the
projection $K_p$ of $K$ in $\G(\QQ_p)$ is neat then $K$ is neat.
Notice that if $K$ is a neat compact open subgroup of $\G(\AAf)$ then  
all of the $\Gamma_g$ in the decomposition~(\ref{e1}) are.

Neatness is preserved by conjugacy and intersection
with an arbitrary subgroup. Moreover if $\rho : \G \lto \HH$ is a  
$\QQ$-morphism of linear algebraic $\QQ$-groups
and $g \in \G(\QQ)$ (resp. $\G(\AAf)$) is neat then its image  
$\rho(g)$ is also neat.

We recall the following well-known lemma:
\begin{lem} \label{finiteneat}
Let $K = \prod_p K_p$ be a compact open subgroup of $\G(\AAf)$ and let
$l$ be a prime number. There exists an open subgroup $K'_l$ of $K_l$
such that the subgroup $K':= K'_l \times \prod_{p \not = l}K_p$ of $K$
is neat.
\end{lem}

\begin{proof}
As noticed above if $K'_l$ is neat then $K':= K'_l \times \prod_{p
   \not = l}K_p$ is neat. As a subgroup of a neat group is neat, it is
   enough to show that a special maximal compact open subgroup $K_l
   \subset \G(\QQ_{l})$
   contains a neat subgroup $K'_l$ with finite index. By \cite[p.5]{P} one
   can take, $K'_l = K_{l}^{(1)}$ the first congruence kernel.
\end{proof}

\subsubsection{Integral structures} \label{integral}
Let $(\G, X)$ be a Shimura datum and $K \subset \G(\AAf)$ a neat
compact open subgroup. 
We can fix a $\ZZ$-structure on $\G$ and its subgroups by choosing a
finitely generated free $\ZZ$-module $W$, a
faithful representation $\xi \colon \G \into \GL(W_{\QQ})$ and taking the
Zariski closures in the $\ZZ$-group-scheme $\GL(W)$. If we choose the  
representation $\xi$ in such a way that $K$ is contained in $\GL(\Zhat \otimes_{\ZZ} W)$  
(i.e. $K$ stabilizes $\Zhat \otimes_{\ZZ} W$) and $\xi$ factors through $\G^{\ad}$,
this induces canonically a $\ZZ$-variation of Hodge structure on $\Sh_K(\G,X)_{\CC}$:
cf. \cite[section 3.2]{EdYa}. If $K = \prod_{p \; \textnormal{prime}}
K_p$ then for almost all primes $l$ the group $K_l$ is a hyperspecial
maximal compact open subgroup of $\G(\QQ_l)$ which coincides with
$\G(\ZZ_l)$.

\subsubsection{Good position with respect to a torus} \label{good_position}

\begin{defi}
Let $l$ be a prime number, $\G$ a reductive $\QQ_l$-group and $\TT \subset
\G$  a split torus. A compact open subgroup $U_l$ of $\G(\QQ_l)$ is
said to be in good position with respect to $\TT$ if $U_l \cap
\TT(\QQ_l)$ is the maximal compact open subgroup of $\TT(\QQ_l)$.

If $\G$ is a reductive $\QQ$-group, $\TT \subset
\G$ a torus and $l$ a prime number splitting $\TT$, we say that a compact open subgroup $U_l$ of $\G(\QQ_l)$ is
in good position with respect to $\TT$ if it is in
good position with respect to $\TT_{\QQ_{l}}$.
\end{defi}

\begin{lem} \label{splitgood}
Suppose that $(\G, X)$ is a Shimura datum, 
$K = \prod_{p \; \textnormal{prime}} K_p$ is a neat open compact
subgroup of $\G(\AAf)$ and  $\rho: \G \hookrightarrow 
\GL_n$ is a faithful rational representation such that $K$ is
contained in $\GL_n(\Zhat)$. Let $\TT \subset \G$ be a torus and $l$ be a prime number
splitting $\TT$ such that $\TT_{\FF_l}$ is an $\FF_l$-torus. Then the group
$\G(\ZZ_l)$ is in good position with respect to $\TT$. 
\end{lem}

\begin{proof}
Let $\TT'$ be the scheme-theoretic closure of $\TT$ in ${(\GL_{n})}_{\ZZ_l}$. 
The scheme $\TT'$ is a a flat group scheme affine and of finite type
over $\ZZ_l$ whose fibers $\TT_{\FF_l}$ over $\FF_l$ and $\TT_{\QQ_l}$ over
$\QQ_l$ are tori. Hence by \cite[Exp.X, cor.4.9]{SGA3} the group
scheme $\TT'$ is a torus over $\ZZ_l$. As its generic fiber
$\TT_{\QQ_l}$ is split, $\TT'$ is split by \cite[Exp.X,
cor.1.2]{SGA3}. Hence $\G(\ZZ_l) \cap \TT(\QQ_l) = \TT'(\ZZ_l)$ is a
maximal compact subgroup of $\TT(\QQ_l) = \TT'(\QQ_l)$ and the result follows.
\end{proof}

\subsection{$p$-adic closure of Zariski-dense groups.} 

We will use the following well-known result (we provide a proof for completeness):
\begin{prop} \label{approx}
Let $H$ be a  subgroup of $\GL_n(\ZZ)$ and let $\HH$ be the
Zariski closure of $H$ in $\GL_{n, \ZZ}$. Suppose that $\HH_{\QQ}^0$ is
semisimple. Then for any prime number $p$ the closure of $H$ in 
$\HH(\ZZ_p)$ is open. 
\end{prop}

\begin{proof}
The case when $H$ is finite is obvious. Suppose that $H$ is infinite. Since
$\HH(\ZZ_p)$ is compact and $H$ is infinite, the closure $H_p$ of $H$
in $\HH(\ZZ_p)$ is not discrete. Then it is a $p$-adic analytic group
and it has a Lie algebra $L$ which is a Lie subalgebra of
the Lie algebra $\textnormal{Lie} \; \HH$ of $\HH$ and projects
non-trivially on any factor of $\textnormal{Lie} \; \HH$. By construction
$L$ is invariant under the adjoint action of $H$,
thus also under the adjoint action of the Zariski closure $\HH$ of
$H$. As $\HH_{\QQ}^0$ is semisimple one deduces $L_\QQ =
\textnormal{Lie} \; \HH_\QQ$, which implies that $H_p$ 
is open in $\HH(\ZZ_p)$.
\end{proof}

\rem The easy proposition~\ref{approx} can be strengthened to the following
remarkable theorem, due independently to Weisfeiler and Nori, which
was used in \cite{EdYa} but which we will not need:

\begin{theor}[\cite{weis}, \cite{Nori}] \label{Wei-Nori}
Let $H$ be a finitely generated subgroup of $\GL_n(\ZZ)$ and let $\HH$ be the
Zariski closure of $H$ in $\GL_{n, \ZZ}$. Suppose that $\HH(\CC)$ has  finite
fundamental group. Then the closure of $H$ in $\GL_n(\AAf)$ is open in
the closure of $\HH(\ZZ)$ in $\GL_n(\AAf)$.
\end{theor}

\section{Degrees on Shimura varieties.} \label{degree}
In this section we recall the results we will need on projective
geometry of Shimura varieties and prove the crucial corollary~\ref{compare-degrees}  
which compares the degrees of
a subvariety of $\Sh_K(\G,X)$ with respect to two different line bundles.

\subsection{Degrees.} \label{degrees1}
We will need only basics on numerical intersection theory as recalled in
\cite[chap.1, p.15-17]{laz}. Let $X$ be a complete irreducible complex variety and $L$ a
line bundle on $X$ with topological first Chern class $c_1(L) \in
H^2(X, \ZZ)$. Given $V \subset X$ an irreducible
subvariety we define the degree of $V$ with respect to $L$ by 
$$\deg_L V= c_1(L)^{\dim V} \cap [V] \in H_0(X,\ZZ)=\ZZ \;\;,$$ where $[V] \in H_{2\dim
  V}(X, \ZZ)$ denotes the fundamental class of $V$ and $\cap$ denotes
the cap product between $H^{2\dim V}(X, \ZZ)$ and $H_{2 \dim V}(X,
\ZZ)$. We also write $\deg_L V= \int_V c_1(L)^{\dim V}$. It satisfies
the projection formula: given $f: Y \lto X$ a generically finite
surjective proper map one has
$$\deg_{f^*L} Y = (\deg f) \deg_L X\;\;.$$

When the  subvariety $V$ is not irreducible, let $V = \cup_i V_i$ be its decomposition into irreducible components.
We define
$$
\deg_L V = \sum_i \deg_L V_i\;\;.
$$

When the variety $X$ is a disjoint union of irreducible components
$X_i$, $1 \leq i \leq n$, the function
$\deg_L$ is defined as the sum $\sum_{i=1}^n \deg_{L_{|X_{i}}}$.

\subsection{Nefness.}
Recall (cf. \cite[def. 1.4.1]{laz}) that a line bundle $L$ on a
complete scheme $X$ is said to be {\em nef} if $\deg_L C \geq 0$ for
every irreducible curve $C \subset X$. We will need the following
basic result (cf. \cite[theor.1.4.9]{laz}):
\begin{theor}[Kleiman] \label{kleiman}
Let $L$ be a line bundle on a complete complex scheme $X$. Then $L$ is
nef if and only if for every irreducible subvariety $V \subset X$ one
has $\deg_L V \geq 0$.
\end{theor}

\subsection{Baily-Borel compactification.} 

\begin{defi} 
Let $(\G, X)$ be a Shimura datum and $K \subset \G(\AAf)$ a neat
compact open subgroup.
We denote by $\overline{\Sh_K(\G,X)_{\CC}}$ the Baily-Borel
compactification of $\Sh_K(\G,X)_{\CC}$, cf. \cite{BB}.
\end{defi}

The Baily-Borel compactification $\overline{\Sh_K(\G,X)_{\CC}}$ is a  
normal projective variety.
Its boundary $\overline{\Sh_K(\G,X)_{\CC}} \setminus \Sh_K(\G,X)_{\CC}$ has 
complex codimension $>1$ if and only if $\G$ has no
split $\QQ$-simple factors of dimension~$3$. The following  
proposition summarizes basic properties of
$\overline{\Sh_K(\G,X)_{\CC}}$ that we will use.

\begin{prop} \label{deg}
\begin{enumerate}
\item The line bundle of holomorphic forms of maximal degree on $X$  
descends to $\Sh_K(\G,X)_{\CC}$
and extends uniquely to an ample line bundle $L_{K}$ on  
$\overline{\Sh_K(\G,X)_{\CC}}$ such that, at the
generic points of the boundary components of codimension one, it is
given by forms with logarithmic poles. Let $K_1$ and
$K_2$ be neat compact open subgroups of $\G(\AAf)$ and $g$ in $\G(\AAf)$
such that $K_2 \subset gK_1g^{-1}$. Then the morphism from
$\Sh_{K_{2}}(\G, X)_{\CC}$ to $\Sh_{K_{1}}(\G, X)_{\CC}$ induced by  
$g$ extends to a morphism $f:
\overline{\Sh_{K_{2}}(\G, X)_{\CC}} \lto \overline{\Sh_{K_{1}}(\G,  
X)_{\CC}}$, and the line bundle
$f^{*}L_{K_{1}}$ is canonically isomorphic to $L_{K_{2}}$ .

\item The canonical model $\Sh_K(\G,X)$ of $\Sh_K(\G,X)_{\CC}$  over  
the reflex field $E(\G, X)$
admits a unique extension to a model $\overline{\Sh_K(\G,X)}$ of  
$\overline{\Sh_K(\G,X)_{\CC}}$ over $E(\G, X)$.
The line bundle $L_K$ is naturally defined over $E(\G, X)$.

\item Let $\varphi : (\HH, Y) \lto (\G, X)$ be a morphism of Shimura  
data and $K_{\HH} \subset \HH(\AAf)$, $K_{\G} \subset
\G(\AAf)$ neat compact open subgroups with $\varphi(K_{\HH}) \subset  
K_{\G}$. Then the canonical map
$\phi : \Sh_{K_{\HH}}(\HH, Y) \lto \Sh_{K_{\G}}(\G, X)$ induced by  
$\varphi$ extends to a morphism still denoted by
$\phi : \overline{\Sh_{K_{\HH}}(\HH, Y)} \lto \overline{\Sh_{K_{\G}}(\G, X)}$.
\end{enumerate}
\end{prop}

\begin{proof}
The first statement is \cite[lemma 10.8]{BB} and \cite[prop.8.1,
sections 8.2, 8.3]{P}. The second one is \cite[theor.12.3.a]{P}. The
third statement is \cite[theorem p.231]{Sat1} (over $\CC$) and  
\cite[theor. 12.3.b]{P} (over $E(\G, X)$).
\end{proof}

\begin{defi} \label{defdegree}
Given a complex subvariety $Z \subset \Sh_{K}(\G, X)_{\CC}$ we
will denote by $\deg_{L_{K}}Z$ the degree of the compactification
$\overline{Z} \subset \overline{\Sh_{K}(\G, X)_{\CC}}$  with respect  
to the line bundle $L_{K}$. We will write $\deg Z$ when it is clear to
which line   
bundle we are referring to.
\end{defi}

\begin{rem} \label{compconn}
Let $\G$ be a connected semisimple algebraic
$\QQ$-group of Hermitian type (and of non-compact type) with  
associated Hermitian domain $X$. Recall that a subgroup
$\Gamma \subset \G(\QQ)$ is called an arithmetic lattice if $\Ga$ is
commensurable to $\G(\QQ) \cap \GL_n(\ZZ)$, where we fixed a faithful
$\QQ$-representation $\xi: \G \hookrightarrow \GL_n$. This definition is
independent of the choice of $\xi$. If $\Gamma \subset \G(\QQ)$ is a neat arithmetic
lattice the quotient $\Gamma \backslash X$ is a smooth
quasi-projective variety, which is projective if and only if $\G$ is
$\QQ$-anisotropic (cf. \cite{bor}). The Baily-Borel
compactification $\overline{\Gamma \backslash X}$ of the quasi-projective
complex variety $\Gamma \backslash X$ and the bundle $L_{\Gamma}$ on
$\overline{\Gamma \backslash X}$ are well-defined (cf. \cite{BB}).
\end{rem}

\subsubsection{Comparison of degrees for Shimura subdata.}

\begin{prop} \label{compare-degrees0}
Let $\phi:\Sh_{K}(\G, X)_{\CC} \lto \Sh_{K'}(\G', X')_{\CC}$ be a morphism of
Shimura varieties associated to a Shimura subdatum $\varphi :(\G, X)  
\lto (\G',
X')$, a neat compact open subgroup $K$ of $\G(\AAf)$ and a neat compact open
subgroup $K'$ of $\G'(\AAf)$ containing $\varphi(K)$. Then the line
bundle $$\Lambda_{K,
   K'}:= \phi^{*}L_{K'} \otimes L_{K}^{-1}$$ on $\overline{\Sh_{K}(\G, X)_{\CC}}$ is nef.
\end{prop}

This proposition is a corollary of the following
\begin{prop} \label{cop-deg}
Let $\varphi : \G \lto \G'$ be a $\QQ$-morphism of connected  
semisimple algebraic
$\QQ$-groups of Hermitian type (and of non-compact type) inducing a
holomorphic totally geodesic embedding of the associated Hermitian
domains $\phi : X^+ \lto X'^{+}$. Let $\Gamma \subset \G(\QQ)$ be a neat
arithmetic lattice and $\Gamma' \subset \G'(\QQ)$ a neat arithmetic
lattice containing $\varphi(\Gamma)$. Then the line bundle 
$$\Lambda_{\Ga, \Ga'}:=\phi^*L_{\Gamma'} \otimes L_{\Gamma}^{-1}$$ on
$\overline{\Gamma \backslash X^+}$ is nef.
\end{prop}

\sspace
\noindent
\begin{proof}[Proposition~\ref{cop-deg} implies the
proposition~\ref{compare-degrees0}]
Let $C \subset \overline{\Sh_{K}(\G, X)_{\CC}}$ be an irreducible curve. To prove
that $\deg_{\Lambda_{K,K'}} C \geq 0$ one can assume without
loss of generality that $C$ is
contained in the connected component $\overline{S_K} = \overline{\Gamma_{K} \backslash X^+}$
and that $\phi:\overline{\Sh_{K}(\G, X)_{\CC}} \lto \overline{\Sh_{K'}(\G', X')_{\CC}}$ maps
$\overline{S_K}$ to $\overline{S_{K'}}= \overline{\Gamma_{K'} \backslash X'^{+}}$. The
morphism of reductive $\QQ$-groups $\varphi : \G \lto \G'$ induces a
$\QQ$-morphism $\overline{\varphi}: \G^\der \lto \G'^{\ad}$ of
semisimple $\QQ$-groups. Let $\Gamma$ denote the neat lattice  
$\G^\der (\QQ) \cap K \subset
\G^\der(\QQ)$ and $\Gamma'$ the neat lattice of $\G^{\ad}(\QQ)$
image of $\Gamma_{K'}$. Notice that $\Gamma' \backslash X'^{+} =
\Gamma_{K'} \backslash X'^{+}$.  Consider the diagram
\begin{equation}
\xymatrix{
\overline{\Gamma \backslash X^+}  \ar[dr]^{\phi \circ \pi}  \ar[d]^{\pi}
\\
\overline{\Gamma_{K} \backslash X^+}   \ar[r]_{\phi}
&\overline{\Gamma' \backslash X'^{+}}}
\end{equation}
with $\pi$ the natural finite  map.
The proposition~\ref{deg} $(1)$ extends to this
setting:
$$ \pi^{*} (L_{\Gamma_{K}}) = L_{\Gamma}\;\;.$$
Thus $$\pi^* \Lambda_{K,K'} = \Lambda_{\Gamma, \Gamma'}\;\;.$$
Let $d$ denote the degree of $\pi$. 
By the projection formula one obtains:
$$
\deg_{\Lambda_{K,K'}}C = \frac{1}{d} \deg_{\Lambda_{\Ga,\Ga'}}\pi^{-1}(C)\;\;.$$
Now $\deg_{\Lambda_{\Ga,\Ga'}}\pi^{-1}(C) \geq 0$ by proposition~\ref{cop-deg}.

\end{proof}

\begin{proof}[Proof of the proposition~\ref{cop-deg}]
Let $C \subset \overline{\Gamma \backslash X^+}$ be an irreducible
curve. We want to show that $\deg_{\Lambda_{\Ga, \Ga'}} C \geq
0$. First notice that by the projection formula and by   
proposition~\ref{deg} $(1)$, we can
assume that the group $\G$ is simply connected and the group $\G'$ is adjoint.

Let $\G = \G_1\times \cdots \times \G_r$ be the decomposition of $\G$
into $\QQ$-simple factors. Let $\varphi_{i} : \G_i \lto \G'$, $1 \leq i
\leq r$ denote the components of $\varphi: \G \lto \G'$. If $\Gamma_1
\subset \Gamma$ is a finite index subgroup and $p :\overline{ \Gamma_{1}
\backslash X^+} \lto \overline{\Gamma \backslash X^+}$ is
the corresponding finite morphism, by proposition~\ref{deg}  
the line bundle
$\Lambda_{\Gamma_{1}, \Gamma'}$ corresponding to $\phi \circ p$ is  
isomorphic to
$p^*\Lambda_{\Gamma, \Gamma'}$. The fact that
$\deg_{\Lambda_{\Gamma, \Gamma'}}C \geq 0$ is once more implied by
$\deg_{\Lambda_{\Gamma_{1}, \Gamma'}} p^{-1}(C) \geq 0$. Thus we can
   assume that $\Gamma = \Ga_1 \times \cdots \times
\Ga_r$, with $\Ga_i$ a neat arithmetic subgroup of $\G_i(\QQ)$. The
variety $\overline{\Gamma \backslash X^+}$ decomposes into a product
$$
  \overline{\Gamma \backslash X^+} = \overline{\Gamma_1 \backslash
    X_{1}^{+}} \times \cdots \times 
\overline{\Gamma_{r} \backslash X_{r}^{+}}
$$
and the line bundle $\Lambda_{\Gamma, \Gamma'}$ on $\overline{\Gamma
   \backslash X^+}$ decomposes as
$$ \Lambda_{\Gamma, \Gamma'} = \Lambda_{\Gamma_{1}, \Gamma'} \boxtimes
\cdots \boxtimes \Lambda_{\Gamma_{r} , \Gamma'}\;\;,$$
with $\Lambda_{\Gamma_{i}, \Ga'}= \phi_{i}^{*}L_{\Gamma'} \otimes  
L_{\Gamma_{i} }^{-1}$ the
corresponding line bundle on $\overline{\Gamma_{i} \backslash X_{i}^{+}}$.
Let $p_{i} : \overline{ \Gamma \backslash X^+ } \lto
\overline{\Gamma_{i} \backslash X_{i}^{+}}$ be the natural projection.
As $$\deg_{\Lambda_{\Gamma, \Ga'}} C = \sum_{i=1}^{r} \deg_{p_{i}^{*}
     \Lambda_{\Ga_{i}, \Ga'}}
C\;\;,$$ we have reduced the proof of the proposition to the case
where $\G$ is $\QQ$-simple. It then follows from the more precise
following proposition~\ref{positivite}.
\end{proof}

\begin{prop} \label{positivite}
Assume that $\G$ is $\QQ$-simple.
\begin{enumerate}
\item If $\G$ is $\QQ$-anisotropic then the line bundle $\Lambda_{\Ga,
   \Ga'}$ on the smooth complex projective variety $\Ga \backslash X^+$ admits
a metric of non negative curvature.
\item If $\G$ is $\QQ$-isotropic then either the line bundle
   $\Lambda_{\Ga, \Ga'}$ on $\overline{\Ga \backslash X^+}$ is trivial or it
   is ample.
\end{enumerate}
\end{prop}

\begin{proof}
Let $\G' = \G'_1 \times \cdots \times \G'_{r'}$ be the decomposition of $\G'$
into $\QQ$-simple factors and $\varphi_{j}: \G \lto \G'_{j}$, $1 \leq j
\leq r'$, the components of $\varphi : \G \lto \G'$. By naturality of  
$L_{\Ga}$
and $L_{\Ga'}$ (cf. proposition~\ref{deg}) one can assume that  
$\Ga' = \Ga'_1 \times
\cdots \Ga'_{r'}$. Accordingly one has
$$ \Ga' \backslash X'^{+} = \Ga'_{1} \backslash {X'_{1}}^{+} \times  
\cdots \times
\Ga'_{r'} \backslash {X'_{r'}}^{+} \;\;.$$
As $\varphi : \G \lto \G'$ is injective and $\G$ is $\QQ$-simple we  
can without loss of
generality assume that $\varphi_1 : \G \lto \G'_{1}$ is injective.
As $$\Lambda = (\phi_{1}^{*}L_{\Ga'_{1}} \otimes L_{\Ga}^{-1}) \otimes
\phi_{2}^{*}L_{\Ga'_{2}} \otimes \cdots \phi_{r'}^{*}L_{\Ga'_{r}} \;\;,$$
and the $L_{\Ga'_{j}}$, $j \geq 2$, are ample on
$\overline{\Ga'_{j} \backslash {X'_{j}}^{+}}$ it is enough to prove the
statement replacing $\Lambda_{\Ga, \Ga'}$ by
$\phi_{1}^{*}L_{\Ga'_{1}}\otimes L_{\Ga}^{-1}$. Thus we
can assume $\G'$ is $\QQ$-simple.

By the adjunction formula the line bundle
${\Lambda_{\Ga, \Ga'}}_{|\Gamma \backslash X^+}$ restriction of
$\Lambda_{\Ga, \Ga'}$ coincides with   
$\Lambda^{\textnormal{max}}N^{*}$, where $N$
denotes the automorphic bundle on $\Gamma \backslash X^+$ associated to
the normal bundle of $X$ in $X'$ and $N^{*}$ denotes its dual. As $X$
is totally geodesic in $X'$ the curvature form on $N$ is the
restriction to $N$ of the curvature form on $TX'$. As $X'$ is
non-positively curved, the automorphic bundle $N^*$ and thus also the automorphic line bundle
${\Lambda_{\Ga, \Ga'}}_{|\Gamma \backslash X^+}$ admits a Hermitian metric of
non-negative curvature. This concludes the proof of the
proposition in the case $\G$ is
$\QQ$-anisotropic.

Suppose now $\G$ is $\QQ$-isotropic. For simplicity we denote  
$\Lambda_{\Ga, \Ga'}$
by $\Lambda$ from now on. We have to
prove that the boundary components of $\overline{\Gamma \backslash X^+}$ do
not essentially modify the positivity of $\Lambda_{|\Gamma \backslash X^+}$.
We use the notation and the results of Dynkin \cite{Dy}, Ihara
\cite{ih} and Satake \cite{Sat2}. Let
$X=X_1 \times \cdots \times X_r$ (resp. $X' = X'_1
\times \cdots \times X'_{r'}$) be the decomposition of $X$
(resp. $X'$) into irreducible factors. Each $X_i$ (resp. $X'_j$) is
the Hermitian symmetric domain associated to an $\RR$-isotropic
$\RR$-simple factor $\G_i$ (resp. $\G'_j$) of $\G_{\RR}$
(resp. $\G'_{\RR}$). The group $\G_{\RR}$ (resp. $\G'_{\RR}$)  
decomposes as $\G_0 \times
\G_1 \times \cdots \times \G_r$ (resp. $\G'_0 \times \G'_1 \times
\cdots \times \G'_{r'}$) with $\G_0$ (resp. $\G'_0$) an $\RR$-anisotropic
group. Let $\mathbf{m}$ (resp. $\mathbf{m'}$) be the $r$-tuple (resp.  
$r'$-tuple)
of non-negative integers defining the automorphic line bundle $L_{K}$
(resp. $L_{K'}$) (cf. \cite[lemma 2]{Sat2}) and $M_{\phi}$ be the
$r' \times r$-matrix with integral coefficients associated to $\varphi  
: \G \hookrightarrow
\G'$ (cf. \cite[section 2.1]{Sat2}). The automorphic line bundle
${\Lambda}_{|\Gamma \backslash X^+}$ on $\Gamma \backslash X^+$ is
associated to the $r$-tuple of integers $\boldsymbol{\lambda}  
=\mathbf{m'} M_{\varphi} -
\mathbf{m}$ (where $\mathbf{m}$ and $\mathbf{m'}$ are seen as row
vectors). It admits a locally homogeneous Hermitian metric of non-negative
curvature if and only if $\lambda_i \geq 0$, $1 \leq i \leq
r$ (in which case we say that $\boldsymbol{\lambda}$ is
non-negative).

\begin{lem}
The row vector $\boldsymbol{\lambda}$ is non-negative.
\end{lem}

\begin{proof}
As $\G$ and $\G'$ are defined over $\QQ$, both $\mathbf{m}$ and
$\mathbf{m'}$ are of rational type by \cite[p.301]{Sat2}.
So $m_i= m$ for all $i$, $m'_j =m'$ for all $j$.
The equality $\boldsymbol{\lambda} =\mathbf{m'} M_{\varphi} -
\mathbf{m}$ can be written in coordinates
\begin{equation} \label{mequ}
\forall i \in \{1, \cdots, r\}, \;\;\lambda_i = \sum_{1\leq j \leq r'}
m_{j,i} \, m' -m\;\;,
\end{equation}
with $M_{\varphi}= (m_{j,i})$.
Fix $i$ in $\{1, \cdots r\}$ and let us prove that $\lambda_i \geq 0$. As the
$m_{i,j}$'s and $m'$ are non-negative, it is enough to exhibit one
$j$, $1 \leq j \leq r'$, with $m_{j,i} \, m' -m \geq 0$.
Choose $j$ such that the component $\varphi_{i,j} : X_i \lto X'_j$ of the
map $\varphi : X_1 \times \cdots \times X_r \lto X'_1
\times \cdots \times X'_{r'}$ induced by $\varphi : \G \lto \G'$ is an
embedding. Recall that with the notation
of \cite[p.290]{Sat2} one has $$m_i = < H_{1,i}, H_{1,i}>_{i}
\;\;,$$
where $\mathfrak{h}_{i}$ denotes the chosen Cartan subalgebra of
$\mathfrak{g}_{i}(\RR)$ and $<,>_{i}$ denotes the canonical scalar
product on $\sqrt{-1}\mathfrak{h}_{i}$.
This gives the equality:
\begin{equation} \label{posit}
m_{j,i} \, m'_{j} -m_{i}  =
<\phi_{j}(H_{1,i}), \phi_{j}(H_{1,i})>_{j} - < H_{1,i}, H_{1,i}>_{i}\;\;.
\end{equation}

As $\G_i$ is $\RR$-simple, any two invariant non-degenerate forms on
$\sqrt{-1}\mathfrak{h}_{i}$ are proportional:
there exists a positive real constant $c_{i,j}$ (called by Dynkin
\cite[p.130]{Dy} the index of $\varphi_{i, j} : \G_i \lto \G_j$) such that
$$ \forall  X, Y  \in \sqrt{-1}\mathfrak{h}_{i}, \;\; <\phi_{j}(X),
\phi_{j}(Y)>_{j}= c_{i,j} < X, Y>_{i}\;\;.$$
Equation~(\ref{posit}) thus gives:
\begin{equation} \label{mequ1}
m_{j,i} \, m'_{j} -m_{i}  =
(c_{i,j} -1) < H_{1,i}, H_{1,i}>_{i}\;\;.
\end{equation}
By \cite[theorem 2.2. p.131]{Dy} the constant $c_{i,j}$ is a positive
integer. Thus  $m_{j,i} \, m'_{j} -m_{i}$ is non-negative and
this finishes the proof that $\boldsymbol{\lambda}$ is non-negative.
\end{proof}

By \cite[cor.2 p.298]{Sat2} the sum $M = \sum_{1\leq j \leq r'} m_{j,i}$
is independent of $i$ ($1 \leq i \leq r$). This implies that
$\boldsymbol{\lambda}$ is of rational type: one of the $\lambda_i$ is
non-zero if and only if all are. In this case $\boldsymbol{\lambda}$
is positive of rational type and $\Lambda$ is ample on $\overline{\Ga
   \backslash X^{+}}$ by \cite[theor.1]{Sat2}.

If $\boldsymbol{\lambda}=0$, the line bundle $\Lambda_{|\Ga
   \backslash X^{+}}$ is trivial. As $\G$ is $\QQ$-simple, if $\G$ is not
   locally isomorphic to $\mathbf{SL}_2$ the line bundle $\Lambda$ on  
$\overline{\Ga
   \backslash X^{+}}$ is trivial.

The last case is treated in the following lemma:

\begin{lem}
If $\boldsymbol{\lambda}=0$ and $\G$ is locally isomorphic to
  $\mathbf{SL}_2$, then $\phi : \G \lto \G'$ is a local isomorphism and the
  line bundle $\Lambda$ on $\overline{\Ga
   \backslash X^{+}}$ is trivial.
\end{lem}

\begin{proof}
It follows from the equation~(\ref{mequ}) that there exists a unique
integer $j$ such that the morphism $\varphi_{j} : \G_{\RR} \lto \G_j$
is non trivial. In particular $\G'$ is $\RR$-simple. Moreover the
equation~(\ref{mequ1}) implies that
index $c$ of $\phi: \G \lto \G'$ is equal to $1$. Thus by
\cite[theorem 6.2 p.152]{Dy} the Lie algebra $\mathfrak{g}$ is a
regular subalgebra of $\mathfrak{g'}$. If $\G'_{\RR}$ is
classical, the equality \cite[(2.36) p.136]{Dy} shows that necessarily
$\phi : \G \lto \G'$ is a local isomorphism. In particular the
line bundle $\Lambda$ on $\overline{\Ga \backslash X^{+}}$ is trivial.
If the group $\G'_{\RR}$ is an exceptional simple Lie group of
Hermitian type (thus $E_6$ or $E_7$), Dynkin shows in \cite[Tables 16,
   17 p.178-179]{Dy} that there is a unique realization of
$\mathfrak{g}$ as a regular subalgebra of $\mathfrak{g'}$ of index
$1$. However this realization is not of Hermitian type: the
coefficient $\alpha'_{1}(\varphi(H_{1}))$ is zero. Thus this case is
impossible.
\end{proof}

This finishes the proof of proposition~\ref{positivite}.

\end{proof}

From the nefness of $\Lambda_{K,K'}$ we now deduce the following
crucial corollary:

\begin{cor} \label{compare-degrees}
Let $\phi:\Sh_{K}(\G, X) \lto \Sh_{K'}(\G', X')$ be a morphism of
Shimura varieties associated to a Shimura subdatum $\varphi :(\G, X)  
\lto (\G',
X')$. Assume that $\Z(\RR)$ is compact (where $\Z$ denotes the centre
of $\G$). Let $K'$ a neat compact open subgroup of $\G'(\AAf)$ and denote
by $K$ the compact open subgroup $K' \cap
\G(\AAf)$ of $\G(\AA_f)$. Then for any irreducible Hodge generic subvariety $Z$ of
$\Sh_{K}(\G,X)$ one has $\deg_{L_{K}}Z \leq  
\deg_{L_{K'}}\phi(Z)$.
\end{cor}

\begin{proof}
As the irreducible components of $Z$ are Hodge generic in $\Sh_{K}(\G,
X)$ and as $\Z(\RR)$ is compact we know
by lemma~$2.2$ in \cite{UllmoYafaev} (and its proof) that $\phi_{|Z} :  
Z \lto Z':= \phi(Z)$ is generically
injective. In particular by the projection formula one has
$$ \deg_{L_{K'}}Z' = \deg_{\phi^{*}L_{K'}} Z \;\;.$$
So the inequality $\deg_{L_{K}}Z \leq \deg_{L_{K'}}Z'$ is equivalent
to the inequality $\deg_{\phi^* L_{K'}}Z \geq \deg_{L_{K}} Z$. 

As $\phi^* L_{K'} = L_K \otimes \Lambda_{K,K'}$ one has
$$ \deg_{\phi^* L_{K'}}Z = \sum_{i=0}^{\dim Z} \begin{pmatrix} \dim Z\\
  i \end{pmatrix}  \int_Z c_1(L_K)^i \wedge
c_1(\Lambda_{K,K'})^{\dim Z -i} \;\;.$$
The inequality $\deg_{\phi^* L_K'}Z \geq \deg_{L_K} Z$ thus follows if
we show:
$$ \forall i \;\;,  0 \leq i \leq \dim Z-1, \quad \int_{Z}
c_1(L_K)^i \wedge c_1(\Lambda_{K,K'})^{\dim Z -i} \geq 0\;\;.$$
As $L_K$ is ample it follows from the nefness of $\Lambda_{K,K'}$ and
Kleiman's theorem~\ref{kleiman}.  
\end{proof}


\section{Inclusion of Shimura subdata.} \label{partb}

In this section we prove a proposition which implies
part~$(b)$ of the theorem~\ref{main-thm1}.
We also prove two auxiliary lemmas on inclusion of Shimura data.

\begin{lemA} \label{splitting}
Let $(\HH,X_{\HH})\subset (\HH', X_{\HH'})$ be an inclusion of Shimura data.
We assume that $\HH$ and $\HH'$ are the generic Mumford-Tate groups on $X_{\HH}$ and
$X_{\HH'}$ respectively.
Suppose that the connected centre $\TT$ of $\HH$ is split by a number field $L$.
Then the connected centre $\TT'$ of $\HH'$ is split by $L$.
\end{lemA}
\begin{proof}
Let $\C':=\HH'/{\HH'}^{\der}$. Then there is an isogeny between $\TT'$ and $\C'$
induced by the quotient
$\pi'\colon \HH' \lto \C'$. The splitting fields of $\TT'$ and $\C'$ are therefore the same.

We claim that for any $\alpha \in X_{\HH'}$, the Mumford-Tate group of $\pi'\alpha$ is $\C'$.
Indeed, as $\C'$ is commutative, and $X_{\HH'}$ is an $\HH'(\RR)$-conjugacy class, 
$\pi'\alpha$ does not depend on $\alpha$. 
Let $\alpha\in X_{\HH'}$ be Hodge generic and let $\C_1$ be the Mumford-Tate group
of $\pi'\alpha$. Then $\alpha$ factors through ${\pi'}^{-1}(\C_1) = \HH'$.
It follows that $\C_1 = \C'$.

Let $\beta$ be a Hodge generic point of $X_{\HH}$.
As $\HH = \TT\HH^{\der}$ and $\HH^{\der}\subset {\HH'}^{\der}$, we have
$$
\pi'(\HH) = \pi'(\TT)\;\;.
$$
As $\pi'(\HH)$ is the Mumford-Tate group of $\pi'\beta$ (because $\HH$ is the Mumford-Tate group of $\beta$), we see that
$$
\pi'(\TT)=\C'\;\;.
$$
As the torus $\TT$ is split by $L$, the torus $\C'$ and therefore also
the torus $\TT'$ are split by $L$.
\end{proof}

\begin{lemA} \label{inclusion1}
Let $(\HH,X_{\HH})\subset (\HH', X_{\HH'})$ be an inclusion of Shimura data.
We assume that $\HH$ and $\HH'$ are the generic Mumford-Tate groups on $X_{\HH}$ and
$X_{\HH'}$ respectively.
Let $\TT$ and $\TT'$ be the connected centres of $\HH$ and $\HH'$ respectively.

Suppose that $\TT\subset \TT'$. Then 
$
\TT = \TT'
$.
\end{lemA}
\begin{proof}
We write
$$
\HH' = \TT' {\HH'}^\der\;\;.
$$
We have $(\TT'\cap \HH)^0 \subset \TT$. On the other hand, by assumption
$\TT \subset (\TT'\cap \HH)^0$, hence
$$
\TT = (\TT'\cap \HH)^0\;\;.
$$

Write
$$
\HH = (\TT'\cap \HH)^0 \HH^{\der}\;\;.
$$

Fix $\alpha$ an element of $X_{\HH}$.
As $X_{\HH'}$ is the $\HH'(\RR)$-conjugacy class of $\alpha$, any element $x \in
X_{\HH'}$ is of the form $g \alpha g^{-1}$ for some $g$ of $\HH'(\RR)$.
Thus $x$ factors through
$$
g(\TT'\cap \HH)^0_{\RR}g^{-1}. \, g{\HH'}^\der_{\RR}g^{-1}=
(\TT'\cap \HH)^0_{\RR} \, {\HH'}^\der_{\RR} \;\;.$$
It follows that the Mumford-Tate group of $x$ is contained in
$(\TT'\cap \HH)^0 {\HH'}^\der$. For $x$ Hodge generic, we obtain
$$
(\TT'\cap \HH)^0 {\HH'}^\der = \HH' \;\;.
$$
Hence $(\TT' \cap \HH)^0 = \TT'$ and $\TT = \TT'$.
\end{proof}

\begin{propA}
Suppose that the set $\Sigma$ in the theorem~\ref{main-thm2} is such that
with respect to a faithful representation
$\rho \colon \G \lto \GL_n$ the centres $\TT_V$ of the generic
Mumford-Tate groups $\HH_V$ lie in one $\GL_n(\QQ)$-orbit as $V$
ranges through $\Sigma$.

We suppose that, after replacing $\Sigma$ by a modification $\Sigma'$, every
$V$ in $\Sigma'$ is \emph{strictly} contained in a special subvariety $V'\subset Z$.

Then the set $\Sigma'$ 
admits a modification $\Sigma''$ such that the centres $\TT_{V'}$ of the generic
Mumford-Tate groups $\HH_{V'}$ lie in one $\GL_n(\QQ)$-orbit as $V'$
ranges through $\Sigma''$.
\end{propA}

\begin{proof}
First note that an inclusion of special subvarieties
$V\subset V'$ corresponds to an inclusion of Shimura data  
$(\HH_V,X_{\HH_{V}})\subset (\HH_{V'}, X_{\HH_{V'}})$ with $\HH_V$ and $\HH_{V'}$ the generic Mumford-Tate groups
on $X_{\HH_{V}}$ and $X_{\HH_{V'}}$ respectively.

By assumption the connected centre $\TT_V$ of $\HH_V$ 
lie in the $\GL_n(\QQ)$-conjugacy class of a fixed
$\QQ$-torus as $V$ ranges through $\Sigma$. Hence the tori $\TT_{V}$,
$V \in \Sigma$, are split by the same field $L$. By
lemma~\ref{splitting} the tori $\TT_{V'}$ connected centers of the
$\HH_{V'}$, $V' \in \Sigma'$, are all split by $L$.
By \cite{UllmoYafaev}, lemma 3.13, part (i), the tori $\TT_{V'}$ lie in finitely many $\GL_n(\QQ)$-conjugacy classes.
The conclusion of the proposition follows.
\end{proof}

\section{The geometric criterion.}  \label{critere}

In this section we show that given a subvariety $Z$ of a Shimura
variety $\Sh_K(\G, X)_\CC$ containing a special subvariety $V$ and satisfying certain
assumptions, the existence of a suitable element $m \in \G(\QQ_l)$
such that $Z \subset T_mZ$ implies that $Z$ contains a special subvariety $V'$
containing $V$ properly. 

\subsection{Hodge genericity}
\begin{defi}
Let $(\G, X)$ be a Shimura datum, $K \subset \G(\AAf)$ a neat compact
open subgroup, $F\subset \CC$ a number field containing the reflex field $E(\G,
X)$ and $Z \subset \Sh_K(\G, X)_\CC$ an $F$-irreducible subvariety. We
say that $Z$ is Hodge generic if one of its geometrically irreducible
components is Hodge generic in $\Sh_K(\G, X)_\CC$.
\end{defi}

\begin{lem} \label{hodgecomponent}
Let $(\G, X)$ be a Shimura datum, $K \subset \G(\AAf)$ a neat compact
open subgroup, $F\subset \CC$ a number field containing the reflex field $E(\G,
X)$ and $Z \subset \Sh_K(\G, X)_\CC$ a Hodge generic $F$-irreducible subvariety.
Let $Z = Z_1 \cup \cdots \cup Z_n$ be the decomposition of $Z$ into  
geometrically irreducible components.
Then each irreducible component $Z_i$, $1 \leq i \leq n$, is Hodge generic.
\end{lem}

\begin{proof}
As $Z$ is Hodge generic, at least one of its irreducible components, say $Z_1$, is Hodge  
generic. Writing $Z= Z_F \times_{\Spec \, F} \Spec \, \CC$ with $Z_F
\subset \Sh_K(\G, X)_F$ irreducible, any
irreducible component $Z_j$, $1 \leq j \leq n$, is of the form $Z_{1}^{\sigma}$
for some element $\sigma \in \Gal(\overline{\QQ}/F)$. As the conjugate
under any element of $\Gal(\overline{\QQ}/F)$ of a special subvariety of $\Sh_{K}(\G,
X)_{\CC}$ is still special, one gets the result. This is a consequence of a theorem
of Kazhdan. See \cite{MilneKazh} for a comprehensive exposition of the proof 
in full generality and the relevant references.
\end{proof}

\subsection{The criterion}
Our main theorem in this section is the following:


\begin{theor} \label{theor2}
Let $(\G, X)$ be a Shimura datum, $X^+$ a connected component of $X$ and $K= \prod_{p \; \textnormal{prime}} K_p
\subset \G(\AAf)$ an open compact subgroup of $\G(\AAf)$. We assume that
there exists a prime $p_0$ such that the compact open subgroup
$K_{p_0} \subset \G(\QQ_{p_{0}})$ is neat. 
Let $F\subset \CC$ be a number field containing the reflex field $E(\G,X)$.

Let $V$ be a special but not strongly special subvariety of $S_K(\G, X)_\CC$ contained in
a Hodge generic $F$-irreducible subvariety $Z$ of $\Sh_K(\G, X)_\CC$. 

Let $l\not = p_0$ be a prime number splitting $\TT_V$
and $m$ an element of $\TT_V(\QQ_l)$. 

Suppose that $Z$ and $m$ satisfy the following conditions:
\begin{itemize}
\item[(1)] $Z \subset T_mZ$.
\item[(2)] 
Let $\lambda\colon\G\lto \G^\ad$ be the natural morphism.
For every $k_1$ and $k_2$ in $K_l$, the element  $\lambda(k_1 m k_2)$
generates an unbounded (i.e. not relatively compact) subgroup of $\G^{\ad}(\QQ_{l})$. 
\end{itemize}

Then $Z$ contains a special subvariety $V'$ containing $V$ properly.
\end{theor}

\begin{proof}

\begin{lem} \label{adjoint}
If the conclusion of the theorem~\ref{theor2} holds for all Shimura data $(\G, X)$
with $\G$ semisimple of adjoint type then it holds for all Shimura data. 
\end{lem}
\begin{proof}

Let $\G$, $X$, $K$, $p_0$, $F$, $V$, $Z$, $l$ and $m$ be as in the statement
of theorem~\ref{theor2}. In particular $Z= Z_F \times_{\Spec \, F}
\Spec \, \CC$ with $Z_F \subset \Sh_K(G, X)_F$ an irreducible subvariety.
Let $(\G^\ad,X^\ad)$ be the adjoint Shimura datum attached to $(\G,X)$
and $(X^\ad)^+$ be the image of $X^+$ under the natural morphism $X
\lto X^\ad$. Let
$K^\ad = \prod_{p \; \textnormal{prime}} K^\ad_p$ be the compact open
subgroup of $\G^\ad(\AAf)$ defined as follows:
\begin{enumerate}
\item $K^\ad_{p_0} \subset \G^\ad(\QQ_{p_{0}})$ is the compact open
  subgroup image of $K_{p_0}$ by $\lambda$.
\item $K^\ad_l \subset \G^\ad(\QQ_l)$ is the compact open subgroup
  image of $K_l$ by $\lambda$.
\item If $p\not\in \{p_0,l\}$, the group $K^\ad_p$ is a maximal compact open subgroup
  of $\G^\ad(\QQ_p)$ containing the image of $K_p$ by $\lambda$.
\end{enumerate}
The group $K^{\ad}$ is neat because $K_{p_{0}}$, and therefore
$K^{\ad}_{p_{0}}$, is. 
As the reflex field $E(\G, X)$ contains the reflex field
$E(\G^\ad , X^\ad)$ and $K^\ad$ contains $\lambda(K)$  there is a finite morphism of Shimura varieties
$
f: \Sh_{K}(\G,X)_F \lto \Sh_{K^\ad}(\G^\ad,X^\ad)_F 
$.

We define the irreducible subvariety $Z^\ad_F$ of $\Sh_{K^\ad}(\G^\ad,X^\ad)_F$ to be the image of $Z_F$ in
$\Sh_{K^\ad}(\G^\ad,X^\ad)_F$ by this morphism. Its base-change $Z^\ad
:=Z^{\ad}_{F} \times_{\Spec \,F} \Spec \, \CC$, which is
$F$-irreducible, coincides with
$f_\CC(Z)$, where $f_\CC: \Sh_{K}(\G,X)_\CC \lto
\Sh_{K^\ad}(\G^\ad,X^\ad)_\CC$ is the base change of $f$.

Let $V^\ad$ be the image $f_\CC(V)$. As $V$ is special but not strongly special,
$V^\ad$ is a special but not strongly special subvariety of
$S_{K^\ad}(\G^\ad, X^\ad)_\CC$. Thus $\TT_{V^{\ad}} = \lambda (\TT_V)$
is a non-trivial torus.

Let $m^\ad := \lambda(m)$.
The inclusion $Z\subset T_mZ$ implies that
$Z^\ad \subset T_{m^\ad}Z^\ad$. As $K^\ad_l = \lambda(K_l)$ the
condition~$(2)$ for $m$ and $K_l$ implies the condition~$(2)$ for $m^\ad$
and $K^\ad_l$.

Thus $\G^\ad$, $X^\ad$, $K^\ad$, $p_0$, $F$, $V^\ad$, $Z^\ad$, $l$ and
$m^\ad$ satisfy the assumptions of theorem~\ref{theor2}. As irreducible components of the preimage of a special subvariety by a finite morphism of Shimura varieties  
are special, it is enough to show that $Z^\ad$ contains a special subvariety $V^{'\ad}$ containing
 $V^\ad$ properly to conclude that $Z$ contains a special subvariety
 $V'$ containing $V$ properly.
\end{proof}

{\em For the rest of the proof of theorem~\ref{theor2}, we are assuming the group $\G$ to be semisimple of adjoint type.
Moreover we will drop the label $(\G,X)$ when it is obvious which Shimura datum
we are referring to.}

We fix a $\ZZ$-structure on $\G$ and its subgroups by choosing a
finitely generated free $\ZZ$-module $W$, a
faithful representation $\xi \colon \G \into \GL(W_{\QQ})$ and taking the
Zariski closures in the $\ZZ$-group-scheme $\GL(W)$. We choose the  
representation $\xi$
in such a way that $K$ is contained in $\GL(\Zhat \otimes_{\ZZ} W)$  
(i.e. $K$ stabilizes $\Zhat \otimes_{\ZZ} W$).
This induces canonically a $\ZZ$-variation of Hodge structure $\mathcal{F}$ on $\Sh_K(\G,X)_{\CC}$
(cf. \cite[section 3.2]{EdYa}), in particular on  
its irreducible component $S_K(\G,X)_{\CC}$.

Let $Z_1$ be a geometrically irreducible component of $Z$ containing $V$.
Let $z$ be a Hodge generic point of the smooth locus $Z_1^\sm$ of $Z_1$.
Let $\pi_1(Z_1^\sm,z)$ be the topological fundamental group of
$Z_1^{\sm}$ at the point $z$.
We choose a point $z^+$ of $X^+$ lying above $z$. This choice
canonically identifies the fibre at $z$ of the locally constant sheaf
underlying $\mathcal{F}$ with the $\ZZ$-module $W$.
The action of $\pi_1(Z_1^\sm,z)$ on this fibre is described by the  
monodromy representation
$$
\rho\colon \pi_1(Z_1^\sm,z) \lto \Gamma_K= \pi_{1}(S_{K}(\G,
X)_{\CC}, z)= \G(\QQ)^+ \cap K  \stackrel{\xi}{\lto} \GL(W)\;\;.
$$
By proposition~\ref{hodgecomponent} the subvariety
$Z_1$ is Hodge generic in $S_K(\G,X)_{\CC}$. Hence the Mumford-Tate group
of $\mathcal{F}_{|Z_1^\sm}$ at $x$ is $\G$. It follows from 
\cite[theor. 1.4]{Mo} and the fact that the group $\G$ is adjoint that the
group $\rho(\pi_1(Z_1^\sm,z))$ is Zariski-dense in $\G$.

Let $l$ be a prime as in the statement.
The proposition~\ref{approx} implies that the $l$-adic
closure of $\rho(\pi_1(Z_1^\sm,z))$ in
$\G(\QQ_l)$ is a compact open subgroup $K'_l \subset K_l$.

Write $K=K^lK_l$ with $K^l = \prod_{p\not= l}K_p$.
Let $\pi_{K_{l}} \colon \Sh_{K^{l}} \lto \Sh_K$ be the Galois pro-{\'e}tale cover
with group $K_l$ as defined in section~\ref{shim_l}. Let
$\widetilde{Z_1}$ be an irreducible component of the preimage of
$Z_1$ in $\Sh_{K^{l}}$ and let $\widetilde{V}$ be an irreducible
component of the preimage of $V$ in $\widetilde{Z_1}$. 

The idea of the proof is to show that the inclusion $Z\subset T_m{Z}$ implies
that $\widetilde{Z_1}$ is stabilized by a ``big'' group and then
consider the orbit of $\widetilde{V}$ under the action of this group.

\begin{lem} \label{lemme1}
The variety $\widetilde{Z_1}$ is stabilized by the group $K'_l$. The
set of irreducible components of $\pi_{K_{l}}^{-1}(Z_1)$ naturally
identifies with the finite set $K_l /K'_l$.
\end{lem}

\begin{proof}
Let $\wt{z}$ be a geometric point of $\wt{Z_1^{\sm}}$ lying over
$z$. As $\pi_{K_{l}} : \Sh_{K^{l}} \lto \Sh_K$ is pro-{\'e}tale, the
set of irreducible components of $\pi_{K_{l}}^{-1}(Z_1)$ 
naturally identifies with the set of connected components of 
$\pi_{K_{l}}^{-1}(Z_1^{\sm})$. This set identifies with the quotient $K_l/
\rho_{\alg}(\varpi_1(Z_{1}^{\sm},z))$ where $\varpi_1(Z_{1}^{\sm},z)$ denotes the algebraic fundamental group of  
$Z_1^{\sm}$ at $z$ and $\rho_{\alg} \colon
\varpi_1(Z_{1}^{\sm},z) \lto K_l \subset \G(\QQ_l)$ denotes the
(continuous) monodromy representation of the $K_l$-pro-{\'e}tale
cover $\pi_{K_{l}}: \pi_{K_{l}}^{-1}( Z_1^{\sm}) \lto Z_1^{\sm}$.
The group $\varpi_1(Z_{1}^\sm,z)$ naturally identifies with the
profinite completion of $\pi_{1}(Z_{1}^{\sm}, z)$. One has the  
commutative diagram
\begin{equation} \label{exact}
\xymatrix{
\pi_{1}(Z_{1}^{\sm}, z) \ar[d]_{i}  \ar[r]^{\rho}
&\G(\QQ) \ar[d]^{j}\\
\varpi_{1}(Z_{1}^{\sm}, z)   \ar[r]_{\rho_{\alg}} &\G(\QQ_l)}
\end{equation}
where $i : \pi_{1}(Z_{1}^{\sm},z) \lto \varpi_{1}(Z_{1}^{\sm}, z)$
and $j: \G(\QQ) \lto \G(\QQ_l)$ denote the natural homomorphisms.
As $i(\pi_{1}(Z_{1}^{\sm},z))$ is dense in $\varpi_{1}(Z_{1}^{\sm},
z)$ and $\rho_{\alg}$ is continuous one deduces that  
$\rho_{\alg}(\varpi_{1}(Z_{1}^{\sm},
z))= K'_{l}$. Thus the set of irreducible components of $\pi_{K_{l}}^{-1}(
Z_1^{\sm})$ identifies with $K_l / K'_l$ and $\wt{Z_{1}^{\sm}}$
is $K'_l$-stable.
\end{proof}

\begin{lem} \label{lemme2}
There exist elements $k_1, k_2$ of $K_l$ and an integer $n\geq 1$ such that
$$
\wt{Z_1} = \widetilde{Z_1} \cdot (k_1 m k_2)^n
$$
\end{lem}
\begin{proof}
Let $Z_i$, $ 2\leq i \leq n$, be the geometrically irreducible
components of $Z$ different from $Z_1$. For each $i \in \{2, \ldots, n\}$, let us fix a
geometrically irreducible component  $\widetilde{Z_{i\,}}$ of $\pi_{K_{l}}^{-1}(Z_i)$.
The inclusion $Z\subset T_mZ$ implies that, for $i \in \{1, \ldots,
n\}$, the component $\widetilde{Z_{i\,}}$ of
$\pi_{K_{l}}^{-1}(Z_i)$ is
also a geometrically irreducible component of $\pi_{K_{l}}^{-1}(T_mZ)$.
As the geometrically irreducible components of $\pi_{K_{l}}^{-1}(T_mZ)$ are of the form 
$\widetilde{Z_{i\,}}\cdot (k_1m k_2)$, $k_1, k_2 \in K_l$, there exists an index $i$, $1 \leq i \leq n$, and two elements $k_1$, $k_2$  in 
$K_l$ such that
\begin{equation} \label{ee0}
\wt{Z_{1}} = \wt{Z_{i\,}} \cdot k_1 m k_2\;\;.
\end{equation}

As $Z$ is $F$-irreducible there exists $\sigma$ in $\Gal(\ol\QQ/F)$
such that $Z_i = \sigma(Z_1)$. As the morphism $\pi_{K_{l}} :
\Sh_{K^{l}} \lto \Sh_K$ is defined over $F$, the subvariety
$\sigma(\wt{Z_{1}})$ of $\Sh_{K^{l}}$ satisfies
$\pi_{K_{l}}(\sigma(\wt{Z_{1}}))= Z_i$. Hence the subvarieties
$\sigma(\wt{Z_{1}})$ 
and $\wt{Z_{i\,}}$ of $\Sh_{K^{l}}$ are both irreducible components of
$\pi_{K_{l}}^{-1}(Z_i)$. Thus there exists an element $k$ of $K_l$
such that 
\begin{equation} \label{ee1}
\widetilde{Z_{i\,}} = \sigma(\widetilde{Z_1}) \cdot k\;\;\;.
\end{equation}
From $(\ref{ee0})$ and $(\ref{ee1})$ and replacing $k_1$ with $kk_1$, we obtain $k_1$, $k_2$ in 
$K_l$ such that
\begin{equation} 
\widetilde{Z_1} = \sigma(\widetilde{Z_1}) \cdot (k_1 m k_2)\;\;\;.
\end{equation}
As the $\G(\AAf)$-action is defined over $F$, the previous equation
implies:
\begin{equation} \label{trans}
\forall \, j \in \NN, \;\;\; \widetilde{Z_1} = \sigma^j(\widetilde{Z_1}) \cdot (k_1 m k_2)^j\;\;\;.
\end{equation}
As the set of irreducible components of $Z$ is finite, there exists a
positive integer $m$ such that $\sigma^m(Z_1) =Z_1$. Thus the Abelian group $(\sigma^m)^{\Z}$
acts on the set of irreducible components of $\pi_{K_{l}}^{-1}(Z_{1})$. By
lemma~\ref{lemme1} this set is finite. So there exists a positive integer $n$
(multiple of $m$) such that $\sigma^n(\wt{Z_1}) = \wt{Z_1}$. 
The equality~(\ref{trans}) applied to $j=n$ concludes the proof of the lemma.
\end{proof}
 
From the lemmas~\ref{lemme1} and \ref{lemme2} one obtains the
\begin{cor}
Let  $U_{l}$ be the group $\langle K'_{l}, (k_1 m k_2)^{n} \rangle$.
The variety $\widetilde{Z_1}$ is stabilized by $U_l$.
\end{cor}
 
We now conclude the proof of theorem~\ref{theor2}. Let $\G = \prod_{i=1}^s \G_i$ be the
decomposition of the semisimple $\QQ$-group of adjoint type $\G$ into
$\QQ$-simple factors and $X =
\prod_{i=1}^s X_i$ (resp. $X^+ = \prod_{i=1}^s X_i^+$) the associated
decomposition of $X$ (respectively $X^+$). The
Shimura datum $(\G, X)$ is the product of the Shimura data $(\G_i,
X_i)$, $1\leq i \leq s$, where each $\G_i$ is simple of adjoint type. Let $p_i : \G \lto \G_i$, $1 \leq i \leq s$, denote the natural
projections. Let $(\G_{> 1}, X_{>1})$ be the Shimura datum $(\prod_{i=2}^s \G_i,\prod_{i=2}^s X_i)$.

By the assumption made on $m$, the group $U_l$ is unbounded in
$\G(\QQ_l)$. After possibly renumbering the factors,  we can assume that
$p_1(U_{l})$ is unbounded in $\G_{1}(\QQ_{l})$. In particular the
torus $\TT_{V,1}:=p_{1}(\TT_V)$ is non-trivial. Indeed if it was trivial, then the
group $p_1(U_l)$ would be contained in $p_1(K_l)$ which is  
compact.

Let $\G_{1, \QQ_{l}}= \prod_{j=1}^r \HH_j$ be the decomposition of
$\G_{1, \QQ_{l}}$ into $\QQ_l$-simple factors.  Again, up to
renumbering we can assume that the image of $U_l$ under the projection $h_1 : \G_{ \QQ_{l}}
\lto \HH_1$ is unbounded in $\HH_1(\QQ_{l})$. Let $\HH_{>1} =
\prod_{j=2}^r \HH_j$. 
Let $\tau: \wt{\G_{\QQ_{l}}} \lto \G_{\QQ_{l}}$ (resp. $\tau_1 :
\wt{\HH_{1}} \lto \HH_{1}$) be the universal cover of $\G_{\QQ_{l}}$
(resp. $\HH_{1}$).

\begin{lem}
The group $U_{l} \cap \HH_{1}(\QQ_{l})$ contains the group $\tau_1(\wt{\HH_{1}} (\QQ_l))$
with finite index.
\end{lem}
\begin{proof}
Let $\wt{h_1} : \wt{\G_{\QQ_{l}}} \lto \wt{\HH_{1}}$ be the canonical
projection. Let $\wt{U_{l}}= \tau^{-1}(U_{l}) \subset  
\wt{\G_{\QQ_{l}}}(\QQ_{l})$.
As $U_{l}$ is an open non-compact subgroup of $\G_{\QQ_{l}}(\QQ_{l})$,  
the group $\wt{U_{l}}$ is open
non-compact in $\wt{\G_{\QQ_{l}}}(\QQ_{l})$. As $h_1(U_l)$ is
non-compact in $\HH_{1}(\QQ_{l})$ the projection  
$\wt{h_1}(\wt{U_{l}})$ is open non-compact
in the group $\wt{\HH_{1}}(\QQ_{l})$.

Notice that the group $\wt{U_{l}} \cap  \wt{\HH_{1}}(\QQ_{l})$  
is normalized by the subgroup $\wt{h_1}(\wt{U_{l}})$ of $\wt{\HH_{1}}(\QQ_{l})$. 
Indeed, given $h \in \wt{h_1}(\wt{U_{l}})$, let $g \in \wt{U_{l}}$
satisfying $\wt{h_1}(g)=h$. As $\wt{\HH_{1}}$ is a direct
factor of $\wt{\G_{\QQ_{l}}}$ one obtains: $$(\wt{U_{l}} \cap
\wt{\HH_{1}}(\QQ_{l}))^{h} = (\wt{U_{l}} \cap
\wt{\HH_{1}}(\QQ_{l}))^{g}= (\wt{U_{l}} \cap \wt{\HH_{1}}(\QQ_{l}))
\;\;.$$

As $\wt{h_1}(\wt{U_{l}})$ is open non-compact and normalizes
$\wt{U_{l}} \cap  \wt{\HH_{1}}(\QQ_{l})$, it follows from
\cite[theor.2.2]{Pink} that
$\wt{U_{l}} \cap  \wt{\HH_{1}}(\QQ_{l})=\wt{\HH_{1}}(\QQ_{l})$. 
As $\tau_1$ is an isogeny of algebraic groups, we get that
$U_{l} \cap \HH_{1}(\QQ_{l})$ contains
$\tau_1(\wt{\HH_{1}} (\QQ_l))$ with finite index.
\end{proof}

Define $K_{1,l}$ as the compact open subgroup $p_1(K_l)$ of
$\G_{1,\QQ_l}$ and $K_{>1,l}$ as the compact open subgroup $(p_2
\times \cdots \times p_s)(K)$ of $\G_{>1,\QQ_l}$. As $U_l$ is an open
subgroup of $\G_{\QQ_l}(\QQ_l)$ it contains a compact open subgroup of
 $\G_{1,\QQ_l}(\QQ_l) =\prod_{j=1}^r \HH_j(\QQ_l)$, in particular a
 compact open subgroup $U_{l,1}$ of $K_1 \cap
 \HH_{>1}(\QQ_l)$. Similarly $U_l$ contains a compact open subgroup
 $U_{l,>1}$ of $K_{>1}$.
The previous lemma shows that $U_l$ contains the unbounded open subgroup $\tau_1(\wt{\HH_{1}} (\QQ_l)) \cdot U_{l,1} \cdot U_{l, >1}$.

\begin{defi} \label{redefi}
We replace $U_l$ by its subgroup $\tau_1(\wt{\HH_{1}} (\QQ_l)) \cdot U_{l,1} \cdot U_{l, >1}$.
We denote by $V'$ the Zariski closure
$\overline{\pi_{K_{l}}(\wt{V} \cdot U_l)}^{\textnormal{Zar}}$.
\end{defi}

As $\widetilde{Z_1}$ is stabilised by $U_l$, the variety $V'$ is a
subvariety of $Z$.

\begin{lem} \label{V'special}
The subvariety $V'$ of $Z$ is special.
\end{lem}

\begin{proof}

Define $K_i:= p_i(K)$, $1 \leq i \leq s$, and $\K:= \prod_{i=1}^s K_i$. As
the group $K_{p_{0}}$ is
neat its projections $K_{i, p_{0}}$, $1 \leq i \leq s$, are also neat,
hence $\K$ is neat. Let $f: \Sh_{K}(\G,
X)_\CC\lto \Sh_{\K}(\G, X)_\CC$ be the natural finite morphism, $\mathcal{Z}:=
f(Z)$, $\V=f(V)$ and $\V'= f(V')$. As $f$ is a finite morphism
it follows that $\V'$ is also the Zariski closure
$\overline{(f\circ\pi_{K_{l}})(\wt{V} \cdot U_l)}^{\textnormal{Zar}}$
of $(f\circ\pi_{K_{l}})(\wt{V} \cdot U_l)$ in $\Sh_{\K}(\G, X)_\CC$. 

As in the proof of lemma~\ref{adjoint} it is enough to show that $\V'$ is
special to conclude that $V'$ is special. 

Let $K_{> 1}$ be the compact open subgroup $\prod_{i=2}^s K_{i}$
of $\G_{> 1}(\AAf)$. The connected
component $S_\K(\G, X)_{\CC}$ of the Shimura
variety $\Sh_{\K}(\G, X)_{\CC}$ decomposes as a product
$$S_{\K}(\G, X)_{\CC} = S_{K_{1}}(\G_{1}, X_{1})_{\CC} \times
S_{K_{> 1}}(\G_{> 1}, X_{> 1})_{\CC}\;\;$$ with $S_{K_{>1}}(\G_{>1}, X_{> 1})_{\CC}=  
\prod_{i=2}^s S_{K_{i}}(\G_i, X_i)_{\CC}$.

Let $\V_{> 1}$ denote the special subvariety of $S_{K_{> 1}}(\G_{>  
1}, X_{> 1})_{\CC}$ projection of $\V$. Thanks to the definition~\ref{redefi} of $U_l$ the
inclusion 
\begin{equation} \label{incl}
\V' \subset S_{K_{1}}(\G_{1}, X_{1})_{\CC} \times \V_{> 1}
\end{equation}
holds. 

For an element $q \in \G_1(\QQ)^+$, we let $\Gamma_q$ be the subgroup of $\G_1(\QQ)^+$ generated by $\Gamma:= K_1\cap \G_1(\QQ)^+$ and $q$.
We claim that we can choose $q \in \G_1(\QQ)^+ \cap U_l$ such that the index of $\Gamma$ in $\Gamma_q$ is infinite.

Indeed let $g \in \wt{\HH}_1(\QQ_l)$ be an element contained in a
split subtorus of $\wt{\HH}_{1,\QQ_l}$ but not in the maximal compact
subgroup of this subtorus. Then $g$ is not contained
in any compact subgroup of $\wt{\HH}_1(\QQ_l)$, hence its image
$h:=\tau(g) \in U_l$ is not contained in any compact
subgroup of $\G_1(\QQ_l)$. As $\G_1$ is simple and adjoint it has the weak approximation
property \cite[theorem 7.8]{PlaRap}: the group $\G_1(\QQ)$ is dense
in $\G_1(\QQ_l)$. Let $\overline{\Gamma}^l$ denote the $l$-adic closure of $\Gamma$ in
$\G_1(\QQ_l)$, this is a compact open subgroup of $\G_1(\QQ_l)$
by proposition \ref{approx}. 
As $\G_1(\QQ)^+$ has finite index in $\G_1(\QQ)$, the $l$-adic closure
$\ol{\G_1(\QQ)^+}^l$ of $\G_1(\QQ)^+$ in $\G_1(\QQ_l)$ is an open
subgroup of finite index of $\G_1(\QQ_l)$. 
By replacing $h$ by a suitable positive power, we may assume that $h \in\ol{\G_1(\QQ)^+}^l$.
The group $\ol{\Gamma}^l \cap U_l$ is an open subgroup of $\ol{\G_1(\QQ)^+}^l$, therefore 
 there exist elements $q$ of $\G_1(\QQ)^+$ and
$k$ of $\overline{\Gamma}^l \cap U_l$ such that $h = q k$.
It follows that $q \in \G_1(\QQ)^+ \cap U_l$. 
We claim that $\Gamma$ has infinite index
in $\Gamma_q$. Suppose the contrary. Then the $l$-adic closure 
$\overline{\Gamma_q}^l$ of $\Gamma_q$ in $\G_1(\QQ_l)$ contains
$\overline{\Gamma}^l$ with finite index, hence is compact. But, by construction, 
$h \in \overline{\Gamma_q}^l$ and $h$ is not contained in any compact subgroup of $\G_1(\QQ_l)$.
This gives a contradiction.

Let us show that $\Gamma_q$ is dense in $\G_1(\RR)^+$ (for the Archimedian topology). Let $H$ be the Lie subgroup of $\G_1(\RR)^+$
closure of $\Ga_q$ and let $H^+$ be its connected component of the
identity. First notice that the group $\Gamma$ normalizes $H^+$, hence its Lie
algebra. As $\G_1$ is $\RR$-isotropic, it follows from
\cite[theor. 4.10]{PlaRap} that $\Gamma$ is Zariski-dense in $\G_{1,
  \RR}$. Hence $H^+$ is a product of simple factors of
$\G_1(\RR)^+$. The $\QQ$-simple group $\G_1$ can be written as
the restriction of scalars $\Res_{L/\QQ} \G'_1$, with $L$ a number
field and $\G'_1$ an absolutely almost simple algebraic group over
$L$.  As $H^+ \cap \G_1(\QQ)$ is dense in $H^+$ it follows that $H^+ = \G_1(\RR)^+$ as soon as $H^+$ is
non-trivial. If $H^+$ were trivial the group $\Gamma_q$ would be
discrete in $\G_1(\RR)^+$. As $\Gamma_q$ contains the lattice $\Ga$ of
$\G_1(\RR)^+$, necessarily $\Gamma_q$ would also be a lattice of $\G_1(\RR)^+$, 
containing $\Gamma$ with finite index. This contradicts the fact that $\Gamma_q$
contains $\Gamma$ with infinite index.

Let $x= (x_1, x_{>1}) \in X_1^+ \times X_{>1}^+$ be any point whose
projection in $S_{K_{1}}(\G_1, X_1)_\CC \times S_{K_{> 1}}(\G_{>  
1}, X_{> 1})_{\CC}$ lies in $\V$. Let $\OO:= (\Gamma_q \cdot x_1,
x_{>1})$ be the $\Gamma_q$-orbit of $x$ in $X_1^+ \times X_{>1}^+$. By
definition of the group $\Gamma_q$ the closure of $\OO$ in $X_1^+ \times X_{>1}^+$ is mapped to
$\V'$ under the uniformization map
$$X_1^+ \times X_{>1}^+ \lto S_{K_{1}}(\G_1, X_1)_\CC
\times S_{K_{>1}}(\G_{>1}, X_{>1})_\CC\;\;.$$ 
As $\Gamma_q$ is dense in $\G_1(\RR)^+$ this closure is nothing else
than $X_1^+ \times  x_{>1}$. 
Thus: 
\begin{equation} \label{subcl}
\V' \supset S_{K_{1}}(\G_{1}, X_{1})_{\CC} \times \V_{> 1}\;\;.
\end{equation}
Finally it follows from (\ref{incl}) and (\ref{subcl}) that:
$$\V' = S_{K_{1}}(\G_{1}, X_{1})_{\CC} \times \V_{> 1}\;\;.$$
In particular $\V'$ is special. Hence $V'$ is special.
\end{proof}

\begin{lem}
The subvariety $V'$ of $Z$ contains $V$ properly.
\end{lem}

\begin{proof}
Obviously $V'$ contains $V$. Let us show that $V' \not= V$.
Once more it is enough to show that $\V' \not = \V$.

As the generic Mumford-Tate group $\HH_V$ of $V$, hence of $\V$, centralizes the torus $\TT_V$,
the projection $\HH_{V,1}$ of $\HH_V$ on $\G_{1}$ centralizes the  
non-trivial torus $\TT_{V,1}$ projection
of $\TT_V$ on $\G_{1}$. In particular $\HH_{V,1}$ is a proper
algebraic subgroup of $\G_{1}$. But as $$\V' = S_{K_{1}}(\G_{1},
X_{1})_{\CC} \times \V_{> 1}\;\;,$$
the group $\G_{1}$ is a direct factor of the generic Mumford-Tate group of $\V'$.
\end{proof}

This finishes the proof of theorem~\ref{theor2}.
\end{proof}
\section{Existence of suitable Hecke correspondences.} \label{hecke}

In this section we prove, under some assumptions on the compact open
subgroup $K_{l}$, the existence of Hecke correspondences of small degree
candidates for applying theorem~\ref{theor2}. 
Our main result is the following:

\begin{theorA} \label{good Hecke}
Let $(\G', X')$ be a Shimura datum with $\G'$ semisimple of adjoint type, $X'^+$ a connected component of
$X'$ and $K' = \prod_{p \; \textnormal{prime}} K'_p$ a neat open compact
subgroup of $\G'(\AAf)$. We fix  a faithful rational representation $\rho: \G' \hookrightarrow 
\GL_n$ such that $K'$ is contained in $\GL_n(\Zhat)$. 

There exist positive integers $k$ and $f$ such that the following holds. 

Let $(\G, X)$ be a Shimura subdatum of $(\G',X')$, let $X^+$ be a connected
component of $X$ contained in ${X'}^+$ and $K:=K'\cap G(\AAf)$. 
Let $V$ be a special but not strongly special subvariety of
$S_{K}(\G,X)_\CC$ defined by a Shimura subdatum $(\HH_V,X_{V})$ of
$(\G, X)$. Let $\TT_V$ be the connected centre of $\HH_V$ and $E_V$ the
reflex field of $(\HH_V, X_V)$. 

Let $l$ be a prime number such that $K'_l$ is a hyperspecial maximal
compact open subgroup in $\G'(\QQ_l)$ which coincides with
$\G'(\ZZ_l)$, the prime $l$ splits $\TT_V$ and $(\TT_{V})_{\FF_l}$ is a torus.

There exist a compact open subgroup $I_l\subset K_l:= K'_l \cap
\G(\QQ_l)$ in good position with respect to $\TT_V$ and an element $m \in \TT_V(\QQ_l)$ 
satisfying the following conditions:

\begin{itemize}
\item[(1)] $[K_l: I_l] \leq l^f$.
\item[(2)] Let $I \subset K$ be the compact open subgroup $K^l I_l$ of
  $\G(\AAf)$ (where $K^l := K' \cap \G(\AAf^l)$) and 
$\tau : \Sh_I(\G, X)_\CC \lto \Sh_K(\G, X)_\CC$ be the natural morphism. Let $\wt{V} \subset S_I(\G,
X)_\CC$ be an irreducible component of $\tau^{-1}(V)$. There exists an element $\sigma$ in $\Gal(\ol \QQ/E_V)$ such that $\sigma \wt{V} \subset T_m(\wt{V})$.

\item[(3)] 
For every $k_{1}, k_2 \in I_{l}$ the image of $k_1 m k_2$ in $\G^{\ad}(\QQ_l)$
generates an unbounded subgroup of 
$\G^{\ad}(\QQ_{l})$.

\item[(4)] $[I_l : I_l \cap mI_l m^{-1}] < l^{k}$.
\end{itemize}
\end{theorA}

\begin{rems} \label{remIwa}
\begin{itemize}
\item[(a)] As noticed in the introduction, conclusion~$(3)$ in
  theorem~\ref{good Hecke} can not be ensured if we stay at a level
  $K_l$ which is a maximal compact subgroup of $\G(\QQ_l)$ and do not
  lift the situation to a smaller level $I_l$. For explicit
  counterexamples see remark 7.2 of \cite{Ed2Curves}.
\item[(b)] As already noticed in section~\ref{integral} the condition that $K'_l$ is a hyperspecial maximal
compact open subgroup in $\G'(\QQ_l)$ which coincides with
$\G'(\ZZ_l)$ is satisfied for almost all primes $l$.

\end{itemize}
\end{rems}

\subsection{Iwahori subgroups} We refer to \cite{bt}, \cite{car} and
\cite{Heckop} for more details about
buildings, Iwahori subgroups and Iwahori-Hecke algebras.

\subsubsection{} We first recall the definition of an Iwahori subgroup. Let $l$ be a
prime number. Let $\G$ be a reductive linear algebraic isotropic $\QQ_{l}$-group and $\AA \subset\G$ a
maximal split torus of $\G$. We denote by $\MM \subset \G$ the
centraliser of $\AA$ in $\G$. Let $\X$ be the (extended) Bruhat-Tits building of $\G$ and $\A \subset \X$ the
apartment of $\X$ associated to $\AA$. Let $K^{\mathrm{m}}_{l} \subset
\G(\QQ_{l})$ be a special maximal compact subgroup (c.f \cite[(I), def. 1.3.7 p.22,
def. 4.4.1 p.79]{bt}) of $\G(\QQ_{l})$ in good position with respect
to $\AA$ (cf. section~\ref{good_position} for the notion of ``good position''). We denote by  $x_{0} \in \A$ the unique $K_{l}^\mathrm{m}$-fixed
vertex in $\X$. We choose $C$ a chamber of $\A$ containing $x_0$ in
its closure, we denote by $\Iw_l
\subset K_l^\mathrm{m}$ the Iwahori subgroup fixing $C$ pointwise and
by $\Cone \subset \A$ the unique Weyl chamber with apex at $x_{0}$
containing $C$.

All Iwahori subgroups of $\G(\QQ_l)$ are conjugate, cf. \cite[3.7]{tits}.

\begin{rem}
Strictly speaking (i.e. with the notations of Bruhat-Tits \cite{bt})
the group $\Iw_l$ as defined above is an Iwahori subgroup only in the case where the group $\G^\der$ is 
simply-connected. Our terminology is a well-established abuse of notations. 
\end{rem}

\subsubsection{Iwahori subgroups and unboundedness} \label{iwahori}

\begin{defi}
We denote by $\ordre_{\MM} : \MM(\QQ_{l}) \lto X_{*}(\MM)$ the
homomorphism characterized by 
$$ \forall \, \alpha \in X^*(\MM), \quad < \ordre_{\MM}(m), \alpha> = \ordre_{\QQ_{l}}(\alpha(m)) \;\;,$$
where $\ordre_{\QQ_{l}}$ denotes the normalized (additive) valuation on
$\QQ_{l}^*$ and $X_{*}(\MM)$ (resp. $X^*(\MM)$) denotes the group of
cocharacters (resp. characters) of $\MM$.
We denote by $\Lambda \subset X_{*}(\MM)$ the free $\Z$-module $\ordre_{\MM}(\MM(\QQ_{l}))$.
\end{defi}

The group $\MM(\QQ_{l})$ (in particular the group $\AA(\QQ_{l})$) acts
on $\A$ via $\Lambda$-translations. 

\begin{defi} \label{defi8.1.3}
Let $\Lambda^{+} \subset \Lambda$ be the positive cone associated to
the Weyl chamber $\Cone$.
\end{defi}

Elements of $\Lambda^{+}$ acting on $\A$ map $\Cone$ to $\Cone$.

\begin{prop} \label{noncompact}
Let $m$ be an element of $\AA(\QQ_{l})$ with non-trivial image
$\ordre_{\MM}(m) \in \Lambda^{+}$. Then for any elements $i_{1}, i_{2} \in \Iw_l$, the
element $i_1 m i_2 \in \G(\QQ_{l})$ is not contained in a compact subgroup of
$\G(\QQ_{l})$.
\end{prop}

\begin{proof}
Let $W_{0}$ be the finite Weyl group of $\G$, let $W$ be the modified affine Weyl
group associated to $\A$ and $\Omega$ the finite subgroup of $W$
taking the chamber $C$ to itself. Let $\Delta= \{\alpha_1,\ldots,
\alpha_m\}$ be the set of affine roots on $\A$ which are positive on
$C$ and whose null set $H_{\alpha}$ is a wall
of $C$. For $\alpha \in \Delta$ we denote by $S_{\alpha}$ the
reflexion of $\A$ along the wall $H_{\alpha}$. The group $W$ is generated by
$\Omega$ and the $S_{\alpha}$'s, $\alpha \in \Delta$. It identifies with the semi-direct
product $W_{0} \ltimes \Lambda$ (cf. \cite[p.140]{car}). 

Recall the Bruhat-Tits
decomposition:
\begin{equation} \label{brutits}
\G(\QQ_{l}) = \Iw_l \cdot W \cdot \Iw_l \;\;\;,
\end{equation}
where by abuse of notations we still write $W$ for a set of
representatives of $W$ in $\G(\QQ_{l})$.
Let $r : \G(\QQ_{l}) \lto W$ be the map sending $g \in \G(\QQ_{l})$
to the unique $r(g) \in W$ such that $r(g) \in \Iw_l g\Iw_l$. Geometrically
speaking the map $r$ essentially coincides with the retraction
$\rho_{\A, C}$ of the Bruhat-Tits building $\X$ with centre the chamber $C$ onto the
apartment $\A$ (\cite[I, theor.2.3.4]{bt}).

Let $\He(\G, \Iw_l)$ be the Hecke
algebra (for the convolution product) of bi-$\Iw_l$-invariant compactly
supported continuous complex functions on $\G(\QQ_{l})$.
By the equation~(\ref{brutits}) this is an associative algebra with a vector space basis $T_{w} =
1_{\Iw_l w\Iw_l}$, $w \in W$, where $1_{\Iw_lw\Iw_l}$ denotes the characteristic
function of the double coset $\Iw_lw\Iw_l$. A presentation of the
algebra $\He(\G, \Iw_l)$ with generators $T_{\omega}$, $\omega \in
\Omega$, and $T_{\alpha}$, $\alpha \in \Delta$, is given in \cite[theorem 3.6 p.142]{car} (or
\cite[p.242-243]{bor1}). Given $w \in W$ let $l(w) \in \NN$ be the number of hyperplanes $H_{\alpha}$
separating the two chambers $C$ and $wC$.  One obtains in particular
(cf. \cite[theorem 3.6 (b)]{car} or \cite[section 3.2, 1) and 6)]{bor}):
\begin{equation} \label{hi1}
\forall \,w, w' \in W, \;\;\; T_w \cdot T_{w'} = T_{ww'} \;\; \textnormal{if } l(ww')= l(w)
  +l(w')\;\;.
\end{equation}

Let $\delta \in X^*(\MM)$ be the determinant of the adjoint action of
$\MM$ on the Lie algebra of $\NNN$. For $\lambda \in \Lambda^+ \subset
W$ one shows the equality (cf. \cite[(1.11)]{Heckop}):
\begin{equation} \label{hi2}
l(\lambda) = \langle \delta, \lambda \rangle \;\;.
\end{equation}
In particular any two elements $\lambda$,
$\mu$ in $\Lambda^{+}\subset W$ satisfy $l(\lambda \cdot \mu)=
l(\lambda) + l(\mu)$ (where the additive law of $\Lambda$, seen as a
subgroup of $W$, is written mutiplicatively). Thus the equation~(\ref{hi1}) implies the relation:
\begin{equation} \label{hec}
T_{\lambda} T_{\mu} = T_{\lambda \cdot \mu} \;\;.
\end{equation}

\begin{rem}
Equality~(\ref{hec}) is stated in \cite[(1.15)]{Heckop} for the Iwahori-Hecke
algebra of a split adjoint group, but the proof generalizes to our setting.
\end{rem}

Let $m$, $i_{1}$, $i_{2}$ as in the statement of the proposition and denote
by $g$ the element $i_1 m  i_2 \in \G(\QQ_{l})$. As $r(g)
=\ordre_{\MM}(m)$ belongs to $\Lambda^+$ it follows from (\ref{hec})
that:
$$ r(g^{n}) = n \cdot r(g) = n \cdot \ordre_{\MM}(m) \;\;.$$
This implies that the chamber $\rho_{\A, C}(g^n C) = n \cdot
\ordre_{\MM}(m) +  C$ leaves
any compact of $\A$ as $n$ tends to infinity. As a corollary the chamber $g^n C$
of $\X$ also leaves any compact of $\X$
when $n$ tends to infinity. This proves that the group $g^{\ZZ}$ is not
contained in a compact subgroup of $\G(\QQ_{l})$.
\end{proof}

\subsubsection{Lifting}
Recall that the notion of ``good position'' was defined in section~\ref{good_position}.
The following lemma controls uniformly the lifting to an Iwahori level and to
the intersection of two Iwahori
subgroups both contained in a given special maximal compact subgroup:

\begin{lem}\label{Iwahori}
Let $\G$ be a reductive $\QQ$-group.
\begin{itemize}

\item[$(a)$] For any prime $l$, any $\QQ_l$-split torus
$\TT \subset \G_{\QQ_{l}}$ and any maximal compact subgroup
$K_l \subset \G(\QQ_{l})$ in good position with respect to $\TT$,  
there exists an Iwahori subgroup $\Iw_l$ of $K_l$ in good position with respect to $\TT$.

\item[$(b)$] There exists an integer $f$ such that for any reductive $\QQ$-subgroup
$\HH \subset \G$ and any prime $l$ such that $\HH_{\QQ_{l}}$ is 
$\QQ_l$-isotropic the following holds~:
\begin{itemize}
\item[(i)] for any maximal compact subgroup $K_l$ of $\HH(\QQ_l)$, any Iwahori subgroup
$\Iw_l \subset K_l$ is of index $[K_l : \Iw_l]$ smaller than $l^f$.
\item[(ii)] for any maximal compact subgroup $K_l$ of $\HH(\QQ_l)$,
  any Iwahori subgroup $\Iw^1_l$ of $K_l$ and any Iwahori subgroup
  $\Iw^2_l$ of $\HH(\QQ_l)$ such that both $\Iw^1_l$ and $\Iw^2_l$ are
  contained in a common special maximal compact subgroup, the index $[K_l : \Iw^1_l \cap \Iw^2_l] $ is smaller than $l^f$.
\end{itemize}
\end{itemize}
\end{lem}

\begin{proof}
To prove $(a)$ let $l$, $\TT$ and $K_l$ be as in the statement. Choose a
maximal split torus $\AA$ of $\G_{\QQ_{l}}$ containing $\TT_{\QQ_{l}}$, denote by
$\MM$ the centraliser of $\AA$ in $\G_{\QQ_{l}}$ and choose any
minimal parabolic $\PP$ of $\G_{\QQ_{l}}$ with Levi $\MM$. Let $\A$
be the apartment of the Bruhat-Tits building $\X$ of $\G_{\QQ_{l}}$
associated to $\AA$, let $x \in \A$ be
the unique point of $\X$ fixed by the maximal compact subgroup $K_l$,
and let $\Cone \subset \A$ be any Weyl chamber containing $x$ whose
stabiliser at infinity in $\G(\QQ_l)$ is $\PP(\QQ_{l})$. Let $C$ be
the unique chamber of
$\Cone$
containing $x$ in its closure. Then by construction the Iwahori
subgroup $\Iw_l\subset K_l$ fixing $C$
satisfies that $\Iw_l \cap \AA(\QQ_{l})$ is the maximal compact open
subgroup of $\AA(\QQ_l)$. In particular $\Iw_l\cap \TT(\QQ_l)$
is the maximal compact open subgroup of $\TT(\QQ_l)$.

To prove $(b)(i)$: first notice that among maximal compact subgroups of
$\HH(\QQ_l)$ the hyperspecial ones have maximal volume,
cf. \cite[3.8.2]{tits}. Thus one can assume that $K_l$ is
hyperspecial. In this case the index $[K_l :
\Iw_l]$ coincides with $\sum_{w \in W_0} q_w$ where $W_0$ denotes the
finite Weyl group of $\HH_{\QQ_{l}}$ and $q_w$ denotes $[\Iw_l w \Iw_l: \Iw_l]$ for
$w \in W_0$. With the notations of \cite[section 3.3.1]{tits} for a
reduced word $w= r_1 \cdots r_j \in W_0$ one has $q_w =
l^d$ with $d= \sum_{i=1}^j d(\nu_i)$, where $\nu_i$ denotes the vertex
of the local Dynkin diagram of $\HH_{\QQ_{l}}$ corresponding to the reflection
$r_i$. As the cardinality of $W_0$ and its length function are bounded
when $\HH$ ranges through reductive $\QQ$-subgroups of $\G$ and $l$ ranges
through prime numbers we are reduced to prove that for any positive
integer $r$ there exists a positive integer $s$ such that $d(\nu_i) \leq s$ for any local Dynkin
diagram of rank at most $r$. This follows from inspecting
the tables in \cite[section 4]{tits}.

To prove $(b)(ii)$ : notice that $$[K_l: \Iw^1_l \cap \Iw^2_l]=[K_l:
\Iw^1_l] \cdot [\Iw^1_l: \Iw^1_l \cap \Iw^2_l]\;\;.$$
As $\Iw^1_l$ and $\Iw^2_l$ are both Iwahori
subgroups of a special maximal compact subgroup $K^{\text{m}}_l$ of
$\HH(\QQ_l)$ the index 
$[\Iw^1_l: \Iw^1_l \cap \Iw^2_l]$ is bounded by $[K^{\text{m}}_l:
\Iw^2_l] =|W_0|$. As the cardinality of $W_0$ is bounded
when $\HH$ ranges through reductive $\QQ$-subgroups of $\G$ and $l$ ranges
through prime numbers, statement $(b)(ii)$ follows from statement
$(b)(i)$ (up to a change of the constant $f$).
\end{proof}

\subsection{A uniformity result.}

The purpose of this section is to prove the following uniformity result:

\begin{prop} \label{uniform}
Let $(\G', X')$ be a Shimura datum with $\G'$ semi-simple of adjoint
type and $X'^+$ a connected component of
$X'$. 
Let $A$ be the positive integer defined in \cite{UllmoYafaev},
proposition 2.9. Then the following holds.

Let $(\G, X)$ be a Shimura subdatum of $(\G', X')$ and $X^+$ a connected
component of $X$ contained in $X'^+$. Let $K \subset \G(\AAf)$ be a neat open compact subgroup of $\G(\AAf)$.

Let $V$ be a special subvariety of $S_K(\G, X)_\CC$ which is not strongly special.
Let $(\HH_V, X_{V})$ be a Shimura datum defining $V$, denote by
$\TT_V$ its connected centre. Let $l$ be a prime splitting $\TT_{V}$ and $E_V$ the
reflex field of $(\HH_V, X_V)$.
For any $m$ in $\TT_{V}(\QQ_l)$ its power $m^A$ satisfies the condition that for some
$\sigma \in \Gal(\ol \QQ/E_V)$ the following inclusion holds in $\Sh_K(\G,X)_\CC$:
$$
\sigma(V) \subset T_{m^A}(V)\;\;.
$$

\end{prop}

\begin{proof}
Let $V$ and $m$ be as in the statement. For simplicity we write $\HH$
for $\HH_V$. We refer to section 2.1 of \cite{UllmoYafaev} for details
and notations
on reciprocity morphisms.

By Proposition 2.9 of \cite{UllmoYafaev} the image of $m^A$ in $\ol{\pi_0}\pi(\HH)$ is of the form 
$r_{(\HH,X_{\HH})}(\sigma)$ for some $\sigma\in \Gal(\ol \QQ/E_V)$.

The variety $V$ is the image of $X^+_{\HH} \times \{1\}$ in $\Sh_{K}(\G,X)$.
Let $\sigma$ be the element of $\Gal(\ol \QQ/E_V)$ as above.
By definition of the Galois action on the set of connected components of
a Shimura variety, we get
$$
\sigma(V) = \ol{X^+_{\HH} \times \{m^A\}} \subset T_{m^A}V
$$
where $\ol{X^+_{\HH} \times \{m^A\}}$ stands for the image of $X^+_{\HH} \times \{m^A\}$
in $\Sh_K(\G,X)_\CC$.
\end{proof}

\subsection{Proof of theorem~\ref{good Hecke}.}
Let $\G'$, $X'$, $X'^{+}$, $K'$, $\rho$, $\G$, $X$, $V$ and $l$ be as
in theorem~\ref{good Hecke}.

\subsubsection{Definition of $m$}
As $V$ is special but not strongly special, the torus $\TT_V^{\ad}  :=
\lambda(\TT_V)$ is a non-trivial torus in $\G^{\ad}$, where $\lambda :
\G \lto \G^\ad$ denotes the natural morphism.

As $K'_l = \G'(\ZZ_l)$ the compact subgroup $K_l = K'_l \cap
\G(\QQ_l)$ of $\G(\QQ_l)$ contains
$\G(\ZZ_l)$. In particular for any element $m \in \TT_V(\QQ_l)$ one
has the inequality: 
\begin{equation} \label{ineg1}
[K_l : K_l \cap mK_l m^{-1}] \leq [K_l: K_l \cap m \G(\ZZ_l) m^{-1}]
\;\;.
\end{equation}

By lemma 2.6 of \cite{UllmoYafaev} the coordinates of the characters
of $\TT_V$ intervening in the representation $\rho_{|\TT_V} : \TT_V
\lto \GL_n$ with respect to a suitable $\ZZ$-basis of $X^*(\TT_V)$ are
bounded uniformly on $V$. By assumption the reduction
$(\TT_V)_{\FF_l}$ is a torus, hence
$(\TT_V)_{\ZZ_l}$ is also a torus by lemma~$3.3.1$ of \cite{EdYa}. 
Thus we can apply proposition 7.4.3 of \cite{EdYa} 
for $r=1$, $q_1= \lambda_{|\TT_V}: \TT_V \lto \TT_V^\ad$ and $e=A$
(the positive integer given by proposition~\ref{uniform}): there exists a 
constant $k_1$ depending only on $\G'$, $X'$ and $K'$, and an element $m \in \TT_V(\QQ_l)$ 
such that $\lambda(m)$  does not lie  in a compact subgroup of
$\TT_V^\ad(\QQ_l)$ and satisfies 
\begin{equation} \label{ineg2}
[K_l : K_l \cap m^A \G(\ZZ_l) m^{-A}]
<l^{k_1}\;\;.
\end{equation}

\subsubsection{Definition of $I_l$}
As $l$ splits $\TT_V$ and $(\TT_V)_{\FF_l}$ is a
torus, the group $\G(\ZZ_l)$, and thus also $K_l$, is in good position
with respect to $\TT_V$ by lemma~\ref{splitgood}.

Let $f$ be the constant defined in
lemma~\ref{Iwahori}, (b) (for the ambient group $\G'$).
We claim that there exists an Iwahori subgroup $\Iw_l^1$ of $\G(\QQ_l)$
such that $[K_l : K_l \cap \Iw^1_l] < l^f$. Indeed let $K^1_l$ be any maximal compact subgroup of
$\G(\QQ_l)$ containing $K_l$. As $K_l$ is in good position with
respect to $\TT$ the group $K_l^{1}$ too. By
lemma~\ref{Iwahori}(b)(i) there exists an Iwahori subgroup $\Iw^1_l \subset
K^{1}_l$ in good position with respect to $\TT_V$ and satisfying
$[K_l^{1}: \Iw_l^1] <l^f$. This implies $[K_l : K_l \cap \Iw^1_l]
< l^f$ as required.

Let $\mathbf{A}$ be a maximal split torus of   
$\G_{\QQ_{l}}$ containing $\TT_{V, \QQ_l}$ and such that $\Iw_l^1$ is
in good position with respect to $\mathbf{A}$. Let $\MM$ be its
centralizer in $\G_{\QQ_{l}}$.
Choose $K^{\text{m}}_l$ a special maximal compact subgroup containing
$\Iw^1_l$. Let $\X$ be the Bruhat-Tits building of $\G_{\QQ_{l}}$.
Denote by $\A\subset \X$ the apartment fixed by
$\mathbf{A}$, by $x \in \A$ the unique special vertex fixed by
$K^{\text{m}}_l$ and by $C^1$ the unique chamber of $\A$ fixed by $\Iw^1_l$. The
vertex $x$ lies in the closure of $C^1$. The vector $\ordre_{\MM}(m) \in \Lambda:=
\ordre_{\MM}(\MM(\QQ_{l}))$ is non-trivial. Let $\Cone \subset \A$ be
a Weyl chamber of $\A$ with apex $x$ such that $C^1 +
\ordre_{\MM}(m) \subset \Cone$. In particular:
\begin{equation} \label{positiv}
\ordre_{\MM}(m) \in \Lambda^+ 
\setminus\{0\}\;\;,
\end{equation}
where $\Lambda^+ \subset \Lambda$ denotes the
positive cone associated to the Weyl chamber $\Cone$.

Finally let $\Iw^2_l$ be the Iwahori
subgroup of $K^{\text{m}}_l$ fixing the unique chamber of $\Cone$ with apex
$x$. As $\Iw^2_l$ is the fixator of a chamber
of $\A$ it is in good position with respect to $\mathbf{A}$, hence
also with respect to $\TT_V$.

\begin{defi}
We define $I_l := \Iw^1_l \cap \Iw^2_l \cap K_l$.
\end{defi}

\begin{rem}
Lifting to the Iwahori level $I_l$ chosen as above will enable us to apply
proposition~\ref{noncompact}, as the Iwahori $\Iw^2_l$ is in the
required position with respect to $m$. The definition of $I_l$ is
simpler in the case where $K_l$ is hyperspecial. In this case
necessarily $K^1_l= K^{\text{m}}_l = K_l$ and we can take $\Iw^1_l =
\Iw^2_l$. Moreover the choice of $\Iw^2_l$ is unique if $m$ is
regular.
\end{rem}

\subsubsection{End of the proof} Let us show that the uniform constants $k = (k_1+f)$ and
$f$, the open subgroup $I_l$ and the element $m^A \in \TT_V(\QQ_l)$ satisfy the conclusions of the
theorem~\ref{good Hecke}.

As the groups $K_l$, $\Iw^1_l$ and $\Iw^2_l$ are in good position with respect to
$\TT_V$, the group $I_l$ is also in good position with respect to
$\TT_V$. As $\Iw^1_l$ and $\Iw^2_l$ are both contained in the
special maximal compact subgroup $K^{\text{m}}_l$, the
lemma~\ref{Iwahori}(b)(ii) implies the following inequality:
\begin{equation} \label{ineg3}
[K_l : I_l] = [K_l : K_l  \cap \Iw^1_l \cap \Iw^2_l] \leq [K^1_l:
\Iw^1_l \cap \Iw^2_l] < l^f \;\;.
\end{equation}
This is condition~$(1)$ of theorem~\ref{good Hecke}.

By proposition \ref{uniform} there exists $\sigma \in \Gal(\ol \QQ/ E_V)$
such that $\sigma(\wt{V}) \subset T_{m^A} \wt{V}$: this is condition~$(2)$ of theorem~\ref{good Hecke}.

Let $\AA^{\ad}$ be the maximal split torus $\lambda(\mathbf{A})$ of $\G^\ad_{\QQ_l}$, denote by $\MM^\ad :=\lambda(\MM)$ its centralizer in
$\G^\ad_{\QQ_{l}}$, let $\Iw_l$ be the Iwahori $\lambda(\Iw^2_l)$ of
$\G^\ad(\QQ_l)$, let $C^\ad$ be the unique chamber of the Bruhat-Tits building
$\X^\ad$ of $\G^\ad_{\QQ_{l}}$ fixed by $\Iw_l$ and $x^\ad$ the vertex in the
closure of $C^{\ad}$ fixed by $\lambda(K^{\text{m}}_l)$. Finally let
$\Cone^\ad \subset \A^\ad$ be the unique Weyl chamber with apex $x^{\ad}$
and containing $C^{\ad}$ and $\Lambda^{\ad, +} \subset
\Lambda^\ad:=\ordre_{\MM^\ad}(\MM^\ad(\QQ_l))$ the associated positive
cone. It follows from (\ref{positiv}) that
$\ordre_{\MM^\ad}(\lambda(m))$ lies in $\Lambda^{\ad,+}$; it is non-zero as $\lambda(m)$  does not lie  in a compact subgroup of
$\TT_V^\ad(\QQ_l)$ (hence of $\mathbf{A}^{\ad}(\QQ_l)$). Hence also $\ordre_{\MM^\ad}(\lambda(m^A))$ belongs
to $\Lambda^{\ad,+} \setminus\{0\}$. It follows from the
proposition~\ref{noncompact} that for any $k_1$, $k_2$ in $\Iw^2_l$
(in particular for any $k_1$, $k_2$ in $I_l$) the image of $k_1 m^A  k_2$ in $\G^\ad(\QQ_l)$ generates an unbounded
subgroup of $\G^\ad(\QQ_l)$. This is condition~$(3)$ of theorem~\ref{good Hecke}.

Finally from the inequalities~(\ref{ineg1}), (\ref{ineg2})
  and (\ref{ineg3}) one deduces:
\begin{equation} \label{ineg4}
\begin{split}
[I_l: I_l \cap m^A I_l m^{-A}] & = [I_l : I_l \cap m^AK_l m^{-A}] \cdot [
I_l \cap m^AK_l m^{-A} :  I_l \cap m^A I_l m^{-A}] \\
& \leq [K_l : K_l \cap
m^AK_l m^{-A}] \cdot [K_l : I_l] \\
& \leq  [K_l: K_l \cap m^A \G(\ZZ_l) m^{-A}] \cdot [K_l : I_l] \leq
l^{k_1 +f}= l^k\;\;.
\end{split}
\end{equation}
This is condition~$(4)$ of theorem~\ref{good Hecke}.

This finishes the proof of theorem~\ref{good Hecke}.



\section{Conditions on the prime~$l$.} \label{le_premier}

In this section, we use theorem~\ref{GaloisOrbits},
theorem~\ref{theor2} and theorem~\ref{good
  Hecke} to show (under one of the assumptions of the theorem~\ref{main-thm1}) that the
existence of a prime number $l$ satisfying certain conditions forces a
subvariety $Z$ of $\Sh_K(\G,X)_\CC$ containing a 
special subvariety $V$ which is not strongly special to contain a special subvariety $V'$
containing $V$ properly. 


\subsection{Situation.} \label{situation} 
We will consider the following set of data:

Let $(\G', X')$ be a Shimura datum with $\G'$ semi-simple of adjoint
type and let  ${X'}^+$ a
connected component of $X'$. We fix $R$, as in definition~\ref{alpha}
for $\G'$, $X'$ and $X'^+$, a uniform bound on the degrees of  the
Galois closures of the fields $E(\HH,X_{\HH})$ 
with $(\HH,X_{\HH})$ ranging through the Shimura subdata of $(\G',X')$.

Let $K' = \prod_{p \; \textnormal{prime}} K'_p$ be a neat compact open subgroup of $\G'(\AAf)$.
We fix a faithful representation $\rho\colon \G' \hookrightarrow
\GL_n$ such that $K'$ is contained in $\GL_n(\Zhat)$.
We suppose that with respect to $\rho$,
the group $K'_3$ is contained in the principal congruence subgroup of
level three of $\GL_n(\ZZ_3)$.

Fix $N$ be a positive integer, let $B$ and $C(N)$ be the constants from the theorem \ref{GaloisOrbits},
$k$ the positive integer defined in theorem~\ref{good Hecke}  for the data $\G'$,
$X'$, $X'^{+}$ and $K'$, and $f$ 
the positive integer defined in theorem~\ref{good Hecke} for the data $\G'$,
$X'$ and $X'^{+}$.

Consider an infinite  set $\Sigma$ of special subvarieties of
$S_{K'}(\G', X')_\CC$. For each $W$ in $\Sigma$, we let $(\HH_W,X_{W})$
be a Shimura subdatum of $(\G', X')$ defining 
$W$. Let $\TT_W$ be the connected centre of $\HH_W$ and $\alpha_W$,
$\beta_W$ be as in definitions \ref{notation} and \ref{alpha}.

\begin{rem}
Let $(\G,X)$ be a Shimura subdatum of $(\G',X')$, define $K=K'\cap
G(\AAf)$ and choose $X^+$ a connected 
component of $X$ contained in ${X'}^+$. Let
$
p\colon \Sh_K(\G, X)_\CC \lto \Sh_{K'}(\G',X')_\CC
$ be the natural morphism. If $V \subset S_K(\G, X)_\CC$ is a special subvariety which is an
irreducible component of $p^{-1}(W)$ for some $W \in \Sigma$, then $V$
is still defined by the Shimura subdatum $(\HH_V:=\HH_W, X_V:=X_W)$ of
$(\G, X)$. Accordingly we have
$\TT_V= \TT_W$, $\alpha_V= \alpha_W$, $\beta_V= \beta_W$.
\end{rem}

\subsection{The criterion.}
We can now state the main result of this section:

\begin{theor} \label{criterium_for_l}

Let $\G'$, $X'$, $X'^+$, $R$, $K'$, $N$, $k$, $f$ and $\Sigma$ as in  the
situation~\ref{situation}.

We assume either the GRH or that the tori $\TT_W$ lie in
one $\GL_n(\QQ)$-orbit as $W$ ranges through $\Sigma$.

Let $(\G,X)$ be a Shimura subdatum of $(\G',X')$ with reflex field
$F_\G:=E(\G,X)$. Define $K=K'\cap G(\AAf)$ and choose $X^+$ a connected 
component of $X$ contained in ${X'}^+$. Let 
$
p\colon \Sh_K(\G, X)_\CC \lto \Sh_{K'}(\G',X')_\CC
$ be the natural morphism. 

Let $W \in \Sigma$, let $V \subset S_K(\G, X)_\CC$ be an irreducible component of
$p^{-1}(W)$ and let $Z$ be a Hodge generic $F_\G$-irreducible
subvariety of $\Sh_{K}(\G,X)_\CC$ containing $V$.

Define $r:= \dim Z - \dim V$ and suppose $r>0$. 
Suppose moreover that $V$ and $Z$
satisfy the following conditions:

\begin{itemize}
\item[(1)]
the variety $V$ is special but
not strongly special in $\Sh_{K}(\G,X)_\CC$.

\item[(2)]
there exists a prime $l$ such that $K'_l$ is a hyperspecial maximal
compact open subgroup in $\G'(\QQ_l)$ which coincides with
$\G'(\ZZ_l)$, the prime $l$ splits $\TT_V$, the reduction $(\TT_{V})_{\FF_l}$
is a torus and the following inequality is satisfied:
\begin{equation} \label{la_condition}
l^{(k+2f)\cdot 2^{r}} \cdot (\deg_{L_{K}} Z)^{2^{r}} <
C(N)\alpha_V \beta_V^N \;\;.
\end{equation} 
\end{itemize}

Then $Z$ contains a special subvariety $V'$ that contains $V$
properly.
\end{theor}

\begin{proof}

The proof of theorem~\ref{criterium_for_l} proceeds by induction on $r
= \dim Z - \dim V >0$. For simplicity we denote $d_Z := \deg_{L_K} Z$.

\subsubsection{Case $r=1$.} \label{edix}
Let $\G$, $X$, $X^+$, $K$, $F_\G$, $W$, $V$ and $Z$ as in
theorem~\ref{criterium_for_l} with $\dim Z - \dim V =1$.
The inequality~(\ref{la_condition})
for $r=1$ gives us:
\begin{equation} \label{la_condition_pour_r=1}
l^{2(k+2f)} \cdot d_Z^2 <
C(N)\alpha_V \beta_V^N\;\;.
\end{equation}

Let $I_l \subset K_l$ and $m \in \TT_V(\QQ_l)$ satisfying the
conclusion of theorem~\ref{good Hecke}. Let $I\subset K $ be the neat
compact open subgroup $K^{l}I_l$ of $\G(\AAf)$ and 
$
\tau \colon \Sh_I(\G,X)_\CC \lto \Sh_K(\G,X)_\CC
$
the finite morphism of Shimura varieties deduced from the inclusion $I
\subset K$. It follows from the condition~(1) in theorem~\ref{good
  Hecke} that the degree of $\tau$ is bounded above by $l^f$.

Let $\wt{V} \subset S_I(\G, X)_\CC$ be an irreducible component of the preimage
$\tau^{-1}(V)$, this is a special but not strongly special
subvariety of $S_I(\G,X)_\CC$ still defined by
the Shimura subdatum $(\HH_V, X_V)$. Let $E_V = E(\HH_V, X_V)$. Notice that $F_\G\subset E_V$.
By the projection formula stated in section \ref{degrees1}
and proposition~\ref{deg}(1),
 we have the
inequality 
$$\deg_{L_{I}}(\Gal(\ol \QQ/E_V)\cdot \wt{V}) \geq
\deg_{L_{K}}(\Gal(\ol \QQ/E_V) \cdot V)\;\;.$$ 
By the corollary~\ref{compare-degrees} the following inequality holds~:
$$ \deg_{L_{K}}(\Gal(\ol \QQ/E_V) \cdot V) \geq \deg_{L_{K_{\HH_{V}}}} (\Gal(\ol \QQ/E_V) \cdot V) \;\;.$$
On the other hand, as $R$ satisfies the 
definition~\ref{alpha},  theorem~\ref{GaloisOrbits} applied to the special
subvariety $V$ of $S_K(\G, X)_\CC$ provides the following lower
bound:

$$  \deg_{L_{K_{\HH_{V}}}}(\Gal(\ol\QQ/E_V) \cdot V)  > C(N) \alpha_V \beta_V^N
\;\;.
$$
We thus obtain:
\begin{equation} \label{compareIwa}
\deg_{L_{I}}(\Gal(\ol \QQ/E_V)\cdot \wt{V}) > C(N)  \alpha_V \beta_V^N \;\;.
\end{equation}

Let $\wt{Z}$ be an $F_\G$-irreducible component of $\tau^{-1}(Z)$ containing
$\wt{V}$. In particular $\wt{Z}$ is Hodge generic in $\Sh_I(\G,X)_\CC$
and is the union of the $\Gal(\ol{F}/F_\G)$-conjugates of a geometrically irreducible component of $\tau^{-1}(Z)$.
The image of $\wt{Z}$ in $\Sh_K(\G,X)_{\CC}$ is $Z$ and as $\tau$ is of
degree bounded above by $l^f$ the following inequality follows from section~\ref{degrees1}:
\begin{equation} \label{degI}
\deg_{L_{I}} \wt{Z} \leq   l^f \cdot d_Z \;\;\;.
\end{equation}

As the morphism $\tau\colon \Sh_I(\G,X)_{\CC}\lto \Sh_K(\G,X)_{\CC}$ is
finite and preserves the property of a subvariety  
of being special, exhibiting a special subvariety $V'$ such that $V\subsetneq V'\subset Z$ 
is equivalent to exhibiting a special subvariety
$\wt{V'}$ such that $\wt{V}\subsetneq \wt{V'}\subset \wt{Z}$.

By conclusion~$(2)$ of theorem~\ref{good
  Hecke} there exists $\sigma \in \Gal(\ol \QQ/E_V)$ such that $ \sigma \wt{V}
\subset T_m\wt{V} \subset T_m\wt{Z}$. As $T_m \wt{Z}$ is defined over
$F_\G$ hence over $E_V$  we deduce that
$\wt{V} \subset T_m\wt{Z} \cap \wt{Z}$ and thus $\Gal(\ol \QQ/E_V)\cdot \wt{V} \subset \wt{Z} \cap
T_m \wt{Z}$.

If $\wt{Z}$ and $ T_m\wt{Z}$ have no common (geometric) irreducible component, then any
$\sigma(\wt{V})$, $\sigma \in \Gal(\ol \QQ /E_V)$, is an irreducible component of $\wt{Z} \cap T_m
\wt{Z}$ for dimension reasons. We get
\begin{equation} \label{fulton}
\begin{split}
 C(N)\alpha_V \beta_V^N \leq \deg_{L_{I}} (\Gal(\ol \QQ /E_V) \cdot \wt{V})
 & \leq \deg_{L_{I}} (\wt{Z} \cap T_m \wt{Z}) \\
 & \leq
(\deg_{L_{I}} \wt{Z})^2 [I_l : I_l \cap mI_l m^{-1}] < l^{k+2f} \cdot d_Z^2\;\;,
\end{split}
\end{equation}
where the first inequality on the left comes from the
inequality~(\ref{compareIwa}), the second from Bezout's theorem (as in \cite{Fulton}, Example 8.4.6) and
the last one from inequality~(\ref{degI}) and the condition~$(4)$ on
$m$ from theorem~\ref{good Hecke}.
This contradicts the inequality~(\ref{la_condition_pour_r=1}). 
Therefore, the intersection $\wt{Z}\cap T_m\wt{Z}$ is not proper and, as both $\wt{Z}$ and
$T_m\wt{Z}$ are defined over $F_\G$ and $\wt{Z}$ is $F_\G$-irreducible, we have $\wt{Z} \subset T_m \wt{Z}$. 
 
As $m$ also
satisfies condition~$(3)$ of theorem~\ref{good Hecke}, we can apply
theorem~\ref{theor2} to this $m$: there exists $\wt{V'}$ special
subvariety of $\wt{Z}$ containing $\wt{V}$ properly.

\subsubsection{Case $r>1$.} \label{r>1}

Fix $r>1$ an integer and suppose by induction that the conclusion of
theorem~\ref{criterium_for_l} holds for all Shimura subdata $(\G,X)$
of $(\G', X')$, connected components $X^+$ of $X$ contained in $X'^+$,
compact open subgroups $K= K' \cap \G'(\AAf)$, varieties $W\in
\Sigma$, and subvarieties $V$ and $Z$ of $\Sh_K(\G, X)_\CC$ as in the statement of
theorem~\ref{criterium_for_l}, satisfying moreover 
$0<\dim Z - \dim V <r$.

Now let $\G$, $X$, $X^+$, $K$, $F$, $W$, $V$ and $Z$ satisfying the
assumptions of theorem~\ref{criterium_for_l} with $\dim Z = \dim V +r$. 
Let $I$, $m$, $\wt{V}$ and $\wt{Z}$ be constructed as in the case
$r=1$. In particular the inequalities~(\ref{compareIwa}) and (\ref{degI}) still hold.

\sspace
Suppose that $\wt{Z}\subset T_m\wt{Z}$. In this case we can apply
theorem~\ref{theor2} with this $m$: there exists $\wt{V'}$ special
subvariety of $\wt{Z}$ containing $\wt{V}$ properly. This implies that
there exists $V'$ special subvariety of $Z$ containing $V$ properly.

\sspace
Suppose now that the intersection $\wt{Z}\cap T_m\wt{Z}$ is proper.
The same argument as in the case $r=1$ shows that this is equivalent to $\wt{Z}$ not being contained in $T_m\wt{Z}$.
 As the intersection $\wt{Z}\cap T_m\wt{Z}$ contains $\wt{V}$, we choose
an $F_\G$-irreducible component $\wt{Y}\subset \Sh_I(\G, X)_\CC$ of $\wt{Z}\cap T_m\wt{Z}$ containing
$\wt{V}$ and we denote by $Y$ its image in $\Sh_K(\G, X)_\CC$. Thus
$Y$ is $F_\G$-irreducible and satisfies $r_Y := \dim Y-\dim V< r$.
To show that $r_Y>0$ we need to check that $\wt{V}$ is not a component of $\wt{Z}\cap T_m(\wt{Z})$.
As $\wt{Z}\cap T_m(\wt{Z})$ is defined over $F_\G$ hence over $E_V$,
we have $
\Gal(\ol\QQ/E_V)\cdot \wt{V} \subset \wt{Z}\cap T_m(\wt{Z})
$. If $\wt{V}$ were a component of $\wt{Z}\cap
T_m(\wt{Z})$ by taking degrees and arguing as in the proof of
inequality~(\ref{fulton}) one still obtains:
$$
C(N)\alpha_V \beta_V^N < l^{k+2f}d_Z^2\;\;.
$$
This contradicts the condition (2). Hence $0 < r_Y <r$.

Let $(\HH,X_\HH)$ be a Shimura subdatum of $(\G, X)$ defining the
smallest special subvariety of $S_I(\G, X)_\CC$
containing a geometrically irreducible component of $\wt{Y}$
containing $\wt{V}$, let $X_\HH^+\subset X^+$ be the corresponding connected
component of $X_\HH$. We define $K_\HH:= K\cap \HH(\AAf)$ and $I_\HH := I \cap\HH(\AAf)$. 
We have the following commutative diagram:
$$
\xymatrix{
\Sh_{I_{\HH}} (\HH, X_\HH)_\CC \ar[r]^q \ar[d]_\tau & \Sh_{I}(\G,
X)_\CC \ar[d]^\tau \\
\Sh_{K_{\HH}} (\HH, X_\HH)_\CC \ar[r]_q  & \Sh_{K}(\G,
X)_\CC \;\;.
}
$$
Let $F_\HH$ be the reflex field $E(\HH,X_\HH)$ and let $\wt{V_\HH}$ be an
irreducible component of $q^{-1}(\wt{V})$ contained in
$S_{I_\HH}(\HH,X_\HH)_\CC$. We denote $V_\HH:=\tau(\wt{V_\HH})$ its
image in  $\Sh_{K_{\HH}}(\HH, X_\HH)_\CC$. Hence $V_\HH$ is also an
irreducible component of $(p \circ q)^{-1}(W)$.

Let $\wt{Y_\HH} \subset \Sh_{I_\HH}(\HH,X_\HH)_\CC $ be an
$F_\HH$-irreducible component of 
$q^{-1}(\wt{Y})$ containing $\wt{V_{\HH}}$. In particular $\wt{Y_\HH}$ is
an $F_\HH$-irreducible Hodge generic subvariety of $\Sh_{I_{\HH}}(\HH,
X_\HH)_\CC$. We define $Y_\HH:=\tau(\wt{Y_\HH})$, it is an
$F_\HH$-irreducible Hodge generic subvariety of $\Sh_{K_\HH}(\HH,
X_\HH)_\CC$.

Finally we have the commutative diagram of triples of varieties:
$$
\xymatrix{
(\Sh_{I_{\HH}} (\HH, X_\HH)_\CC, \wt{Y_\HH}, \wt{V_{\HH}})  \ar[r]^q \ar[d]_\tau & (\Sh_{I}(\G,
X)_\CC, \wt{Y}, \wt{V}) \ar[d]^\tau \\
(\Sh_{K_{\HH}} (\HH, X_\HH)_\CC, Y_\HH, V_\HH) \ar[r]_q  & (\Sh_{K}(\G,
X)_\CC, Y, V) \;\;.
}
$$
Notice that $Y_\HH$ satisfies
\begin{equation} \label{upper}
\deg_{L_{K_{\HH}}} Y_\HH \leq \deg_{L_{I_{\HH}}} 
\wt{Y_\HH} \leq \deg_{L_{I}} q(\wt{Y_\HH}) \leq \deg_{L_{I}} \wt{Y} \leq 
\deg_{L_{I}} (\wt{Z} \cap T_m\wt{Z})< l^{k+2f} d_Z^2 \;\;.
\end{equation}
Indeed, the inequality $\deg_{L_{K_{H}}} Y_\HH \leq \deg_{L_{I_{\HH}}}
\wt{Y_\HH}$ comes from section~\ref{degrees1}, the inequality  $\deg_{L_{I_{\HH}}} 
\wt{Y_\HH} \leq \deg_{L_{I}} q(\wt{Y_\HH})$ from
corollary~\ref{compare-degrees}, the inequality $\deg_{L_{I}}
q(\wt{Y_\HH}) \leq \deg_{L_{I}} \wt{Y}$ from the inclusion
$q(\wt{Y_\HH}) \subset \wt{Y}$, the inequality $\deg_{L_{I}} \wt{Y} \leq 
\deg_{L_{I}}(\wt{Z} \cap T_m\wt{Z})$ from the fact that $\wt{Y}$ is an
$F_\G$-irreducible component of $\wt{Z} \cap T_m \wt{Z}$, and the last
inequality on the right is proven as in (\ref{fulton}).

\begin{prop} \label{check-induction}
The data $\HH$, $X_\HH$, $X_\HH^+$, $K_\HH$, $F_\HH$, $W$, $V_\HH$ and $Y_\HH$ satisfy the conditions of theorem~\ref{criterium_for_l} (in
place respectively of $\G$, $X$, $X^+$, $K$, $F_\G$, $W$, $V$ and $Z$). 
\end{prop}

\begin{proof}
Let
$r_\HH:=\dim Y_\HH -\dim V_\HH$, thus $r_\HH = r_Y>0$.

We first check that  $\HH$, $X_\HH$, $X_\HH^+$, $K_\HH$, $F_\HH$, $W$,
$V_\HH$ and $Y_\HH$ satisfy condition~$(2)$ of theorem~\ref{criterium_for_l}, for
the same prime $l$. 
From the inequality~(\ref{upper}) we obtain:
$$
l^{(k+2f)\cdot 2^{r_\HH}} (\deg_{L_{K_{\HH}}} Y_\HH)^{2^{r_{\HH}}}
\leq
 l^{(k+2f) \cdot 2^{r_{\HH}+1}} d_Z^{2^{r_{\HH}+1}}
$$
and, as $r_\HH+1 \leq r$, we deduce from the inequality~(\ref{la_condition}) that
\begin{equation} \label{encore2}
l^{(k+2f)\cdot 2^{r_{\HH}}}(\deg_{L_{K_{\HH}}} Y_\HH)^{2^{r_{\HH}}}< C(N)\alpha_V\beta_V^N\;\;.
\end{equation}
This is condition~$(2)$ for $\HH$, $X_\HH$, $X_\HH^+$, $K_\HH$,
$F_\HH$, $W$, $V_\HH$ and $Y_\HH$. 

Proposition~\ref{check-induction} then follows from the
following lemma proving that $\HH$, $X_\HH$, $X_\HH^+$, $K_\HH$,
$F_\HH$, $W$, $V_\HH$ and $Y_\HH$ satisfy the condition~$(1)$ of theorem~\ref{criterium_for_l}.
\end{proof}

\begin{lem} \label{notspecial}
The special subvariety $V_\HH$ is not strongly special in
$\Sh_{K_\HH}(\HH,X_\HH)_\CC$.
\end{lem}
\begin{proof}
Suppose the contrary. Then $\TT_{V_\HH} (= \TT_{V})$ is contained in the connected centre $Z(\HH)^0$ of $\HH$
and by the lemma \ref{inclusion1},
$\TT_V = Z(\HH)^0$.
Recall that $K_{\TT_{V}}^{\mathrm{m}} $ denotes the maximal compact open subgroup of $\TT_V(\AAf)$ and
$K_{\TT_{V}} = \TT_V(\AAf)\cap K_{\HH}$.
Let
$
K_{\HH}^{\mathrm{m}} := K_{\TT_{V}}^{\mathrm{m}} K_{\HH}
$
and let 
$$
\pi \colon \Sh_{K_{\HH} }(\HH , X_{\HH})_\CC \lto 
\Sh_{K_{\HH}^{\mathrm{m}}}(\HH , X_{\HH})_\CC
$$
be the natural morphism.
Notice that $K_{\HH}^{\mathrm{m}}/K_{\HH}=K_{\TT_{V}}^{\mathrm{m}}/K_{\TT_{V}}$ acts transitively on the fibres of $\pi$. 
 Let $A$ be the positive integer defined by
\cite[prop.2.9]{UllmoYafaev} for Shimura subdata of $(\G', X')$
(notice that the
constant $A$
already appeared in proposition~\ref{uniform}). Let
$\Theta_A \subset K_{\TT_{V}}^{\mathrm{m}}/K_{\TT_{V}}$ be the 
image of the map $x\mapsto x^A$ on $K_{\TT_{V}}^{\mathrm{m}}/K_{\TT_{V}}$.

\begin{sublem}  \label{contient}
The orbit $\Theta_A V_{\HH}$ is contained
in $\Gal(\ol\QQ/E_V)\cdot V_{\HH}\cap \pi^{-1}\pi(V_{\HH})$. 
\end{sublem}

\begin{proof}
Let $f \colon \Sh_{K_{V_{\HH}}}(\HH_V, X_V)_\CC \lto \Sh_{K_{\HH}}(\HH, X_{\HH})_\CC$
be the morphism defining $V_{\HH}$, where $K_{V_{\HH}}:= K_\HH \cap
\HH_V(\AAf)$. It is naturally
$K_{\TT_{V}}^\textrm{m}/K_{\TT_{V}}$-equivariant and defined over
$\Gal(\ol\QQ/E_V)$.

Let $\mathcal{V}$ be the component of $\Sh_{K_{V_\HH}}(\HH_V, X_V)$ such that
$V_{\HH} = f(\mathcal{V})$. The
$K_{\TT_{V}}^\textrm{m}/K_{\TT_{V}}$-equivariance of $f$ implies:
$$\forall \; \theta \in \Theta_A, \quad f(\theta \cdot \mathcal{V}) = \theta \cdot V_{\HH}\;\;.$$

On the other hand, by the first claim of  \cite[lemma 2.15]{UllmoYafaev}
applied to the Shimura datum $(\HH_V, X_V)$, we see that
$$
\theta \cdot \mathcal{V} = \sigma \mathcal{V} 
$$
for some $\sigma \in \Gal(\ol\QQ/E_V)$. 
Hence 
$$
\theta \cdot V_{\HH} = f(\theta \cdot \mathcal{V}) = f(\sigma \mathcal{V}) = \sigma f(\mathcal{V}) = \sigma V_{\HH}\;\;.
$$
Hence the result.
\end{proof}

\begin{sublem} \label{iteration}
There exists a geometrically irreducible subvariety $Y'$ of $Y_{\HH}$
defined over $\ol\QQ$ and containing $V_{\HH}$
such that the following holds:
\begin{enumerate}
\item $\deg_{L_{K_{\HH}}}\Gal(\ol\QQ/E_V)\cdot Y' \leq (\deg_{L_{K_{\HH}}} Y_{\HH})^{2^{r_{\HH}}}$.
\item 
The variety
 $\Theta_A Y'$ is contained in $\Gal(\ol\QQ/E_V)\cdot Y'$.
\end{enumerate}
\end{sublem}
\begin{proof}
Let $Y_1$ be a geometrically irreducible component of $Y_\HH$ containing $V_{\HH}$.

If $\Theta_A Y_1$ is contained in $\Gal(\ol\QQ/E_V)\cdot Y_1$,
then take $Y'=Y_1$.
As $\Gal(\ol\QQ/E_V)\cdot Y'$ is contained in $Y_{\HH}$, the
condition~$(1)$ is obviously satisfied.

Otherwise there exists a $\theta \in \Theta_A$ such that $\theta Y_1$ is not a $\Gal(\ol\QQ/E_V)$-conjugate of
$Y_1$. Let $\ol{Y_1}= \Gal(\ol\QQ/E_V)\cdot Y_1$. Recall that the
action of $\Theta_A$ commutes with the action of $\Gal(\ol\QQ/E_V)$. In particular the
intersection $\ol{Y_1} \cap \theta \ol{Y_1}$ is proper. Moreover, as
$\theta V_\HH$ is a $\Gal(\ol\QQ/E_V)$-conjugate of $V_\HH$ by sublemma~\ref{contient}, we obtain that:
$$
V_{\HH}\subset \ol{Y_1} \cap \theta\ol{\cdot Y_1}\;\;.
$$

Let $Y_2$ be a geometrically irreducible component of $\ol{Y_1} \cap \theta \ol{Y_1}$ containing $V_{\HH}$ and let $\ol{Y_2} = \Gal(\ol\QQ/E_V)\cdot Y_2$.
We have
$$
\ol{Y_2}\subset \ol{Y_1}\cap \theta\ol{Y_1} \subset Y_{\HH}\cap \theta Y_{\HH}\;\;.
$$

It follows that
$$
\deg_{L_{K_{\HH}}}  \ol{Y_2} \leq
(\deg_{L_{K_{\HH}}}Y_\HH)^2 \;\;.
$$
On the other hand:
$$
\deg_{L_{K_{\HH}}} \Gal(\ol\QQ/E_V) \cdot V_{\HH} >
C(N)\alpha_V\beta_V^N > (\deg_{L_{K_{\HH}}}Y_\HH)^{2^{r_{\HH}}} 
$$
where the first left inequality follows from
theorem~\cite[Theorem 2.19]{UllmoYafaev} applied with $Y = V_{\HH}$
and the second one from~(\ref{encore2}).
These inequalities show that $\dim Y_2  > \dim V_{\HH}$.

We now iterate the process replacing $Y_1$ by $Y_2$. As $\dim V_\HH <
\dim Y_2 < \dim Y_1 =\dim Y_\HH$ after at most $r_{\HH} = \dim Y_\HH
- \dim V_\HH$ iterations we construct the variety $Y'$ satisfying the 
required conditions.
\end{proof}

We now finish the proof of lemma \ref{notspecial}.
Condition~$(2)$ of sublemma~\ref{iteration} enables us to apply theorem~\cite[theor.2.19]{UllmoYafaev}:
$$
\deg_{L_{K_{\HH}}} (\Gal(\ol\QQ/E_V)\cdot Y') \geq C(N)\alpha_V\beta_V^N\;\;.
$$

By sublemma~\ref{iteration} $(1)$, and inequality~(\ref{encore2}) we have
$$
\deg_{L_{K_{\HH}}} (\Gal(\ol\QQ/E_V)\cdot Y') \leq (\deg_{L_{K_{\HH}}} Y_\HH)^{2^{r_{\HH}}}
<C(N) \alpha_V \beta_V^N\;\;.
$$

These inequalities yield a contradiction. This finishes the proof
of lemma~\ref{notspecial}.
\end{proof}

Let us now finish the proof of theorem~\ref{criterium_for_l}.
As $r_{\HH} <r$ by induction hypothesis we can apply
theorem~\ref{criterium_for_l} to $\HH$, $X_\HH$, $X_\HH^+$, $K_\HH$,
$F_\HH$, $W$, $V_\HH$ and $Y_\HH$. Thus $Y_\HH$ contains a special subvariety $V'_\HH$ which
contains $V_\HH$ properly. This implies that $Z$ contains a special
subvariety $V'$ which contains $V$ properly.
This finishes the proof of theorem~\ref{criterium_for_l} by induction on $r$.
\end{proof}


\section{The choice of a prime $l$.} \label{choix}
In this section we complete the proof of the theorem~\ref{main-thm2},
and thus also of the main theorem~\ref{main-thm}, using
the theorem~\ref{criterium_for_l}. The choice of a prime $l$ satisfying the conditions
of the theorem~\ref{criterium_for_l}
will be made possible by the effective Chebotarev theorem, which we
now recall.

\subsection{Effective Chebotarev.}

\begin{defi}
Let $L$ be a number field of degree $n_L$ and absolute discriminant
$d_L$. Let $x$ be a positive real number. We denote by $\pi_L(x)$ the
number of primes $p$ such that $p$ is split in $L$ and $p\leq x$.
\end{defi}

\begin{prop} \label{chebotarev}
Assume the Generalized Riemann Hypothesis (GRH).
There exists a constant $A$ such that the following holds.
For any number field $L$ Galois over $\QQ$ and
for any $x> \max(A,2\log(d_L)^2 (\log(\log(d_L)))^2)$ we have
$$
\pi_L(x) \geq \frac{x}{3n_L \log(x)}\;\;.
$$
Furthermore, if we consider number fields such that $d_L$ is constant,  
then the assumption of the GRH can be dropped.
\end{prop}

\begin{proof}
The first statement (assuming the GRH) is proved in the Appendix N of  
\cite{EdHilbert}
and the second is a direct consequence of the classical Chebotarev theorem.
\end{proof}


\subsection{Proof of the theorem~\ref{main-thm2}.}

\begin{proof}
Let $\G$, $X$, $X^+$, $R$, $K$ and $Z$ be as in theorem~\ref{main-thm2}. Thus
$(\G,X)$ is a Shimura datum with $\G$ semisimple of adjoint type,
$X^+$ is a connected component of $X$, the positive integer $R$ is as
in definition~\ref{alpha}, the group $K=\prod_{p \; \textnormal{prime}} K_p$ is a neat compact open subgroup of
$\G(\AAf)$ and $Z \subset S_K(\G, X)_\CC$ is a Hodge generic
geometrically irreducible subvariety containing a Zariski dense set $\mathbf{\Sigma}$ of special subvarieties, which
is a union of special subvarieties $V$, $V \in \Sigma$, all of the same dimension
$n(\Sigma)$ such that for any
modification $\Sigma'$ of $\Sigma$ the set $\{ \alpha_V\beta_V, \; V \in \Sigma'\}$ is
unbounded.
We want to show, under each of the two assumptions $(1)$ or $(2)$ of
theorem~\ref{main-thm2} separately, that for every $V$ in
$\Sigma$ there exists a special subvariety $V'$ such that $V
\subsetneq  V' \subset Z$ (possibly after replacing $\Sigma$ by a modification).

From now on, we fix a faithful rational representation $\rho: \G \hookrightarrow 
\GL_n$ such that $K$ is contained in $\GL_n(\Zhat)$. In the case of the
assumption~$(2)$ in theorem~\ref{main-thm2}, we take for $\rho$ the representation
which has the property that the centres $\TT_V$ lie in
one $\GL_n(\QQ)$-conjugacy class (possibly replacing $K$ by $K\cap  
\GL_n(\Zhat)$) as $V$ ranges through $\Sigma$. 

\begin{lem} \label{lem10.2.1}
Without any loss of generality we can assume that:
\begin{enumerate}

\item The group $K_{3}$ is contained in the congruence subgroup of level three (with respect to the faithful representation $\rho$).

\item After possibly replacing $\Sigma$ by a modification, 
$\Sigma$ consists of special but not strongly special subvarieties.
\end{enumerate}
\end{lem}

\begin{proof}
To fulfill the first condition, let $\wt{K} = \wt{K_3} \times \prod_{p
  \not = 3} K_p$ be a finite index subgroup of $K$ with $\wt{K_3}$
contained in the congruence subgroup of level three (with respect to
the faithful representation $\rho$). Let $\wt{Z}$ be an irreducible component of the preimage of $f^{-1}(Z)$,
where $f: S_{\wt{K}}(\G,X)_\CC \lto S_K(\G,X)_\CC$ is the canonical
finite morphism. Then 
$\wt{Z}$ contains a Zariski-dense set
$\wt{\Sigma}$, which is a union of special subvarieties $V$, $V \in \wt{\Sigma}$, all of the same dimension
$n(\Sigma)$: $\wt{\Sigma}$ is the set of all irreducible components $\wt{V}$
of $f^{-1}(V)$ contained in $\wt{Z}$ as $V$ ranges through
$\Sigma$. Notice that for any modification $\wt{\Sigma}'$ of
$\wt{\Sigma}$ the set $\{\alpha_{V'}\beta_{V'}, \; V' \in
\wt{\Sigma}'\}$ is unbounded: $\beta_{V'} = \beta_{f(V')}$ and $\alpha_{V'}$ is equal to 
$\alpha_{f(V')}$ up to a factor independent of $V'$. Thus $\wt{Z}$
satisfies the assumptions of theorem~\ref{main-thm2}. As a subvariety of
$\Sh_{K}(\G, X)_{\CC}$ is special if and only if some (equivalently any) irreducible component of its preimage by $f$
is special, theorem~\ref{main-thm2} for $\wt{Z}$ implies
theorem~\ref{main-thm2} for $Z$.

For the second condition : otherwise there is a modification
$\Sigma'$ of $\Sigma$
consisting only of strongly special subvarieties. Contradiction with
the assumption that the set $\{ \alpha_V \beta_V, V \in \Sigma' \}$ is
unbounded. 
\end{proof}

Let $B$ be the constant depending on $\G$, $X$ and $R$ given by the theorem \ref{GaloisOrbits}.
Fix $N$ a positive integer and let $C(N)$ be the real number depending
on 
$R$ and $N$  given by the theorem \ref{GaloisOrbits}. 
Let $k$ the constant depending on the data $\G$, $X$,
$X^+$, $K$ defined in the theorem~\ref{good Hecke}. Let $f$ be
the constant depending on the data $\G$, $X$,
$X^+$ defined in the theorem~\ref{good Hecke}.

Let $F_\G$ be the reflex field $E(\G,X)$.
As $Z$ contains a Zariski dense set of special subvarieties, $Z$ is defined over $\ol\QQ$.
We replace $Z$ by the union of its conjugates under $\Gal(\ol\QQ/F_\G)$. Thus $Z$ is now an
$F_\G$-irreducible $F_\G$-subvariety of $\Sh_K(\G,X)_\CC$.

For all primes $l$ larger than a constant $C$, the group $K_l$ is a hyperspecial maximal compact open
subgroup of $\G(\QQ_l)$ (cf. \cite[3.9.1]{tits}) and furthermore $K_l = \G(\ZZ_l)$, where the
$\ZZ$-structure on $\G$ is defined by taking the Zariski closure in 
$\GL_{n,\ZZ}$ via $\rho$.
\begin{prop} \label{suffisant}
To prove theorem~\ref{main-thm2} it is enough to show that for any $V$ in $\Sigma$ (up to a
modification), there exists a prime $l>C$ satisfying the following
conditions:
\begin{enumerate}
\item the prime $l$ splits $\TT_V$.
\item $\TT_{V,\FF_l}$ is a torus.
\item 
$l^{(k+2f)\cdot 2^{r}} \cdot (\deg_{L_{K}} Z)^{2^{r}} <
C(N)\alpha_V \beta_V^N\;\;, \qquad \text{where} \; r = \dim Z - \dim V\;\;.$
\end{enumerate}
\end{prop}

\begin{proof}
Let $V$ be an element of $\Sigma$.

\sspace
Let us check that the conditions of the theorem~\ref{criterium_for_l}
are satisfied for $\G= \G'$, $X$, $X^+$, $K$, $F_\G$, $W=V$ and $Z$:
\begin{itemize}
\item[-] condition~(1) of theorem~\ref{criterium_for_l} is automatically satisfied because $\G = \G'$
  and $\Sigma$ consists of special but non strongly special
  subvarieties of $\G$ by lemma~\ref{lem10.2.1}$(2)$.
\item[-] the conditions of
proposition~\ref{suffisant} immediately imply that the
condition~$(2)$ of theorem~\ref{criterium_for_l} is
satisfied. 
\end{itemize}

\sspace
As the set $\{\alpha_V\beta_V, V \in \Sigma\}$ is unbounded the
difference $r:= \dim Z - n(\Sigma)$ is necessarily positive.
We now apply the theorem  \ref{criterium_for_l}: for any $V$ in
$\Sigma$ there exists a special subvariety $V'$ of $S_K(\G, X)_\CC$
such that $ V \subsetneq V' \subset Z$. 
\end{proof}

Therefore, in order to prove theorem~\ref{main-thm2}, it remains to check the existence of the prime $l$ satisfying the 
conditions of proposition~\ref{suffisant}. We first prove the following.

\begin{prop} \label{choice_l}
For every $D>0$, $\epsilon > 0$ and every integer
$m \geq \max(\epsilon,6)$, there exists  an integer $M$
such that (up to a modification of $\Sigma$):
for every $V$ in $\Sigma$ with $\alpha_V\beta_V$
larger than $M$ there exists a prime $l>C$ satisfying the following conditions:
\begin{enumerate}
\item the prime $l$ splits $\TT_V$.
\item $(\TT_{V})_{\FF_l}$ is a torus.
\item $l < D\alpha_{V}^{\epsilon} \beta_{V}^{m}$.
\end{enumerate}

\end{prop}

\begin{proof} 
For $V$ in $\Sigma$ recall that $n_V$ is the degree of the  
splitting field $L_V$ of $\C_{V}= \HH_{V} /\HH_{V}^{\textnormal{der}}$
over $\QQ$. By the proof of \cite[lemma 2.5]{UllmoYafaev} the number $n_V$ is bounded above by some
positive integer $n$ as $V$ ranges through $\Sigma$. 

Fix $D>0$, $\epsilon > 0$ and $m \geq \max(\epsilon, 6)$.
For $V$ in $\Sigma$, let
$$
x_V := D \alpha_V^{\epsilon} \beta_V^{m}\;\;.
$$

\begin{lem} \label{penultieme}
Up to a modification of $\Sigma$ the following inequality holds for
every $V$ in $\Sigma$:
\begin{equation} \label{lowerpi}
\pi_{L_{V}}(x_V) \geq \frac{D^{\frac{1}{2}}}{3n} \cdot  \alpha_V^{\frac{\epsilon}{2}} \cdot   
\beta_V^{\frac{m}{2}} \;\;.
\end{equation}
\end{lem}

\begin{proof}
As we are assuming either the GRH, or that the connected centres
$\TT_V$ of the generic Mumford-Tate groups $\HH_V$ of $V$ lie in one  
$\GL_n(\QQ)$-conjugacy class under $\rho$ as $V$ ranges through
$\Sigma$, in which case $d_{L_{V}}$ is independent of $V$, we can
apply proposition~\ref{chebotarev}: 
$$
\pi_{L_{V}}(x_V) \geq \frac{x_V}{3 n \log(x_V)}
$$
provided that $x_V$ is larger than some absolute constant and $\beta_V^3$.
Notice moreover that if $x_V\geq 4$ then
$
\sqrt{x_V} \geq \log(x_V)
$. 

It follows that
$$
\pi_{L_{V}}(x_V) \geq \frac{\sqrt{x_V}}{3 n} =
\frac{D^{\frac{1}{2}}}{3n} \cdot  \alpha_V^{\frac{\epsilon}{2}} \cdot   
\beta_V^{\frac{m}{2}} \;\;,
$$
hence the result, provided that $x_V$ is larger than some absolute constant and
$\beta_V^3$ for all $V$ in $\Sigma$.

It remains to show that (up to a modification of $\Sigma$) the quantity
$x_V$ is larger than any given constant and
$\beta_V^3$ for all $V \in \Sigma$. As $\alpha_V \beta_V$ is unbounded as $V$ ranges through any
modification of $\Sigma$, we can assume up to a modification of
$\Sigma$ that $\beta_V$ is non-zero for all $V$ in $\Sigma$. As
$\beta_V$ is the logarithm of a positive integer there exists $b>0$
such that $\beta_V >b$ for all $V$ in $\Sigma$. Then up to a
modification of $\Sigma$:

- either the inequality $\beta_V \leq
1$ holds for all $V
\in \Sigma$. In this case the assumption that $\alpha_V \beta_V$ is
unbounded as $V$ ranges through any modification of $\Sigma$ implies that $\alpha_V$ can
be ensured to be larger than any given constant (hence also larger than any given constant and $\beta_V^3$) for all $V$ in
$\Sigma$ and we are done.

- or the inequality $\beta_V >1$
holds for all $V \in \Sigma$.
On the one hand $m \geq \varepsilon$ hence we have $x_V = D (\alpha_V
\beta_V)^\varepsilon \beta_V^{m-\varepsilon} \geq D (\alpha_V
\beta_V)^\varepsilon$. On the
other hand as $m \geq 6$ one has $x_V \geq D (\alpha_V
\beta_V)^{\inf(\varepsilon, 3)} \beta_V^3$. Up to a new modification of
$\Sigma$ we can assume that for all $V$ in $\Sigma$ the quantity
$(\alpha_V \beta_V)^{\inf{(\varepsilon, 3)}}$ is larger than any
given constant
and we are also done in this case.

This finishes the proof of lemma~\ref{penultieme}.
\end{proof}

Let $i_V$ be the number of primes $p$ \emph{unramified} in $L_V$
such that $K^{\mathrm{m}}_{\TT_V,p}\not= K_{\TT_V,p}$.
To prove the proposition~\ref{choice_l} it is enough to show that
$\pi_{L_{V}}(x_V)> \max(C, i_V)$ if
$\alpha_V\beta_V$ is large enough.
Indeed, this will yield a prime $l>C$ satisfying $l <
D\alpha_{V}^{\epsilon} \beta_{V}^{m}$ and such that: $l$ is split in
$L_V$ and $K_{\TT_V,l} = K^{\mathrm{m}}_{\TT_V,l}$.
These last two conditions imply that $\TT_{V,\FF_l}$ is a torus
(we refer to the proof of lemma 3.17 of \cite{UllmoYafaev}
for the proof of this fact).


\begin{lem} \label{lowalpha}
Let $c$ be the uniform constant from the Proposition 4.3.9 of
\cite{EdYa}. Then:
$$
\alpha_V \geq (Bc)^{i_V}\cdot {i_V}!\qquad .
$$ 
\end{lem}
\begin{proof}
Notice that 
$$ \alpha_V= \prod_{\substack{p \, \textnormal{prime}\\
    K^{\mathrm{m}}_{{\TT_V},p}\not= K_{{\TT_V},p}}} \max(1, B \cdot
|K^{\mathrm{m}}_{\TT_V,p}/K_{\TT_V,p}|) 
\geq \prod_{\substack{p \, \textnormal{prime} \\ p \,
    \textnormal{unramified in} \, L_V \\
    K^{\mathrm{m}}_{{\TT_V},p}\not= K_{{\TT_V},p} }}  B \cdot
|K^{\mathrm{m}}_{\TT_V,p}/K_{\TT_V,p}|\;\;.$$ 
By proposition 3.15 of \cite{UllmoYafaev}, for $p$ unramified
in $L_V$ and such that $K_{\TT_V,p}^{\mathrm{m}}\not= K_{\TT_V,p}$ we have 
$
|K^{\mathrm{m}}_{\TT_V,p} / K_{\TT_V,p}| \geq cp
$.
Thus
$$ \alpha_V \geq (Bc)^{i_V} \cdot \Big(\prod_{\substack{p \, \textnormal{prime} \\ p \,
    \textnormal{unramified in} \, L_V \\
    K^{\mathrm{m}}_{{\TT_V},p}\not= K_{{\TT_V},p} }} p \Big) \geq (Bc)^{i_V} \cdot
{i_V}!\qquad ,$$
where we used in the last inequality that the $p$th prime in $\NN$ is at least $p$.

\end{proof}

\begin{defi}
Given a positive real number $t$ we denote by $\Sigma_t$
 the set of $V$ in $\Sigma$ with $i_V>t$.
\end{defi}

To finish the proof of proposition~\ref{choice_l} we proceed by dichotomy:

\begin{itemize}
\item Suppose that for any $t$ the set $\Sigma_{t}$ is a modification
of $\Sigma$. In particular the function $i_V$ is
unbounded as $V$ ranges through $\Sigma$.   For simplicity, we let $B':=Bc$.
Recall the well-known inequality: for every integer $n> 1$,
$$
e\cdot  \left(\frac{n}{e}\right)^n  < n! < e\cdot n \cdot \left(\frac{n}{e}\right)^{n} \;\;.
$$
The lower bound for $\alpha_V$ provided by lemma~\ref{lowalpha} gives:
$$
\alpha_V> e \left(\frac{B' i_V}{e}\right)^{i_V} >\left(\frac{B' i_V}{e}\right)^{i_V}\;\;.
$$
Hence:
$$
\alpha_{V}^{\frac{\epsilon}{2}} >  \left(\frac{B' i_V}{e}\right)^{\frac{\epsilon  
i_V}{2}}\;\;.
$$
For $i_V>\frac{4}{\epsilon}$ we obtain:
$$\alpha_{V}^{\frac{\epsilon}{2}} > \left(\frac{B' i_V}{e}\right) ^2\;\;.$$

Using the lower bound~(\ref{lowerpi}) for $\pi_{L_V}(x_V)$, the
trivial lower bound $\beta_V\geq 1$ and the fact that $m
\geq 6$,
we obtain that 
\begin{equation*}
\pi_{L_V}(x_V) \geq  \frac{D^{\frac{1}{2}}{B'}^2}{3n e^2}\cdot  i_V^2 \;\;.
\end{equation*}

Hence, whenever  $$i_V  > t = \max \left( \frac{3 ne^2
  }{D^{\frac{1}{2}}{B'}^2},\frac{4}{\epsilon},C \right)$$ we get $\pi_{L_V}(x_V) > \max(i_V,C)$. As the set
$\Sigma_t$ is a modification of $\Sigma$ we get the proposition~\ref{choice_l}.

\item Otherwise there exists a positive number $t$ such that $\Sigma
  \setminus \Sigma_t$ is a modification of $\Sigma$. Replacing
  $\Sigma$ by $\Sigma \setminus \Sigma_t$ we can assume without loss
  of generality that the function $i_V$ is
  bounded by $t$ as $V$ ranges through
  $\Sigma$.
  
We have for any $V$ in $\Sigma$,
  $$ \pi_{L_V}(x_V) > \frac{D^{\frac{1}{2}}}{3n} \cdot
  \alpha_V^{\frac{\epsilon}{2}} \cdot \beta_V^{\frac{m}{2}}\geq \frac{D^{\frac{1}{2}}}{3n}(\alpha_V\beta_V)^{\frac{\epsilon}{2}} \;\;.
  $$
Then $\pi_{L_V}(x_V) > \max(i_V, C)$ as soon as
$$
\frac{D^{\frac{1}{2}}}{3n}(\alpha_V\beta_V)^{\frac{\epsilon}{2}} > \max(t,C)\;\;.
$$ 
Hence we can take $M = \left( \frac{3n\max(t,C)}{D^{1/2}} \right) ^{2/\epsilon}$.  
\end{itemize}
\end{proof}
Let us now finish the proof of theorem~\ref{main-thm2} by showing the existence of the prime $l$ satisfying the 
conditions of proposition~\ref{suffisant}.
Let $r: = \dim Z - n (\Sigma)$.
Let $N$ be a positive integer, at least $6(k+2f)\cdot 2^{r}$ and such
that $m:= \frac{N}{(k+2f) \cdot 2^{r}}$ is an integer.
Let $\varepsilon < \frac{1}{(k+2f) \cdot 2^{r}}$, $D =
\{\frac{C(N)}{(\deg_{L_{K}} Z)^{2^{r}}}\}^\frac{1}{(k+2f) \cdot
  2^{r}}$.

Let $M$ be the integer obtained from proposition~\ref{choice_l}
applied to $\epsilon$, $m$ and $D$. Up to a modification of $\Sigma$
we can assume that any $V$ in $\Sigma$ satisfies $\alpha_V \beta_V
>M$. Thus by proposition~\ref{choice_l} for every $V \in \Sigma$ we can choose a
prime $l>C$ such that $l$ splits $\TT_V$, the reduction $\TT_{V,\FF_l}$ is a
 torus and $l < D \alpha_V^\epsilon \beta_V^{m}$: this is proposition~\ref{suffisant}.
\end{proof}

\vspace{0.5cm}
\small{

\noindent Bruno Klingler : Institut de Math\'ematiques de Jussieu,
Paris.

\noindent email : \texttt{klingler@math.jussieu.fr}

\sspace
\noindent
Andrei Yafaev : University College London, Department of Mathematics.

\noindent
email : \texttt{yafaev@math.ucl.ac.uk}}

\end{document}